\documentclass[oneside, 12pt]{amsart}
\usepackage[left=2.5cm,top=2.5cm,bottom=2.5cm,right=2.5cm]{geometry}
\usepackage{amsmath, amsthm, amssymb}
\usepackage{enumerate}
\usepackage{amsfonts, amscd, bbm, eucal, latexsym, extarrows, mathrsfs}
\usepackage[utf8]{inputenc}
\usepackage[T1]{fontenc}
\usepackage[bookmarks=false]{hyperref}
\usepackage[usenames,dvipsnames]{color}
\usepackage{mathtools}
\usepackage{fancyhdr}
\usepackage[all,cmtip]{xy}
\usepackage{color}
\usepackage{microtype}

\usepackage{fourier}

\providecommand{\vol}{\operatorname{vol}}

\providecommand{\rk}{\operatorname{rk}}
\providecommand{\Gr}{\operatorname{Gr}}

\providecommand{\Spec}{\operatorname{Spec}}
\providecommand{\cone}{\operatorname{cone}}

\usepackage{mathtools}

\def\and{\textrm{ and }}

\theoremstyle{plain}
\newtheorem{theorem}{Theorem}[section]
\newtheorem*{definition-intro}{Definition}
\newtheorem*{cor-intro}{Corollary}

\newtheorem{proposition}[theorem]{Proposition}		
\newtheorem{corollary}[theorem]{Corollary}
\newtheorem{lemma}[theorem]{Lemma}
\newtheorem{conjecture}[theorem]{Conjecture}

\newtheorem{definition}[theorem]{Definition}

\theoremstyle{remark}
\newtheorem{remark}[theorem]{Remark}

\newcommand{\taubcov}[1]{\tau_{BCOV}{( #1})}

\newcommand{\Hdprim}[2]{H^{{ #1}, { #2}}_{\operatorname{prim}}}
\newcommand{\Hprim}[1]{H^{{ #1}}_{\operatorname{{}prim}}}

\newcommand\bb[1]{\mathbb{#1}}

\newcounter{tmp}

\newcommand{\ov}{\overline}

\newcommand{\ebold}{{\bf e}}

\newcommand{\CBbb}{\mathbb C}
\newcommand{\DBbb}{\mathbb D}

\newcommand{\HBbb}{\mathbb H}

\newcommand{\PBbb}{\mathbb P}
\newcommand{\QBbb}{\mathbb Q}
\newcommand{\RBbb}{\mathbb R}

\newcommand{\ZBbb}{\mathbb Z}

\newcommand{\Ccal}{\mathcal C}

\newcommand{\Fcal}{\mathcal F}

\newcommand{\Ocal}{\mathcal O}

\newcommand{\Tcal}{\mathcal T}
\newcommand{\Ucal}{\mathcal U}
\newcommand{\Vcal}{\mathcal V}

\DeclareMathOperator{\End}{End}
\DeclareMathOperator{\Hom}{Hom}

\DeclareMathOperator{\id}{id}

\DeclareMathOperator{\tr}{tr}

\DeclareMathOperator{\Real}{Re}

\DeclareMathOperator{\Res}{Res}
\DeclareMathOperator{\nt}{nt}

\newcommand{\sing}{{\rm sing}}

\newcommand{\prim}{\rm prim}

\numberwithin{equation}{section}
\begin{document}

\title[BCOV invariants]{BCOV invariants of Calabi--Yau manifolds \\ and degenerations of Hodge structures}
\author[Eriksson]{Dennis Eriksson}
\author[Freixas i Montplet]{Gerard Freixas i Montplet}
\author[Mourougane]{Christophe Mourougane}
\address{Dennis Eriksson \\ Department of Mathematics \\ Chalmers University of Technology and  University of Gothenburg}
\email{dener@chalmers.se}

\address{Gerard Freixas i Montplet \\ C.N.R.S. -- Institut de Math\'ematiques de Jussieu - Paris Rive Gauche}
\email{gerard.freixas@imj-prg.fr}

\address{Christophe Mourougane\\Institut de Recherche Math\'ematique de Rennes (IRMAR)}
\email{christophe.mourougane@univ-rennes1.fr}

\begin{abstract}
Calabi--Yau manifolds have risen to prominence in algebraic geometry, in part because of  mirror symmetry and enumerative geometry. After Bershadsky--Cecotti--Ooguri--Vafa (BCOV), it is expected that genus 1 curve counting on a Calabi--Yau manifold is related to a conjectured invariant, only depending on the complex structure of the mirror, and built from Ray--Singer holomorphic analytic torsions. To this end, extending work of Fang--Lu--Yoshikawa in dimension~3, we introduce and study the BCOV invariant of Calabi--Yau manifolds of arbitrary dimension. To determine it, knowledge of its behaviour at the boundary of moduli spaces is imperative. To address this problem, we prove general results on degenerations of $L^2$ metrics on Hodge bundles and their determinants, refining the work of Schmid. We express the singularities of these metrics in terms of limiting Hodge structures, and derive consequences for the dominant and subdominant singular terms of the BCOV invariant. 

\end{abstract}

\maketitle
\setcounter{tocdepth}{3}
\tableofcontents

\section{Introduction}
\begingroup
\setcounter{tmp}{\value{theorem}}
\setcounter{theorem}{0}
\renewcommand\thetheorem{\Alph{theorem}}

In this article we introduce and study a real valued invariant for Calabi--Yau manifolds, depending only on the complex structure, expected to be birationally invariant and to encode genus 1 curve counting on a mirror Calabi--Yau manifold. The invariant has as its origin the remarkable theoretical physics article of Bershadsky--Cecotti--Ooguri--Vafa \cite{bcov}. It is a genus~1 counterpart of the Yukawa coupling studied by Candelas--de la Ossa--Green--Parkes \cite{COGP} in their work on mirror symmetry and genus 0 Gromov--Witten invariants. The physics theory of \cite{bcov} received a mathematical treatment by Fang--Lu--Yoshikawa \cite{FLY}, where they defined what is nowadays called the BCOV invariant for three dimensional Calabi--Yau manifolds, in the strict sense. They confirmed the predictions in \cite{bcov} for the mirror of the quintic Calabi--Yau $3$-fold. An important and open question, already raised in \cite[Sec. 5.8]{bcov}, is to extend these constructions and results to general dimensions, and this is the purpose of our study.

\subsection{The BCOV invariant of a Calabi--Yau manifold}
Let $Z$ be a Calabi--Yau manifold of dimension $n$, meaning a compact K\"ahler manifold with trivial canonical bundle. Given a choice of Ricci flat K\"ahler metric $\omega$ on $Z$, the invariant $T(Z,\omega)$ proposed by \cite{bcov} is a combination of Ray--Singer holomorphic torsions, that can be written as follows:
\begin{equation}\label{TXOMEGA}
    T(Z,\omega)=  \prod_{p,q = 1}^{n} (\det \Delta_{\overline{\partial}}^{(p,q)})^{(-1)^{p+q} pq }
\end{equation}
where $\det  \Delta_{\overline{\partial}}^{(p,q)}$ refers to the zeta regularized determinant of the $\overline{\partial}$-Laplacian acting on $(p,q)$-forms on $Z$.  This depends on the K\"ahler form, as opposed to the mirror symmetry principle that it should be an invariant of the complex structure in the $B$-model.  An intrinsic definition in dimension 3, when $h^{0,1}=h^{0,2}=0$, was provided by \cite{FLY}, and was accomplished by multiplying by normalizing factors. A similar renormalization was noted in \cite{PesWitt}. The correction term seems to be specific to dimension 3 and does not offer an obvious extension to higher dimensions.  To our surprise, a general normalization was suggested by Kato's formalism of heights of motives \cite{Katoheight}. 

In our approach, the BCOV invariant  is in fact realized as the quotient of two natural metrics on the so-called BCOV line bundle. This bundle can be thought of as a weighted product of determinants of Hodge bundles, defined for a single compact complex manifold $Z$ as
\begin{equation*}\label{eq:bcovlinebundle}
    \lambda_{BCOV}(Z) = \bigotimes_{p,q} \det H^q(Z, \Omega_Z^p)^{(-1)^{p+q} p}.
\end{equation*}
On the one hand, $\lambda_{BCOV}(Z)$ can be equipped with the Quillen-BCOV metric $h_{Q,BCOV}$ introduced in \cite{FLY}, and independent of any choice of K\"ahler structure. On the other hand, in the current article we exhibit a renormalized $L^2$ metric on the BCOV line bundle, that we call the $L^2$-BCOV metric. Precisely, for a single Calabi--Yau manifold $Z$ with K\"ahler form $\omega$, we define the $L^{2}$-BCOV norm of an element $\sigma$ of the BCOV line by
\begin{displaymath}
    h_{L^{2},BCOV}(\sigma,\sigma)=h_{L^{2}}(\sigma,\sigma)\prod_{k=1}^{2n}\vol_{L^{2}}(H^{k}(Z,\ZBbb),\omega)^{(-1)^{k+1}k/2}.
\end{displaymath}
Here $h_{L^{2}}$ is the product of $L^{2}$ metrics on the BCOV line provided by Hodge theory (harmonic representatives with respect to $\omega$) and $\vol_{L^{2}}(H^{k}(Z,\ZBbb),\omega)$ is the covolume of the lattice $H^{k}(Z,\ZBbb)_{\nt}\subset H^{k}(Z,\RBbb)$ with respect to the $L^{2}$ scalar product, where $H^{k}(Z,\ZBbb)_{\nt}$ is the maximal torsion free quotient of $H^{k}(Z,\ZBbb)$. We adopt the convention that the volume equals 1 for those factors with $H^{k}(Z,\RBbb)=0$. In Proposition \ref{prop:L2-independent} we show that the $L^{2}$-BCOV metric thus defined is also independent of the choice of K\"ahler metric. The BCOV invariant of the Calabi--Yau manifold $Z$ is then defined as the proportionality factor comparing the Quillen-BCOV and $L^{2}$-BCOV metrics.

\begin{definition-intro}
Let $Z$ be a  Calabi--Yau manifold. Then we let
\begin{displaymath}
    \tau_{BCOV}(Z)=h_{Q,BCOV}/h_{L^{2},BCOV} \in \RBbb_{>0}.
\end{displaymath}
We refer to $\taubcov{Z}$ as the BCOV invariant of $Z$.
\end{definition-intro}
Note that since the two metrics are independent of the choice of K\"ahler structure, the invariant only depends on the complex structure of $Z$. It generalizes the construction for strict Calabi--Yau $3$-folds in \cite{FLY} to  general Calabi--Yau manifolds of arbitrary dimension. The writing of the invariant as a quotient of metrics lends itself to computing the second variation of $\tau_{BCOV}(Z)$ as the complex structure is deformed.  If $f\colon X\to S$ is a K\"ahler morphism of connected complex manifolds whose fibers are Calabi--Yau $n$-folds, we prove in Proposition \ref{prop:ddc-log-B} that the function $s\mapsto\log\tau_{BCOV}(X_{s})$ satisfies the differential equation
    \begin{equation}\label{eq:ddc-log-BB}
        dd^{c}\log\tau_{BCOV}=\sum_{k=0}^{2n}(-1)^{k}\omega_{H^{k}}-\frac{\chi}{12}\omega_{WP}.
    \end{equation}
    Here, $\chi$ is the topological Euler characteristic of a general fiber, and $\omega_{WP}$ and $\omega_{H^{k}}$ are the Weil--Petersson and Hodge type forms for the family $f\colon X\to S$. The equation \eqref{eq:ddc-log-BB} is a higher dimensional version of the holomorphic anomaly equation at genus 1 in the mirror symmetry literature.
    
    In general, the BCOV invariant is a non-trivial function. In some cases, such as abelian varieties of dimension at least two or hyperk\"ahler varieties, it is constant in complex moduli (cf. Proposition \ref{prop:ken-ichi}). In the case of moduli of polarized Calabi--Yau manifolds, it is equal to the BCOV torsion in  \eqref{TXOMEGA}, up to a constant which depends on the polarization. 


\subsection{Asymptotic behaviour of the BCOV invariant}
The only known strategy to approach the BCOV predictions in mirror symmetry consists in seeing the BCOV invariant as a function on a moduli space of Calabi--Yau varieties, then exploiting the holomorphic anomaly equation \eqref{eq:ddc-log-BB}. We refer the reader to e.g. \cite{Klemmboundary} for a discussion in the mathematical physics literature.

Since moduli spaces are in general non-compact, the differential equation determines $\log\tau_{BCOV}$ at most up to a generally non-constant pluriharmonic function. To fix this indeterminacy, known as the holomorphic ambiguity, it is essential to know the limiting behaviour of the BCOV invariant as one approaches the boundary of the moduli space. We provide a general answer to this question, by identifying an explicit topological expression for such boundary conditions for one-parameter normal crossings degenerations (see Theorems \ref{thmB}, \ref{thmC} below). 

Suppose $f: X \to \DBbb$ is a projective morphism of complex manifolds, with Calabi--Yau $n$-fold fibers outside the origin.\footnote{Actually a slightly stronger technical condition is required, that $f$ extends to a projective morphism $\widetilde{X}\to S$ of compact complex manifolds, with $\dim S=1$. See \textsection \ref{subsec:notations}.}
In Theorem \ref{thm:general-bcov} and Proposition \ref{prop:loglogbcov} we prove that the BCOV invariant behaves as
\begin{equation}\label{eq:BCOVasymptoticintro}
  \log\tau_{BCOV}(X_{t})=\kappa_{f}\log|t|^{2} + \varrho_f \log  \log|t|^{-1} + \text{continuous},
\end{equation}
for constants $\kappa_f, \varrho_f  \in\QBbb$, as $t \to 0$. The rationality of $\kappa_f$ was established in the three-dimensional case by Yoshikawa \cite{yoshikawa5}. Under further assumptions on the degeneration, we obtain general expressions for $\kappa_{f}$ and $\varrho_{f}$, in terms of the geometry of the special fiber and the limiting mixed Hodge structures. In this introduction, we focus on some relevant situations, when $\kappa_f$ and $\varrho_f$ take a particularly pleasant form.


 To state the first such result, let $f \colon X \to \DBbb$ be a projective morphism of complex manifolds with Calabi--Yau $n$-fold fibers outside the origin, and such that the special fiber $X_0=\sum_{i}D_{i}$ is a simple normal crossings divisor. Introduce the notation $D(k) = \bigsqcup_{I} \bigcap_{i \in I} D_{i}$ where the union is over index subsets $I$ of cardinality $k$.  Denote by $\chi(D(k))$ the topological Euler characteristic of $D(k)$.

\begin{theorem}\label{thmB} In the above situation, suppose also that $f$ is Kulikov, \emph{i.e.} $K_{X} \simeq \Ocal_{X}$. Then
\begin{displaymath}
        \kappa_{f}=\sum_{k=1}^{n+1}(-1)^{k}\frac{k(k-1)}{24}\chi(D(k)).
\end{displaymath}
In the case of strict Calabi--Yau $3$-folds (\emph{i.e.} $h^{0,1} = h^{0,2} = 0$),  this expression can be further simplified to
\begin{displaymath}
    \kappa_f = \frac{\chi(D(2)) - 4 Q}{12}
\end{displaymath}
where $Q = \# D(4)$ is the number of quadruple points.
\end{theorem}
The case of dimension three in the theorem confirms a conjecture of Liu--Xia \cite[Conj. 0.5]{LiuXia} and is found in Corollary \ref{prop:LiuXia:special}. 
The theorem, which follows from Proposition \ref{cor:liuxiahigherdim}, can thus be seen as a far-reaching generalization thereof. 

Another important example of generic singularities for several Calabi--Yau families is the case of ordinary double points,  often called conifold singularities. For these singularities we prove in Theorem~\ref{thm:bcov-odp-dim-n}:
\begin{theorem}\label{thmC}
Let $f\colon X\to\DBbb$ be a projective morphism of complex manifolds, with smooth Calabi--Yau $n$-fold fibers outside the origin. If the special fiber $X_{0}$ has at most ordinary double point singularities, then
    \begin{itemize}
        \item if $n$ is odd,
            \begin{displaymath}
                 \kappa_{f}=\frac{n+1}{24}\#\sing(X_{0})\quad\text{and}\quad\varrho_f = \#\sing(X_{0}).
            \end{displaymath}
        \item if $n$ is even,
            \begin{displaymath}
                  \kappa_{f}=-\frac{n-2}{24}\#\sing(X_{0})\quad\text{and}\quad\varrho_f =0. 
            \end{displaymath}
    \end{itemize}
    Here $\#\sing(X_{0})$ denotes the number of singular points in the fiber $X_{0}$.
\end{theorem}
 For $\kappa_f$, the case $n=3$ of the statement was already established by Fang--Lu--Yoshikawa \cite{FLY}, and it was a key point to their proof of the BCOV conjecture for the mirror quintic family. For $n=4$ the theorem corroborates an expectation suggested by work of Klemm--Pandharipande in \cite[Sec. 4 \& 6]{Klemm-Pandharipande}. The case of general $n$ confirms a conjecture of Fang--Lu--Yoshikawa from 2004 (private communication). The refined information for $\varrho_f$ had not been considered before, and is one of the applications of our approach.  

In dimensions three and four we have more general formulas for $\kappa_f$ only assuming a smooth total space (cf. Theorem \ref{thm:LiuXia-general} and Theorem \ref{thm:LiuXia-general4}). These rely on a careful study of Serre duality on K\"ahler extensions due to T. Saito \cite{SaitoTakeshiBored}. In these cases $\kappa_f$ involves in particular the total dimension of the vanishing cycles of the family.

\subsection{Asymptotic behaviour of $L^2$ metrics and monodromy eigenvalues} \label{section:asymptoticalL2}
In order to generally study the asymptotics of the BCOV invariant (cf. \eqref{eq:BCOVasymptoticintro}) one needs a precise control on the asymptotics of Quillen and $L^{2}$ metrics. The singularities of the Quillen-BCOV metric were already dealt with in our previous paper \cite{cdg}, itself relying on results of Yoshikawa \cite{yoshikawa}. As for $L^{2}$ metrics on Hodge bundles, we elaborate on the work of Schmid \cite{schmid} and Peters \cite{Peters-flatness}, and obtain explicit asymptotics in terms of limiting Hodge structures.

The framework of our analysis is a projective morphism $f\colon X\to\DBbb$ of complex manifolds, with normal crossings special fiber. Denote by $\Omega_{X/\DBbb}^{p}(\log)$ the sheaf of relative differential $p$-forms with logarithmic poles along the special fiber. Then the higher direct image sheaves $R^q f_\ast \Omega_{X/\DBbb}^{p}(\log)$ are locally free. Given a choice of rational K\"ahler structure, the vector bundles $R^q f_\ast \Omega_{X/\DBbb}^{p}(\log)$ carry corresponding $L^{2}$ metrics, possibly singular at the origin. In Theorem \ref{thm:expansion-hodge} we determine the singularities of the induced metrics on $\det R^q f_\ast \Omega_{X/\DBbb}^{p}(\log)$. The statement involves the limiting Hodge structure on $H^{p+q}(X_\infty)$, the cohomology of a general fiber. Recall that this is a mixed Hodge structure with a Hodge type filtration $F^\bullet_\infty$ and a weight filtration $W_\bullet$. The semi-simple part of the monodromy, denoted $T_s$, acts upon the whole mixed Hodge structure and admits a logarithm $^{\ell}\log T_s$ with eigenvalues in $2\pi i\cdot \QBbb \cap (-1, 0] $. 

\begin{theorem}\label{ThmMetric}
Let $\sigma$ be a holomorphic trivialization of $\det R^{q}f_{\ast}\Omega_{X/\DBbb}^{p}(\log)$. Then we have a real analytic asymptotic expansion for its $L^{2}$ norm
\begin{displaymath}
        \log h_{L^2}(\sigma,\sigma) = \alpha^{p,q} \log|t|^2 + \beta^{p,q} \log \log|t|^{-1}+ C + O\left(\frac{1}{\log|t|}\right),
\end{displaymath}
with
\begin{displaymath}
    \alpha^{p,q}=-\frac{1}{2\pi i}\tr\left(^{\ell}\log T_{s}\mid \Gr^{p}_{F_{\infty}} H^{k}(X_{\infty})\right)\quad (k=p+q)
\end{displaymath}
and
\begin{displaymath}
    \beta^{p,q}=\sum_{r=-k}^{k}r\dim\Gr_{F_\infty}^{p}\Gr^{W}_{k+r}H^{k}(X_{\infty})
\end{displaymath}
and a constant $C \in \RBbb.$
\end{theorem}
 This result refines \cite[Prop. 2.2.1]{Peters-flatness}. Notice that he could only show that $\alpha^{p,q}$ is a rational number extracted from the monodromy in a non-precise manner, and similarly for $\beta^{p,q}$. Our contribution thus clarifies the exact relationship to the limiting mixed Hodge structure.

In the semi-stable case (hence unipotent monodromies) the coefficient $\alpha^{p,q}$ is known to be zero. The determination of the coefficient $\alpha^{p,q}$ thus reduces to the following purely algebraic geometric consideration. To compare to the semi-stable case we exhibit a diagram
\begin{displaymath}
    \xymatrix{
        Y\ar[d]_{g}\ar[r]  &X\ar[d]^{f}\\
        \DBbb\ar[r]^{\rho}  &\DBbb,
    }
\end{displaymath}
where $\rho(t)=t^{\ell}$ is some ramified covering and $Y \to \DBbb$ is semi-stable. The diagram is Cartesian over $\DBbb^{\times}$. There is a natural inclusion of vector bundles
\begin{displaymath}
    \rho^{\ast}  R^q f_\ast \Omega_{X/\DBbb}^{p}(\log)\subseteq R^q g_\ast \Omega_{Y/\DBbb}^{p}(\log).
\end{displaymath}
The quotient is a torsion sheaf supported on the origin, and can hence be written as

\begin{displaymath}
    \bigoplus_{j=1}^{h^{p,q}}\frac{\Ocal_{\DBbb,0}}{t^{a_{j}}\Ocal_{\DBbb,0}},
\end{displaymath}
for some integers $a_{j}\geq 0$. The rational numbers $\alpha_{j}^{p,q} = \frac{a_j}{\ell}\in [0,1)$ are independent of the choice of semi-stable reduction, and we call them the elementary exponents of the $(p,q)$ Hodge bundle. Their sum is equal to the sought for $\alpha^{p,q}$ in Theorem \ref{ThmMetric}. We prove the following fundamental theorem in the theory of degenerations of Hodge structures, of independent interest:

\begin{theorem} \label{thmD}
The elementary exponents $\alpha^{p,q}_{j}$ of the $(p,q)$ Hodge bundle are such that  $\exp(-2\pi i \alpha^{p,q}_{j})$ constitute the eigenvalues of $T_{s}$, the semi-simple part of the monodromy acting on $\Gr_{F_{\infty}}^{p}H^{p+q}(X_{\infty})$ (with multiplicities).

\end{theorem}

 In fact our arguments provide more general results for Deligne extensions for degenerations of Hodge structures. These results constitute the content of Section \ref{section:monodromy} (see in particular Theorem \ref{thm:eigenvalues} and Corollary \ref{cor:eigenvalues}). Theorem \ref{ThmMetric} and Theorem \ref{thmD} generalize Theorem A of \cite{cdg} and analogous results by Halle--Nicaise \cite{NicaiseHalle} and Boucksom--Jonsson \cite{BJvolumes} which consider constructions coming from canonical bundles. With respect to these references, we stress that the theorems of this section do not assume any Calabi--Yau hypothesis, and apply to general graded pieces of the Hodge filtration. To the best of our knowledge, and despite of their classical appearance and relevance, the results of this section are new.

\subsection{Relations with mirror symmetry}
\setcounter{theorem}{0}

In this section we return to the initial motivation for the introduction of the BCOV invariant, discussed at the beginning of the text. More computational endeavours and applications of the invariant will be the topic of work in preparation. 

Consider $f \colon X \to \DBbb$ to be a maximally unipotent projective degeneration of Calabi--Yau $n$-folds\footnote{This is also known as a large complex structure limit of Calabi--Yau manifolds.}, meaning that for the monodromy operator $T$ on the local system corresponding to $H^n$, we have  $(T-1)^n \neq 0$ and $(T-1)^{n+1} = 0$. Denote by $X_{\infty}$ a general fiber. In this setting, mirror symmetry predicts the existence of a Calabi--Yau mirror $X_{\infty}^\vee$ and an ample class $[H]$ such that
\begin{equation}\label{eq:kappafmirror}
    \kappa_f =  \frac{(-1)^{n+1}}{12} \int_{X_{\infty}^\vee} c_{n-1}(X_{\infty}^\vee) \cap [H].
\end{equation}
See for instance the introduction in Yoshikawa's \cite{Yoshikawa-orbifold} and the discussion by Klemm--Pandhari\-pande \cite[Sec. 4]{Klemm-Pandharipande}.

Generally, even when potential mirrors are known, such as in Batyrev's framework using toric Fano varieties \cite{Batyrev-toric}, this seems to be out of reach. However, in the special case when $X_\infty$ is an abelian variety or a hyperk\"ahler variety, we can confirm that both sides of \eqref{eq:kappafmirror} are zero. The right hand side vanishes, because $X_\infty^\vee$ is also an abelian or hyperk\"ahler variety and so  $c_{n-1}(X_\infty^\vee)=0$. The vanishing of the left hand side is due to the constancy of the BCOV invariant for such families, cf. Proposition \ref{prop:ken-ichi}. Besides, in the direction of the conjecture we can prove the following corollary, which is a consequence of the general form of Theorem \ref{thmB} (Theorem \ref{thm:general-bcov} \emph{infra}) :
\begin{cor-intro}
Suppose that $f \colon X \to \DBbb$ is a projective degeneration of Calabi--Yau varieties, with unipotent monodromies. Then
\begin{displaymath}
    12 \kappa_f \in \ZBbb.
\end{displaymath}
\end{cor-intro}

\setcounter{theorem}{0}





Consider momentarily a Calabi--Yau 3-fold $Z$ and $[H]$ an ample cohomology class in $H^2 (Z)$. Then by the known stability of $T_Z$ and the Bogomolov--Gieseker inequality , we have

\begin{displaymath}
    \int_Z c_2(Z) \cap [H] \geq 0.
\end{displaymath}

Taking into account \eqref{eq:kappafmirror}, this leads us to make the following conjecture:

\begin{conjecture}
If $f \colon X \to \DBbb$  is a projective degeneration of 3-dimensional Calabi--Yau varieties with maximally unipotent monodromy, then $\kappa_f \geq 0$.
\end{conjecture}

We remark that the conjecture is true for abelian $3$-folds, since in this case our BCOV invariant is 1 so that $\kappa_f$ is 0. Most importantly, it is also known to hold for the mirror quintic family, as in this case the BCOV predictions were confirmed by Fang--Lu--Yoshikawa \cite{FLY}. In the general case, we expect that the explicit formulas we exhibit for $\kappa_f$ (cf. Theorem \ref{thmB} for the semi-stable relative minimal setting) will provide a useful tool towards the proof of the conjecture.

More ambitiously, we could ask about the sign of the coefficient $\kappa_{f}$ for an arbitrary degeneration, non-necessarily with maximal unipotent monodromy. This question, for which we don't have a conjectural answer, is closely related to the problem of determining the birational type of the moduli space of polarized Calabi--Yau 3-folds. We thank K.-I. Yoshikawa for bringing our attention to this fact.

In another direction, it is expected (cf. \cite[p. 137]{KontMirror}) that birational Calabi--Yau manifolds have the same $B$-models. Since the BCOV invariant should only depend thereupon, this expectation should afford the following realization:

\begin{conjecture}\label{conj:birationalinv}
If $X$ and $X'$ are birational Calabi--Yau manifolds, then  $\taubcov{X} = \taubcov{X'}$.
\end{conjecture}
The conjecture extends the corresponding conjecture in \cite{FLY} in dimension 3, which in this generality remains open. In these lines, Maillot--R\"ossler \cite{MaRo} proved that if $X$ and $X'$ are defined over $\QBbb$, then a power of the quotient $\taubcov{X}/\taubcov{X'}$ is a rational number.

After a preliminary version of this work was circulated, Y. Zhang was lead to extend our construction and produce a BCOV type invariant for Calabi--Yau pairs $(X,Y)$ (cf. \cite{Yeping}). These pairs consist of a compact K\"ahler manifold $X$ together with a smooth divisor $Y$ in some linear series $|m K_{X}|$. It can be expected that this construction be a useful auxiliary tool in the proof of Conjecture \ref{conj:birationalinv}.

As a final consideration, we remark that the current known constructions of mirror Calabi--Yau varieties, e.g. Batyrev's \cite{Batyrev-toric}, do not always produce Calabi--Yau manifolds, but rather orbifolds. It would thus be desirable to extend these constructions to this context, possibly replacing Dolbeault cohomology with Chen--Ruan cohomology. In this direction we quote the extensive work of Yoshikawa \cite{Yoshikawak31, YoshiborcherdsAst, Yoshikawak32, Yoshikawak33, Yoshikawatrinity, Yoshikawa-orbifold}, who considered  equivariant analytic torsion of K3 surfaces with involution. A running theme is the  relationship between equivariant torsions and Borcherds products. This is remarkable since the non-equivariant BCOV torsion is trivial for K3 surfaces. 

\endgroup

K\"ahler
\setcounter{theorem}{\thetmp}
\subsection{Notations and conventions}\label{subsec:notations}

For the convenience of the reader, we introduce the various notations and conventions that are being used in this work.

A \emph{Calabi--Yau} variety is, for the purposes of this article, a compact connected complex K\"ahler manifold $X$ with trivial canonical bundle $K_X \simeq \Ocal_X$.
We say that a Calabi--Yau variety $X$ of dimension $n$ is a \emph{strict Calabi--Yau} variety if moreover $H^i(X, \Ocal_X)=0$ for $0 < i < n$. Hence, except for $K3$ surfaces, neither hyperk\"ahler varieties are strict Calabi--Yau  varieties, nor are abelian varieties of dimension at least 2. Smooth hypersurfaces of degree $n+1$ in $\PBbb^n$ are strict Calabi--Yau varieties.
By the Bogomolov--Beauville decomposition theorem \cite{Bogomolov, Beauville}, in the algebraic category these can be realized as finite \'etale quotients of varieties of the form $T \times V \times H$, where $T$ is a complex torus, $V$ is a strict Calabi--Yau variety and $H$ a hyperk\"ahler variety.

A \emph{degeneration} $f \colon X \to S$ is a flat morphism of reduced and irreducible complex analytic spaces, with connected fibers and smooth general fibers. Often we will take $S = \DBbb$, a disc centered at $0$, and then we suppose that the morphism is smooth outside of the origin. In that case we denote by $X_\infty$ a general fiber, and by $X_0$ the fiber above the origin. The differentiable type of $X_\infty$ is independent of the choice of a general fiber.
A degeneration $f\colon X\to \DBbb$ is said to have \emph{normal crossings} if $X$ is non-singular and $X_{0}$ is a non-necessarily reduced simple normal crossings divisor in $X$. If $X_{0}$ is furthermore reduced, then we say that $f$ is semi-stable. We may equivalently talk about \emph{normal crossings (resp. semi-stable) degenerations}.

A \emph{projective} morphism  $f\colon X \to S$ is a proper morphism of analytic spaces such that, locally with respect to the base, $f$ factors through a closed immersion $\PBbb^n \times S$ followed by the projection on the second variable. A \emph{projective degeneration} is a morphism which is both a degeneration and projective.
A \emph{rational} K\"ahler structure on a complex manifold is a K\"ahler form such that the associated cohomology class is rational.
A \emph{K\"ahler morphism} $f \colon X \to S$  is a proper submersion of complex manifolds together with a K\"ahler metric $\omega$ on $X$.\footnote{It would be enough to assume the existence of a smooth closed real $(1,1)$ form on $X$, inducing a K\"ahler metric on fibers. } Usually we will confound a K\"ahler metric and its associated K\"ahler form. A \emph{K\"ahler degeneration} $f\colon X\to S$ is a proper morphism of complex manifolds, which is a degeneration and a K\"ahler morphism on the smooth locus. We still require that, locally with respect to the base $S$, $X$ affords a K\"ahler form. For instance, a projective degeneration  $f\colon X\to\DBbb$ between complex manifolds admits a structure of K\"ahler degeneration, by considering a factorization through a projective space bundle.

A \emph{germ} of any of the above types $T$ (\emph{e.g.} K\"ahler) of morphisms is the localization of a morphism of compact complex manifolds, of type $T$. For example, here are two typical situations: i) a germ of a K\"ahler morphism $X\to\DBbb$ is the restriction over a small disc $\DBbb$ of a K\"ahler morphism $Y\to S$ between compact complex manifolds, with $S$ a compact Riemann surface; ii) a germ of a degeneration $X\to\DBbb$ of algebraic varieties is a localization over a disc $\DBbb$ of a degeneration $Y\to S$ between smooth proper algebraic varieties, with one-dimensional $S$.

\section{Logarithmic extensions of Hodge bundles}\label{section:monodromy}
In this section we recall Deligne's extension of a local system over the punctured unit disc $\DBbb^\times$. We also review Steenbrink's construction of the lower extension of the local system of de Rham cohomology of a normal crossings degeneration, together with its Hodge filtration. For these questions we follow closely the presentation expounded in \cite[Chap. 11]{Peters-Steenbrink},  \cite[Sec. 2]{Steenbrink-limits} and \cite[Sec. 2]{Steenbrink-mixedonvanishing}. Finally, we study the behaviour of the relative Hodge filtration with respect to semi-stable reduction. The comparison of Hodge filtrations before and after semi-stable reduction produces some elementary divisors, and we show that they exactly correspond to the eigenvalues of the semi-simple part of the monodromy operator acting on the limiting Hodge filtration.

\subsection{Deligne's extensions}\label{subsec:Deligne-extension}
Let $(\Vcal,\nabla)$ be a finite rank flat holomorphic vector bundle on $\DBbb^{\times}$. Let $q\colon\HBbb\to\DBbb^{\times}$ be the universal covering map $q(\tau)=\exp(2\pi i\tau)$, where $\HBbb$ is the upper half-plane. The $\CBbb$-vector space of multi-valued flat sections of $\Vcal$ is, by definition
\begin{displaymath}
	V_{\infty}:=\Gamma(\HBbb,q^{\ast}\Vcal)^{q^{\ast}\nabla}=\ker\left(q^{\ast}\nabla \colon \Gamma(\HBbb,q^{\ast}\Vcal)\rightarrow\Gamma(\HBbb,q^{\ast}\Vcal\otimes\Omega_{\HBbb})\right).
\end{displaymath}
This vector space is finite dimensional. The monodromy transformation is the endomorphism $T\in\End_{\CBbb}(V_{\infty})$ induced by $\tau\mapsto\tau+1$. We assume that $T$ is quasi-unipotent. We may then introduce the Chevalley decomposition $T=T_{s}T_{u}=T_{u}T_{s}$ of $T$ into a semi-simple endomorphism $T_{s}$ and a unipotent endomorphism $T_{u}$. The logarithm of $T_{u}$ is well-defined. We denote $N=\frac{1}{2\pi i}\log T_{u}$. To define a logarithm of $T_{s}$, one chooses a branch of the logarithm on $\CBbb^{\times}$. For a fixed choice of branch (still denoted $\log$) we denote  $S=\frac{1}{2\pi i}\log T_{s}$. We can thus define $\frac{1}{2\pi i}\log T:=S+N$.

The vector bundle $\Vcal$ on $\DBbb^{\times}$ can be uniquely extended to a vector bundle $\Vcal_{\log}$ on $\DBbb$, referred to as the \emph{Deligne extension}, in such a way that:
\begin{enumerate}[(i)]
	\item\label{Vlog:0} there is an identification $\Vcal_{\log\ \mid\DBbb^{\times}}\overset{\sim}{\longrightarrow}\Vcal$, depending only on the choice of branch of logarithm.
	\item\label{Vlog:1} the connection $\nabla$ extends to a connection with regular singular poles
	\begin{displaymath}
		\nabla\colon\Vcal_{\log}\longrightarrow\Vcal_{\log}\otimes\Omega_{\DBbb}(\log [0]).
	\end{displaymath}
	Here we denote by $\Omega_{\DBbb}(\log [0])$ the sheaf of meromorphic differentials on $\DBbb$ with at most a simple pole at $0$.
	\item\label{Vlog:2} the eigenvalues of $-2\pi i \Res_{0}\nabla$ belong to $(2\pi i\QBbb)\cap\log\CBbb^{\times}$.
	\item\label{Vlog:3} the operator $T$ induces a vector bundle endomorphism of $\Vcal_{\log}$, whose fiber $T_{0}$ is related to $\Res_{0}\nabla$ by
\begin{displaymath}
	T_{0}=\exp(-2\pi i\Res_{0}\nabla).
\end{displaymath}
\end{enumerate}
Two frequent choices are the \emph{upper} and \emph{lower} branches. The upper branch takes values in $\RBbb+2\pi i[0,1)$, and will be denoted $^{\textit{u}}\log$. The lower branch takes values in $\RBbb+2\pi i(-1,0]$, and will be denoted $^{\ell}\log$. In later geometric constructions we will mostly encounter the lower branch.

Whenever the monodromy is unipotent, the extension $\Vcal_{\log}$ is called the \emph{canonical extension} of $\Vcal$.\footnote{In the literature one often refers to the lower extension as the canonical extension, \emph{e.g.} \cite[Def. 11.4]{Peters-Steenbrink}. In the present article, this terminology is reserved to the unipotent case.} Notice that, when it exists, the canonical extension ``commutes'' with any finite base change $q\mapsto q^{\ell}$.

An explicit description of $\Vcal_{\log}$ will be necessary. The regular sections of $\Vcal_{\log}$ are obtained from twisted flat multivalued sections as follows. Let $\ebold\in V_{\infty}$. Since $\ebold(\tau+1)=T\ebold(\tau)$, the twisted section
\begin{equation}\label{eq:twisted}
	\tilde{\ebold}(\tau):=\exp(-\tau \log T)\ebold(\tau)
\end{equation}
is invariant under $\tau\mapsto\tau+1$ and descends to a global holomorphic section of $\Vcal$ on $\DBbb^{\times}$, denoted $\tilde{\ebold}(q)$. From the flatness of $\ebold(\tau)$ and the relation $2\pi i d\tau=dq/q$, it is straightforward to check the equality on $\DBbb^{\times}$
\begin{displaymath}
    \nabla\tilde\ebold(q)=-(\log T)\tilde\ebold(q)d\tau=-\frac{1}{2\pi i}(\log T) \tilde\ebold(q)\frac{dq}{q}.
\end{displaymath}
The procedure $\ebold(\tau)\mapsto\tilde\ebold(q)$ describes a $\CBbb$-linear injective morphism
\begin{displaymath}
	\varphi\colon V_{\infty}\longrightarrow \Gamma(\DBbb^{\times},\Vcal)
\end{displaymath}
and induces an isomorphism, depending only on the choice of branch of logarithm
\begin{displaymath}
    V_{\infty}\otimes\Ocal_{\DBbb^{\times}}\overset{\sim}{\longrightarrow}\Vcal.
\end{displaymath}
Then, one defines
\begin{displaymath}
	\Vcal_{\log}:=\varphi(V_{\infty})\otimes_{\CBbb}\Ocal_{\DBbb}.
 \end{displaymath}
In other words, we declare that $\tilde{\ebold}$ as above is holomorphic at $q=0$. It is straightforward to check that $\Vcal_{\log}$ satisfies the properties \eqref{Vlog:0}--\eqref{Vlog:3} stated above.

A formal consequence of the construction is the isomorphism described as ``first twisting and then evaluating at 0":
\begin{equation}\label{eq:psi}
	\begin{split}
		\psi\colon V_{\infty}&\overset{\sim}{\longrightarrow}\Vcal_{\log}(0)\\
		\ebold &\longmapsto\tilde{\ebold}(0).
	\end{split}
\end{equation}
By definition, this isomorphism is equivariant with respect to the monodromy transformations, namely $T$ acting on $V_{\infty}$ and $T_{0}$ acting on $\Vcal_{\log}(0)$.

A final reminder concerns unipotent reduction. By assumption, the semi-simple endomorphism $T_{s}$ has finite order. We choose $\ell\geq 1$ with $T_{s}^{\ell}=\id$. Let $\rho\colon\DBbb\to\DBbb$ be the finite ramified cover $\rho(t)=t^{\ell}$. The space of multi-valued flat sections of the pull-back $(\rho^{\ast}\Vcal,\rho^{\ast}\nabla)$, say $U_{\infty}$, is actually isomorphic to $V_{\infty}$:
\begin{equation}\label{eq:nu}
	\begin{split}
		\nu\colon V_{\infty}&\overset{\sim}{\longrightarrow} \Gamma(\HBbb, q^{\ast}\rho^{\ast}\Vcal)^{q^{\ast}\rho^{\ast}\nabla}=:U_{\infty}\\
		\ebold(\tau)&\longmapsto \ebold(\ell\tau).
	\end{split}
\end{equation}
The monodromy transformation on $U_{\infty}$ is unipotent and identifies to $T^{\ell}=T_{u}^{\ell}$, with logarithm $\ell N$. The flat vector bundle $(\rho^{\ast}\Vcal,\rho^{\ast}\nabla)$ thus affords a canonical extension, that we denote by $\Ucal=(\rho^{\ast}\Vcal)_{\log}$. The procedure of twisting and evaluating at 0 for $\Vcal_{\log}$ and $\Ucal$ provides a commutative diagram of isomorphisms
\begin{equation}\label{eq:psi-rho}
	\xymatrix{
		\ebold(\tau)\ar@{|->}[d]	&V_{\infty}\ar[r]^{\sim}_{\psi}\ar[d]_{\nu}^{\sim}		&\Vcal_{\log}(0)\ar@{-->}[d]^{\sim}	&\tilde{\ebold}(0)\ar@{|->}[d]\\
		\ebold(\ell\tau)	&U_{\infty}\ar[r]^{\sim}_{ ^{\rho}\psi}	&\Ucal(0)		&^{\rho}\tilde{\ebold}(0).
	}
\end{equation}
The action of $T_{s}$ on $V_{\infty}$ induces an action of $T_{s}$ on $U_{\infty}$ and $\Ucal(0)$ through the isomorphisms above. On $U_{\infty}$, this automorphism is induced by the translation $\tau\mapsto\tau+1/\ell$.

With these notations, it is clear that given linearly independent elements of the form $\tilde\ebold_{1}(0), \ldots, \tilde\ebold_{r}(0)$, the corresponding $^{\rho}\tilde\ebold_{1}(0),\ldots, \ ^{\rho}\tilde\ebold_{r}(0)$ in $\Ucal(0)$ are linearly independent as well, and reciprocally. For future usage we record the following lemma. For simplicity of exposition we restrict ourselves to lower extensions, but a similar statement holds for other extensions. Notice that for the lower extension, it follows from the construction that there is a natural inclusion $\rho^{\ast}\Vcal_{\log}\subseteq (\rho^{\ast}\Vcal)_{\log}=\Ucal$.
\begin{lemma}\label{lemma:monodromy}
Let $\Vcal_{\log}$ be the lower extension of $(\Vcal,\nabla)$, and let $\sigma\in\Gamma(\DBbb,\Vcal_{\log})$ be such that $\sigma(0)\neq 0$. Define $k\geq 0$ as the largest integer such that $^{\rho}\sigma:=t^{-k}\rho^{\ast}\sigma\in\Ucal$. Then $k\leq\ell-1$ and $^{\rho}\sigma(0)\in\Ucal(0)$ is an eigenvector of $T_{s}$ of eigenvalue $\exp(-2\pi i k/\ell) $.
\end{lemma} 
\begin{proof}
We choose $\ebold_{1},\ldots,\ebold_{r}\in V_{\infty}$ a basis of eigenvectors of $T_{s}$, whose eigenvalues we write $\exp(-2\pi i k_{1}/\ell),\ldots,\exp(-2\pi i k_{r}/\ell)$, where $0\leq k_{j}\leq \ell-1$. By construction, it follows that $^{\rho}\tilde\ebold_{j}(0)\in\Ucal(0)$ is an eigenvector of $T_{s}$ of eigenvalue $\exp(-2\pi i k_{j}/\ell)$. We write
\begin{displaymath}
    \sigma=\sum_{j}f_{j}(q)\tilde\ebold_{j}(q),
\end{displaymath}
where the $f_{j}(q)$ are holomorphic functions on $\DBbb$. Because $\sigma(0)\neq 0$, at least one of the functions $f_{j}$ does not vanish at $0$. Observe that
\begin{equation*}\label{eq:ej-rho}
	\begin{split}
        (\rho^{\ast}\tilde\ebold_{j})(t)&=\tilde\ebold_{j}(t^\ell)=\tilde\ebold_{j}(\ell\tau)\\
                                        &=\exp(-2\pi i (S+N) \ell\tau)\ebold_{j}(\ell\tau)\\
                                        &=\exp(-2\pi i (-k_{j}/\ell) \ell\tau)\exp(-2\pi i N \ell\tau)\ebold_{j}(\ell\tau)\\
                                        &=t^{k_{j}}\exp(-2\pi i N \ell\tau)\ebold_{j}(\ell\tau)
    \end{split}
\end{equation*}
where $\exp(-2\pi i \ell N\tau)\ebold_{j}(\ell\tau)=:^{\rho}\tilde\ebold_{j}(t)$ belongs to $\Ucal$.
Hence, the pull-back $\rho^{\ast}\sigma$ can be written
\begin{displaymath}
    \rho^{\ast}\sigma=\sum_{j}t^{k_{j}}f_{j}(0)\ ^{\rho}\tilde\ebold_{j}(t) + O(t^{\ell}).
\end{displaymath}
Define $k$ as the smallest $k_{j}$ such that $f_{j}(0)\neq 0$. In particular $k\leq\ell-1$. Then
\begin{displaymath}
    t^{-k}\rho^{\ast}\sigma=\sum_{k_{j}=k}f_{j}(0)\  ^{\rho}\tilde\ebold_{j}(0)+\sum_{k_{j}>k}t^{k_{j}-k}f_{j}(0)\ ^{\rho}\tilde\ebold_{j}(0) +O(t^{\ell-k})
\end{displaymath}
and thus
\begin{displaymath}
    (t^{-k}\rho^{\ast}\sigma)(0)=\sum_{j\colon k_{j}=k}f_{j}(0)\ ^{\rho}\tilde\ebold_{j}(0)\neq 0,
\end{displaymath}
which is an eigenvector of $T_{s}$ of eigenvalue $\exp(-2\pi i k/\ell)$. This concludes the proof.
\end{proof}
\begin{definition}
With notations as in Lemma \ref{lemma:monodromy}, we define the elementary exponent of $\sigma$ to be the rational number $\kappa(\sigma)=k/\ell\in\QBbb\cap [0,1)$. Hence, $^{\rho}\sigma(0)$ is an eigenvector of $T_{s}$ of eigenvalue $\exp(-2\pi i\kappa(\sigma))$.
\end{definition}

\subsection{Steenbrink's construction}\label{sec:Steenbrink}
Let $X$ be a complex manifold of dimension $m$ with a normal crossings divisor $D$, such that $D_{\textrm{red}}$ is locally given by an equation $z_1 \ldots z_k = 0$ in suitable holomorphic coordinates $z_{1},\ldots, z_{m}$.  Recall that the sheaf of logarithmic differentials $\Omega_X(\log D)$ is the $\Ocal_{X}$-module locally generated by $\frac{dz_1}{z_1}, \ldots, \frac{dz_k}{z_k}, dz_{k+1}, \ldots, dz_m$.

Let $f\colon X\to \DBbb$ be a projective normal crossings degeneration, with fibres of dimension $n$. We denote its restriction to $\DBbb^{\times}$ by $f^{\times}$. Locally, the special fiber is given by an equation $z_1^{m_1} \ldots z_k^{m_k}=0$. Pull-back of differential forms induces an injection $f^{\ast}\Omega_{\DBbb}(\log[0])\to\Omega_{X}(\log X_{0})$. The sheaf of relative logarithmic forms is then defined by
\begin{displaymath}
    \Omega_{X/\DBbb}(\log) = \Omega_X(\log X_{0} )/f^{\ast}\Omega_\DBbb(\log [0]).
\end{displaymath}
This is a locally free sheaf, and we define $\Omega^{p}_{X/\DBbb}(\log)=\wedge^{p}\Omega_{X/\DBbb}(\log)$. The exterior differential makes $\Omega^{\bullet}_{X/\DBbb}(\log)$ a complex, named the logarithmic de Rham complex of $f$. After Steenbrink \cite[Sec. 2]{Steenbrink-limits}, its $k$-th hypercohomology sheaf $R^k f_\ast \Omega_{X/\DBbb}^\bullet(\log)$ defines an extension of the flat bundle $\Vcal:=R^k f^\times_\ast \bb C\otimes\Ocal_{\DBbb^{\times}}$, for which the Gauss-Manin connection has logarithmic singularities, whose residue has eigenvalues in $[0,1)
\cap \QBbb$. It is hence the lower Deligne extension $\Vcal_{\log}$ of $\Vcal$. With notations as in the previous subsection, we let $V_{\infty}$ be the space of multi-valued flat sections of $\Vcal$ on $\DBbb^{\times}$. By parallel transport, the space $V_{\infty}$ is canonically isomorphic to the $k$-th cohomology of a general fiber of $f$, also denoted $H^{k}(X_{\infty},\CBbb)$ or simply $H^{k}(X_{\infty})$. In this setting, the isomorphism \eqref{eq:psi} gives an isomorphism
\begin{equation}\label{eq:steenbrink-psi}
    \psi\colon H^{k}(X_{\infty})\overset{\sim}{\longrightarrow} R^{k}f_{\ast}\Omega^{\bullet}_{X/\DBbb}(\log)(0).
\end{equation}
It is the composition of the isomorphisms stated in \cite[Prop. 2.16 and Thm. 2.18]{Steenbrink-limits}. Notice that this isomorphism depends on the fixed choice of holomorphic parameter on $\DBbb$.

One defines the Hodge filtration bundles $\mathcal F^p\Vcal_{\log}$ of $\Vcal_{\log}=R^k f_\ast \Omega_{X/\DBbb}^\bullet(\log)$, or simply $\Fcal^{p}_{\log}$, by considering the sections of the Hodge filtration on the smooth part $\Fcal^{p}\Vcal$ that extend to $\Vcal_{\log}$. It is equivalently the filtration induced by the b\^ete filtration $\Omega^{\geq p}_{X/\DBbb}(\log)$ of the complex $\Omega_{X/\DBbb}^{\bullet}(\log)$. If $j\colon\DBbb^{\times}\hookrightarrow\DBbb$ is the inclusion, we may thus write
\begin{equation}\label{eq:extension-Fp}
    \Fcal^{p}_{\log}=j_{\ast}(\Fcal^{p}\Vcal)\cap\Vcal_{\log}\subseteq j_{\ast}\Vcal.
\end{equation}
In \cite[Thm 2.11]{Steenbrink-mixedonvanishing} Steenbrink shows that the sheaves $R^{q}f_{\ast}\Omega_{X/\DBbb}^{p}(\log)$ are locally free. From the $E_{1}$ degeneration of the Hodge-de Rham spectral sequence in the smooth case, it then follows that for $R^{k}f_{\ast}\Omega^{\bullet}_{X/\DBbb}(\log)$ we have
\begin{equation}\label{eq:Gr_p}
    \Fcal^{p}_{\log}/\Fcal^{p+1}_{\log}=R^{k-p}f_{\ast}\Omega^{p}_{X/\DBbb}(\log).
\end{equation}
These are called $(p,k-p)$ \emph{Hodge bundles}.

Let us now consider a semi-stable reduction diagram
\begin{equation}\label{eq:ss-reduction}
    \xymatrix{
        Y\ar[d]_{g}\ar[r] ^{r} &X\ar[d]^{f}\\
        \DBbb\ar[r]^{\rho}  &\DBbb,
    }
\end{equation}
where $\rho(t)=t^{\ell}$ and $r$ is generically finite. The previous discussion applied to $R^{k}g_{\ast}^{\times}\CBbb\otimes\Ocal_{\DBbb^{\times}}$ produces the canonical extension $\Ucal=R^{k}g_{\ast}\Omega_{Y/\DBbb}(\log)$ with its Hodge filtration by vector sub-bundles that we may denote $^{\rho}\Fcal^{p}$. Hence on $R^{k}g_{\ast}\Omega_{Y/\DBbb}(\log)$ we have
\begin{displaymath}
    ^{\rho}\Fcal^{p}/ ^{\rho}\Fcal^{p+1}=R^{k-p}g_{\ast}\Omega^{p}_{Y/\DBbb}(\log).
\end{displaymath}
Observe that the $k$-th cohomology of a general fiber of $g$, $H^{k}(Y_{\infty})$, is naturally identified with $H^{k}(X_{\infty})$ (see the isomorphism \eqref{eq:nu}). Hence we now have an isomorphism (cf. \eqref{eq:psi-rho})
\begin{displaymath}
    ^{\rho}\psi\colon H^{k}(X_{\infty})=H^{k}(Y_{\infty})\overset{\sim}{\longrightarrow} \Ucal(0)=R^{k}g_{\ast}\Omega_{Y/\DBbb}^{\bullet}(\log)(0).
\end{displaymath}
The isomorphisms $\psi$ and $^{\rho}\psi$ are to be compared with \cite[(2.12)]{Steenbrink-mixedonvanishing}. Taking the fiber at $0$ of the Hodge filtration $^{\rho}\Fcal^{\bullet}$ and transporting it to $H^{k}(X_{\infty})$ through $^{\rho}\psi$, we obtain Steenbrink's limiting Hodge filtration on $H^{k}(X_{\infty})$, that we denote by $F^{p}H^{k}(X_{\infty})$ or simply $F^{p}_{\infty}$. To sum up, we have
\begin{equation}\label{eq:Finfty}
    \xymatrix{
        & \Ucal(0)      &\supseteq     & ^{\rho}\Fcal^{p}(0)\\
        & H^{k}(X_{\infty})\ar[u]^{\sim}_{ ^{\rho}\psi}    &\supseteq      & F^{p}_{\infty}\ar[u]^{\sim}_{ ^{\rho}\psi}.
    }
\end{equation}
The limiting Hodge filtration depends on the choice of holomorphic parameter on $\DBbb$. The canonical object is the ``nilpotent orbit'' $\lbrace\exp( \lambda N)F_{\infty}^{\bullet}, \lambda\in\CBbb\rbrace$ in a suitable Grassmannian.

\begin{remark}\label{rmk:V-manifolds}
In \cite{Steenbrink-mixedonvanishing} Steenbrink allows $Y$ in \eqref{eq:ss-reduction} to be the $V$-manifold, and $g$ to be semi-stable in the sense of $V$-manifolds. Such a situation naturally arises as follows. We write $f^{-1}(0)=\sum_{i=1}^{r} m_{i}D_{i}$, where the divisors $D_{i}$ are smooth, intersect transversally, and the $m_{i}$ are their multiplicities in the schematic fiber $f^{-1}(0)$. Let $\ell=\mathrm{lcm}(m_{1},\ldots,m_{r})$, $\rho(t)=t^{\ell}$ and perform the base change $X\times_{\rho}\DBbb$, of $f$ by $\rho$. Then define $Y$ as the normalization of $X\times_{\rho}\DBbb$. Steenbrink shows that $Y$ is a $V$-manifold and the structure morphism $g\colon Y\to\DBbb$ is semi-stable in the sense of $V$-manifolds. Our discussion equally applies to this setting.
\end{remark}

To conclude this subsection, we compare Hodge filtrations and Hodge bundles before and after semi-stable reduction. Notice that there are inclusions, induced by pull-back of differential forms from $X$ to~$Y$,
\begin{gather*}
    \rho^{\ast}R^{k}f_{\ast}\Omega_{X/\DBbb}^{\bullet}(\log)\subseteq R^{k}g_{\ast}\Omega_{Y/\DBbb}^{\bullet}(\log),\\
        \rho^{\ast}\Fcal^{p}_{\log}\subseteq \ ^{\rho}\Fcal^{p},\quad
         \rho^{\ast}R^{q}f_{\ast}\Omega^{p}_{X/\DBbb}(\log)\subseteq R^{q}g_{\ast}\Omega^{p}_{Y/\DBbb}(\log).
\end{gather*}
The respective quotients of these inclusions are torsion sheaves supported at the point $0$. We may thus write
\begin{equation*}\label{eq:aj}
    \frac{^{\rho}\Fcal^{p}}{\rho^{\ast}\Fcal^{p}_{\log}}\simeq\bigoplus_{j=1}^{h^{p}}\frac{\Ocal_{\DBbb,0}}{t^{a_{j}^{p}}\Ocal_{\DBbb,0}},
\end{equation*}
for some integers $a_{j}^{p}\geq 0$, and similarly for the Hodge bundles
\begin{equation*}\label{eq:bj}
    \frac{R^{q}g_{\ast}\Omega^{p}_{Y/\DBbb}(\log)}{\rho^{\ast} R^{q}f_{\ast}\Omega^{p}_{X/\DBbb}(\log)}
    \simeq \bigoplus_{j=1}^{h^{p,q}}\frac{\Ocal_{\DBbb,0}}{t^{b_{j}^{p,q}}\Ocal_{\DBbb,0} },
\end{equation*}
for some integers $b_{j}^{p,q}\geq 0$. We indicated by $h^{p}$ the rank of $\Fcal^{p}_{\log}$ and $h^{p,q}$ the rank of $R^{q}f_{\ast}\Omega^{p}_{X/\DBbb}(\log)$.

\begin{lemma}
The integers $a_{j}^{p}$ and $b_{r}^{p,q}$ satisfy
\begin{displaymath}
    0\leq a_{j}^{p}\leq\ell-1,\quad 0\leq b_{r}^{p,q}\leq\ell-1,
\end{displaymath}
where $\ell=\deg\rho$. Moreover, the rational numbers $a_{j}^{p}/\ell$ and $b_{r}^{p,q}/\ell$ are independent of the choice of semi-stable reduction.
\end{lemma}
\begin{proof}
The first statement is a direct consequence of Lemma \ref{lemma:monodromy}. The second assertion easily follows from two facts: first, in the unipotent case, formation of canonical extensions commutes with pull-back by $t\mapsto t^{m}$ on $\DBbb$; second, any two semi-stable reduction diagrams are dominated by a third one. This is combined with \eqref{eq:extension-Fp}--\eqref{eq:Gr_p}, applied in the unipotent case.
\end{proof}
The second claim of the lemma can be reformulated by saying that the rational numbers $a_{j}^{p}/\ell$ and $b_{j}^{p,q}/\ell$ only depend on the local system $R^{k}f_{\ast}^{\times}\CBbb$ and its associated variation of Hodge structures. We give them a name:

\begin{definition}[Elementary exponents of Hodge bundles]\label{def:elementary}
The rational numbers $\alpha^{p}_{j}:=a_{j}^{p}/\ell$ are called the elementary exponents of $p$-th level of the Hodge filtration for $R^{k}f_{\ast}^{\times}\CBbb$. The rational numbers $\alpha^{p,q}_{j}:=b_{j}^{p,q}/\ell$ are called the elementary exponents of the $(p,q)$ Hodge bundle for $R^{k}f_{\ast}^{\times}\CBbb$ ($k=p+q$). We also denote by $\alpha^{p,q}=\sum_{j=1}^{h^{p,q}}\alpha^{p,q}_{j}$, \emph{i.e.} the sum of all the elementary exponents of the $(p,q)$ Hodge bundle.
\end{definition}
In the rest of this section we relate the elementary exponents to the eigenvalues of $T_{s}$ acting on the limiting Hodge filtration.

\subsection{Elementary exponents and eigenvalues of monodromy} \label{subsec:elementary-divisors}
We maintain the setting and notations of the previous subsections. Hence $f\colon X\to\DBbb$ is a normal crossings projective degeneration over the unit disc. Also, $g\colon Y\to\DBbb$ is a semi-stable reduction of $f$ as in \eqref{eq:ss-reduction}, existing after some finite base change $\rho(t)=t^{\ell}$. We compared the Hodge filtrations $\Fcal^{\bullet}_{\log}$ on $\Vcal_{\log}:=R^{k}f_{\ast}\Omega_{X/\DBbb}^{\bullet}(\log)$ and $^{\rho}\Fcal^{\bullet}$ on $\Ucal:=R^{k}g_{\ast}\Omega_{Y/\DBbb}^{\bullet}(\log)$,  producing the elementary exponents of Definition \ref{def:elementary}. The limiting Hodge filtration $F^{\bullet}_{\infty}$ on $H^{k}(X_{\infty})=H^{k}(Y_{\infty})$ was obtained after identifying the latter with $\Ucal(0)$ via the isomorphism $^{\rho}\psi$, and then transporting the filtration $^{\rho}\Fcal^{\bullet}(0)$ through this identification (cf. \eqref{eq:Finfty}). Recall that the semi-simple part of the monodromy operator $T_{s}$ acts on $\Ucal(0)$.

\begin{theorem}\label{thm:eigenvalues}
Let $\sigma_{1},\ldots,\sigma_{h}$ be a basis of the $\Ocal_{\DBbb}$-module $\Fcal^{p}_{\log}$. Then there exists another basis $\theta_{1},\ldots,\theta_{h}$ with the following properties:
\begin{enumerate}[(a)]
	\item\label{item:eigen-1} for every $m=1,\ldots, h$, we have
	\begin{displaymath}
		\bigoplus_{j=1}^{m}\Ocal_{\DBbb}\sigma_{j}=\bigoplus_{j=1}^{m}\Ocal_{\DBbb}\theta_{j}.
	\end{displaymath}
	\item\label{item:eigen-2} there exist integers $0\leq a_{j}\leq \ell-1$ such that the elements $^{\rho}\theta_{j}:=t^{-a_{j}}\rho^{\ast}\theta_{j}$ define a $\Ocal_{\DBbb}$-basis of $^{\rho}\Fcal^{p}$.
	\item\label{item:eigen-3} the elementary exponent $\kappa(\theta_{j})$ equals $a_{j}/\ell$. Hence the fiber at 0 element $^{\rho}\theta_{j}(0)\in \Ucal(0)$ is an eigenvector of $T_{s}$ of eigenvalue $\exp(-2\pi i a_{j}/\ell)$.
\end{enumerate}
In particular, the operator $T_{s}$ preserves the limiting Hodge filtration $F_{\infty}^{\bullet}$ and the rational numbers $a_{j}/\ell$ are, modulo reordering, the elementary exponents of $p$-th level of the Hodge filtration for $R^{k}f_{\ast}^{\times}\CBbb$.
\end{theorem}

\begin{proof}
 Fix a basis of $V_{\infty}$, say $\ebold_{1},\ldots,\ebold_{r}$, made of eigenvectors of $T_{s}$. We write their eigenvalues as $\exp(-2\pi i k_{1}/\ell), \ldots,\  \exp(-2\pi i k_{r}/\ell)$, with $0\leq k_{j}\leq\ell -1$ ordered increasingly. Recall the notations $\tilde\ebold_{j}$ and $^{\rho}\tilde\ebold_{j}$ for the associated $\Ocal_\DBbb$-module bases of $\Vcal_{\log}$ and $\Ucal$. We construct the basis $\theta_{1},\ldots,\theta_{h}$ inductively.\\

We start by applying Lemma \ref{lemma:monodromy} to  $\sigma_{1}$. Hence there exists an integer $a_{1}$, with  $0\leq a_{1}\leq\ell-1$, such that $^{\rho}\sigma_{1}:=t^{-a_{1}}\rho^{\ast}\sigma_{1}\in\Ucal$ and $^{\rho}\sigma_{1}(0)$ is an eigenvector of $T_{s}$ (in particular non-zero). Hence \eqref{item:eigen-1} is satisfied for $k=1$ with the choice $\theta_{1}=\sigma_{1}$. Also \eqref{item:eigen-3} is satisfied.\\

Suppose we already constructed $\theta_{1},\ldots,\theta_{m}$ as in the statement, for some $m<h$. More precisely, this means:
\begin{enumerate}
    \item[\textbullet] condition \eqref{item:eigen-1} holds restricted to the range $1,\ldots, m$;
    \item[\textbullet] for every $j\leq m$ there exists an integer $0\leq a_{j}\leq\ell-1$ such that $^{\rho}\theta_{j}=t^{-a_{j}}\rho^{\ast}\theta_{j}$ belongs to $^{\rho}\Fcal^{p}$, and the vectors $^{\rho}\theta_{1}(0),\ldots, \ ^{\rho}\theta_{m}(0)$ are linearly independent;
    \item[\textbullet] $\kappa(\theta_{j})=a_{j}/\ell$ and $^{\rho}\theta_{j}(0)$ is an eigenvector of $T_{s}$ of eigenvalue $\exp(-2\pi a_{j}/\ell)$.
\end{enumerate}
In particular, the vector $\sigma_{m+1}(0)$ is linearly independent with $\theta_{1}(0),\ldots,\theta_{m}(0)$. We apply Lemma \ref{lemma:monodromy} with $\sigma=\sigma_{m+1}$. We write the $T_{s}$-eigenvalue of $^{\rho}\sigma(0)$ as $\exp(-2\pi i a/\ell)$, for some $0\leq a\leq\ell-1$. Two cases can occur. In the first case, $^{\rho}\sigma(0)$ is already linearly independent with $^{\rho}\theta_{1}(0),\ldots, \ ^{\rho}\theta_{m}(0)$. Then we take $\theta_{m+1}=\sigma$ and $a_{m+1}=a$. In the second case, there is a non-trivial linear relation:
\begin{equation}\label{eq:relation-at-0}
	^{\rho}\sigma(0)=\sum_{j=1}^{m}\mu_{j} \ ^{\rho}\theta_{j}(0).
\end{equation}
We infer for the coefficients $\mu_{j}$ that, either $\mu_{j}=0$, or $\mu_{j}\neq 0$ and then $^{\rho}\theta_{j}(0)$ has eigenvalue $\exp(-2\pi i a/\ell)$. Hence, if $\mu_{j}\neq 0$ then $a_{j}=a$ and $^{\rho}\theta_{j}=t^{-a}\rho^{\ast}\theta_{j}$. We denote by $J(\sigma)$ the non-empty subset of $\lbrace 1,\ldots,m\rbrace$ for which $a_{j}=a$. Define now the sections
\begin{equation*}\label{def:sigma-prime}
	\sigma^{\prime}:=\sigma-\sum_{j\in J(\sigma)}\mu_{j}\theta_{j}\in\Gamma(\DBbb,\Fcal^{p}_{\log})
\end{equation*}
and
\begin{displaymath}
	\hat\sigma^{\prime}=t^{-a}\rho^{\ast}\sigma^{\prime}= \ ^{\rho}\sigma-\sum_{j\in J(\sigma)}\mu_{j} \ ^{\rho}\theta_{j}\in\Gamma(\DBbb,\ ^{\rho}\Fcal^{p}).
\end{displaymath}
Due to \eqref{eq:relation-at-0}, this section is necessarily of the form
\begin{equation*}\label{eq:hat-sigma-prime}
	\hat\sigma^{\prime}=t^{b-a}\sum_{k_{j}=b}\gamma_{j} \ ^{\rho}\tilde\ebold_{j}(t)+\sum_{k_{j}>b}t^{k_{j}-a}\gamma_{j} \ ^{\rho}\tilde\ebold_{j}(t)+t^{-a} R(t^{\ell}),
\end{equation*}
for some constants $\gamma_{j}$, some minimal $b>a$ and $R(t^{\ell})=\rho^{\ast}R (s)$, where $R(s)$ a regular section of $\Vcal_{\log}$ vanishing at $s=0$. Notice the first sum is non-trivial. Otherwise, by the minimality of $b$, it would necessarily be $\sigma^{\prime}=R(s)$, which vanishes at 0 and hence entails
\begin{displaymath}
	\sigma(0)=\sum_{j}\mu_{j}\theta_{j}(0),
\end{displaymath}
contradicting the linear independence of $\sigma(0)=\sigma_{m+1}(0)$ with $\theta_{1}(0),\ldots,\theta_{m}(0)$ observed before.  Now
\begin{displaymath}
	t^{-(b-a)}\hat\sigma^{\prime}=t^{-b}\rho^{\ast}\sigma^{\prime}= \ ^{\rho}\sigma^{\prime}
\end{displaymath}
is a regular section of $^{\rho}\Fcal^{p}$ and its value at 0 is
\begin{displaymath}
	\sum_{k_{j}=b}\gamma_{j} \ ^{\rho}\tilde\ebold_{j}(0),
\end{displaymath}
which is non-zero and an eigenvector of $T_{s}$ of eigenvalue $\exp(-2\pi i b/\ell)$. Hence $\kappa(\sigma^{\prime})=b/\ell$. Furthermore, $^{\rho}\sigma^{\prime}$ is independent with the sections $^{\rho}\theta_{j}$ with $j\in J(\sigma)$. Indeed, the eigenvalue of $^{\rho}\sigma^{\prime}(0)$ is $\exp(-2\pi ib/\ell)$ and for $j\in J(\sigma)$ the eigenvalue of $^{\rho}\theta_{j}(0)$  is $\exp(-2\pi ia/\ell)$, and $b>a$. Still, it could be that $^{\rho}\sigma^{\prime}(0)$ is linearly dependent with the $^{\rho}\theta_{j}(0)$, with $j\not\in J(\sigma)$. If such is the case, then we repeat the argument, starting with a non-trivial linear relation
\begin{equation*}\label{eq:rho-sigma-prime}
	^{\rho}\sigma^{\prime}(0)=\sum_{j\in\lbrace 1,\ldots,m\rbrace\setminus J(\sigma)}\mu_{j}^{\prime} \ ^{\rho}\theta_{j}(0).
\end{equation*}
In particular, this produces a new subset of indices $J(\sigma^{\prime})\subset\lbrace 1,\ldots,m\rbrace$ disjoint with $J(\sigma)$. This procedure clearly comes to an end after a finite number of steps, since $J(\sigma^{\prime})\cap J(\sigma)=\emptyset$ and $\kappa(\sigma^{\prime})=b/\ell>a/\ell=\kappa(\sigma)$. The outcome is a section $\theta_{m+1}$ with the expected properties. Namely, the sections $\theta_{1},\ldots,\theta_{m+1}$ are independent and satisfy
\begin{displaymath}
	\bigoplus_{j=1}^{m+1}\Ocal_{\DBbb}\sigma_{j}=\bigoplus_{j=1}^{m+1}\Ocal_{\DBbb}\theta_{j}.
\end{displaymath}
Furthermore, for some $0\leq a_{m+1}\leq\ell-1$, $^{\rho}\theta_{m+1}=t^{-a_{m+1}}\rho^{\ast}\theta_{m+1}$ is a regular section of $^{\rho}\Fcal^{p}$, and its value at 0 is an eigenvector of $T_{s}$ of eigenvalue $\exp(-2\pi i a_{m+1}/\ell)$.
\end{proof}
\begin{remark}\label{rmk:complements}
\begin{enumerate}[(i)]
    \item Recall that the limiting Hodge structure $F_{\infty}^{\bullet}$ depends on the choice of holomorphic coordinate on $\DBbb$. However, the nilpotent orbit $\lbrace\exp(\lambda N)F_{\infty}^{\bullet}, \lambda\in\CBbb\rbrace$ is canonical. The operators $T_{s}$ and $\exp(N\lambda)$ commute, and therefore  $T_{s}$ preserves $\exp(\lambda N)F_{\infty}^{p}$ as well. Also, the eigenvalues of $T_{s}$ on $F_{\infty}^{p}$ and $\exp(\lambda N)F_{\infty}^{p}$ are the same, and hence they only depend on the nilpotent orbit. This is consistent with the fact that the elementary exponents do not depend on the choice of coordinate on $\DBbb$. It follows that, for the purpose of comparing the elementary exponents with the monodromy eigenvalues, $\DBbb$ can be taken to have any radius, as opposed to the radius 1 assumption we made so far.
    \item For one-variable variations of \emph{polarized} Hodge structures, the invariance of the limiting Hodge filtration under the action of $T_{s}$ is already observed by Schmid \cite[(4.9)]{schmid}. A different argument for the invariance in the geometric case is provided by Steenbrink \cite[Thm. 2.13]{Steenbrink-mixedonvanishing}.
    \item\label{rmk:compl-3} Theorem \ref{thm:eigenvalues} can be stated in the more general context of variations of Hodge structures. We restricted to the geometric case for the sake of conciseness.
\end{enumerate}
\end{remark}
As a consequence, we derive the corresponding statement for the elementary exponents of the Hodge bundles:
\begin{corollary}\label{cor:eigenvalues}
The elementary exponents $\alpha^{p,q}_{j}$ of the $(p,q)$ Hodge bundle for $R^{k}f_{\ast}^{\times}\CBbb$ ($k=p+q$) are such that the $\exp(-2\pi i \alpha^{p,q}_{j})$ constitute the eigenvalues of $T_{s}$ acting on $\Gr_{F_{\infty}}^{p}H^{k}(X_{\infty})$ (with multiplicities). In particular, we have
\begin{displaymath}
    \alpha^{p,q}=-\frac{1}{2\pi i}\tr\left(^{\ell}\log T_{s}\mid \Gr^{p}_{F_{\infty}}H^{p+q}(X_{\infty})\right).
\end{displaymath}
\end{corollary}
\begin{proof}
Recall that the graded quotients $\Fcal^{p}_{\log}/\Fcal^{p+1}_{\log}$ are locally free and isomorphic to $R^{k-p}f_{\ast}\Omega^{p}_{X/\DBbb}(\log)$. Similarly after semi-stable reduction. The result easily follows from Theorem \ref{thm:eigenvalues}, by choosing a basis $\sigma_{1},\ldots,\sigma_{h}$ of $\Fcal^{p+1}_{\log}$ and completing it into a basis $\sigma_{1},\ldots,\sigma_{r}$ of $\Fcal^{p}_{\log}$, in such a way that the $\sigma_{h+1},\ldots,\sigma_{r}$ project onto a basis of $R^{k-p}f_{\ast}\Omega^{p}_{X/\DBbb}(\log)$. 
\end{proof}
\begin{remark}\label{rmk:eigen-prim}
\begin{enumerate}[(i)]
    \item Steenbrink's construction of the limiting Hodge filtration via $V$-manifolds, recalled in Remark \ref{rmk:V-manifolds}, shows that we may take $\ell=\operatorname{lcm}(m_{1},\ldots, m_{r})$ , where the $m_{i}$ are the multiplicities of the irreducible components of the schematic fiber $f^{-1}(0)$.
    \item\label{rmk:eigen-prim-ii} In the polarized setting and for primitive cohomology groups, one has analogous statements to Theorem \ref{thm:eigenvalues} and Corollary \ref{cor:eigenvalues}, formally with the same proof. The formulation and details  are left to the reader.
\end{enumerate}
\end{remark}

For other extensions, the same argument gives an analogous result of Theorem \ref{thm:eigenvalues}. However, it is not clear that the successive quotients of the (extended) Hodge filtration are locally free. Nonetheless, for the upper extension, the situation is much better. Precisely, let $^{u}\Vcal$ be the upper extension of the local system $\Vcal=R^{k}f^{\times}_{\ast}\CBbb\otimes\Ocal_{\DBbb^{\times}}$ (see \textsection\ref{subsec:Deligne-extension}), and extend the Hodge filtration analogously to \eqref{eq:extension-Fp}. Then, by \cite[Prop. 2.9 \& Lemma 2.11]{KollarHigher2}\footnote{The discussion in \emph{loc. cit.} and the references used therein are valid over a one-dimensional disc.}, the extended Hodge filtration $^{u}\Fcal^{\bullet}$ has locally free successive quotients $^{u}\Fcal^{p}/^{u}\Fcal^{p+1}$. By \emph{loc. cit.}, if $k=n+i$ the last piece of the filtration has the form 
\begin{displaymath}
    ^{u}\Fcal^{n}=R^{i}f_{\ast}\omega_{X/\DBbb}. 
\end{displaymath}
We thus have a counterpart of Corollary \ref{cor:eigenvalues} which generalizes our previous \cite[Theorem A]{cdg} for the direct image of the relative canonical sheaf of a degenerating family of Calabi--Yau varieties:

\begin{proposition}\label{prop:kawamata-style}
Let $f: X \to \DBbb$ be a projective degeneration between complex manifolds. If $g: Y \to \DBbb$ denotes a semi-stable reduction as in diagram \eqref{eq:ss-reduction}, there is a natural inclusion of locally free sheaves
$$R^i g_* \omega_{Y/\DBbb} \to \rho^* R^i f_* \omega_{X/\DBbb}.$$
The quotient is of the form $ \bigoplus_{j=1}^{h^{n,i}}\frac{\Ocal_{\DBbb,0}}{t^{a_j}\Ocal_{\DBbb,0} }$. Then the rational numbers $a_j/\deg \rho \in [0,1) \cap \mathbb{Q}$ are independent of the choice of semi-stable reduction, and $\exp(2\pi i a_j/\deg\rho)$ are the eigenvalues of $T_s$ acting on $F^n_{\infty} H^{n+i} (X_{\infty}).$
\end{proposition}

\section{K\"ahler extensions of determinants of Hodge bundles}

In the theory of the BCOV line bundle, it is frequent to work with the K\"ahler extensions of the sheaves of differentials, rather than their logarithmic counterparts (see \emph{e.g.} \cite{FLY, cdg}). In this section we compare the determinant of cohomology of these two kinds of extensions. While in the comparison properties addressed in the previous section we found eigenvalues of the monodromy operator, in this section we encounter other invariant such as the total dimension of vanishing cycles.

\subsection{K\"ahler extensions}
Let $f\colon X\to S$ be a projective degeneration between complex manifolds, of relative dimension $n$. We assume that $S$ is a connected complex curve. Let $S^{\times}\subseteq S$ be the (non-empty and Zariski open) locus of regular values of $f$. Write $X^{\times}=f^{-1}(S^{\times})$. Then $\Omega^{p}_{X^{\times}/S^{\times}}$ is a locally free sheaf and its determinant of cohomology $\lambda(\Omega^{p}_{X^{\times}/S^{\times}})$ is defined.

\begin{definition}\label{def:kahler-extension}
For each $0\leq p\leq n$, consider the complex of locally free sheaves
\begin{eqnarray*}\label{eq:kahler-res}
\widetilde{\Omega}^p_{X/S} \colon
\left(f^\ast \Omega_S\right)^{\otimes p}\to
\left(f^\ast \Omega_S\right)^{\otimes p-1}\otimes \Omega_X\to\cdots
\to
\left(f^\ast \Omega_S\right)\otimes \Omega_X^{p-1}\to
\Omega_X^p,
\end{eqnarray*}
quasi-isomorphic on $X^{\times}$ to $\Omega^p_{X^{\times}/S^{\times}}$.
We define the \emph{K\"ahler extension} of $\lambda(\Omega_{X^\times/S^\times}^p)$, as
\begin{equation}\label{def:kahlerext}
    \lambda(\widetilde{\Omega}_{X/S}^p) = \bigotimes_{k=0}^{p}\lambda\left(\left(f^\ast \Omega_S\right)^{\otimes k}\otimes \Omega_X^{p-k}\right)^{(-1)^{k}}.
\end{equation}
\end{definition}

\begin{remark}
In \cite[Example 4.2]{SaitoTakeshiBored}, the complex $\widetilde{\Omega}^p_{X/S} $ is referred to as the derived exterior power complex $L\Lambda^q K$, applying a functor $L\Lambda^q$ to $K = [f^* \Omega_{S} \to \Omega_X].$ We refer the reader to \emph{loc. cit.} for further discussions of this notion.
\end{remark}
\subsection{Compatibility with Serre duality}
We study the compatibility of the formation of $\lambda(\widetilde{\Omega}^{p}_{X/S})$ with Serre duality. For simplicity, we restrict to the local case $S=\DBbb$ and $f$ is a submersion on $\DBbb^{\times}$. As $X$ is smooth, the relative dualizing sheaf $\omega_{X/\DBbb}$ of $f\colon X\to \DBbb$ is isomorphic to $\det\Omega_{X/\DBbb}\simeq K_{X}\otimes K_{\DBbb}^{-1}$. Then there is a canonical pairing
\begin{displaymath}
    \widetilde{\Omega}^p_{X/\DBbb} \times \widetilde{\Omega}^{n-p}_{X/\DBbb} \longrightarrow \widetilde{\Omega}^n_{X/\DBbb}  \longrightarrow \omega_{X/\DBbb}.
\end{displaymath}
The pairing is described in \cite[p. 420]{SaitoTakeshiBored}\footnote{To be strictly conform with \emph{loc. cit.}, we should work with the projective morphism of regular schemes over $\Spec\CBbb\lbrace t\rbrace$ induced by $f$. The reader will get easily convinced that this abuse of notation is justified, for the results of this section can be checked after localization at the origin.}
and extends the natural pairing on relative K\"ahler differentials. In particular, on the smooth locus, it is a perfect pairing. We conclude there is a morphism of complexes
\begin{equation*}\label{eq:rho-p}
    \rho_p: \widetilde{\Omega}_{X/\DBbb}^p \longrightarrow R\Hom(\widetilde{\Omega}_{X/\DBbb}^{n-p}, \omega_{X/\DBbb}).
\end{equation*}
The cone of $\rho_p$ is acyclic on the smooth locus and its homology is supported on the singular locus. By  Grothendieck-Serre duality, we have a quasi-isomorphism
\begin{equation*}\label{GrothSerre}
    Rf_\ast R\Hom(\widetilde{\Omega}^{n-p}_{X/\DBbb}, \omega_{X/\DBbb}[n])\simeq R\Hom(Rf_\ast \widetilde{\Omega}^{n-p}_{X/\DBbb}, \Ocal_\DBbb).
\end{equation*}
Taking determinants we obtain an isomorphism
\begin{displaymath}
    \lambda\left(R\Hom\left(\widetilde{\Omega}_{X/\DBbb}^{n-p}, \omega_{X/\DBbb}\right) \right) \simeq \lambda(\widetilde{\Omega}^{n-p}_{X/\DBbb})^{(-1)^{n+1}}.
\end{displaymath}
We conclude that there is an isomorphism, written additively for lighter notations,
\begin{displaymath}
    \lambda(\widetilde{\Omega}^p_{X/\DBbb}) \simeq (-1)^{n+1}\lambda(\widetilde{\Omega}^{n-p}_{X/\DBbb}) - \delta_p \cdot \Ocal([0])
\end{displaymath}
where $\delta_p \in \bb Z$. Outside of the origin, it extends the (canonical) Serre duality isomorphism. By d\'evissage we find that
\begin{displaymath}
    \delta_p = \sum (-1)^k \ell_{\Ocal_{\DBbb,0}} \left(R^k f_\ast \operatorname{cone}(\rho_p) \right).
\end{displaymath}
Here $\ell_R$ denotes the length of an $R$-module. It follows from \cite[Cor. 4.5]{SaitoTakeshiBored},
\begin{equation*}\label{eq:takeshicomp}
    \delta_p = (-1)^{n-p} \deg c_{n+1}^{X_0}(\Omega_{X/\DBbb}).
\end{equation*}
Since we suppose that the total space $X$ is smooth, we have the formula (cf. \cite[Ex. 14.1.5]{Fulton})
\begin{equation}\label{eq:classeloc}
\deg c_{n+1}^{X_0}(\Omega_{X/\DBbb}) = (-1)^n \left(\chi(X_\infty) - \chi(X_0)\right)
\end{equation}
where $\chi(X_\infty)$ (resp. $\chi(X_0)$) denotes the topological Euler characteristic of a general (resp. special) fiber of $f\colon X \to \DBbb$.

We summarize the above observations in the following proposition:
\begin{proposition}\label{prop:kahlercomp}
Outside the origin, Serre duality induces an isomorphism
\begin{displaymath}
    \lambda(\Omega^p_{X^\times/\DBbb^\times}) \simeq (-1)^{n+1}\lambda(\Omega^{n-p}_{X^\times/\DBbb^\times}).
\end{displaymath}
It extends to an isomorphism of K\"ahler extensions
\begin{displaymath}
    \lambda(\widetilde{\Omega}^p_{X/\DBbb}) \simeq (-1)^{n+1} \lambda(\widetilde{\Omega}^{n-p}_{X/\DBbb}) + (-1)^{n+1-p} c \cdot \Ocal([0]).
\end{displaymath}
Here $c=\deg c_{n+1}^{X_0}(\Omega_{X/\DBbb})$ is the degree of the localized top Chern class, which is computed by \eqref{eq:classeloc}.
\end{proposition}

\begin{corollary}\label{cor:2lambda}
In even relative dimension $n=2m$, there is an isomorphism
\begin{displaymath}
    2 \lambda(\widetilde{\Omega}_{X/\DBbb}^m) \simeq (-1)^{m+1} c \cdot \Ocal([0]).
\end{displaymath}
\end{corollary}
The above results adapt to a general one dimensional base $S$ and the case of several singular fibers, and integrating over a compact Riemann surface we obtain the following corollary:
\begin{corollary}\label{cor:funny}
Suppose that $S$ is a compact Riemann surface and $f\colon X\to S$ is a degeneration with at most ordinary double point singularities, of even relative dimension. Then the number of singularities is even.
\end{corollary}
\begin{proof}
For ordinary double point singularities, it follows from \eqref{eq:classeloc} that the localized Chern classes compute the number of singular points in the fibers (cf. \cite[Ex. 14.1.5 (d)]{Fulton}). We conclude by Corollary \ref{cor:2lambda}.

\end{proof}
\begin{remark}
In the special case of a Lefschetz pencil of degree $d$ hypersurfaces in $\bb P^{2m+1}$, the above corollary indicates that there are an even number of singular fibers. This is compatible with the fact that the degree of the discriminant variety is $(2m+2)(d-1)^{2m+1}$ by Boole's formula \cite[Chap. 1]{GKZ}.
\end{remark}

\subsection{Comparison with the logarithmic extensions}\label{subsec:comparison}
Let $f\colon X\to \DBbb$ be as before. If the special fiber $X_{0}$ has normal crossings, then $\lambda(\Omega^{p}_{X^{\times}/\DBbb^{\times}})$ affords two natural extensions: the K\"ahler extension $\lambda(\widetilde{\Omega}_{X/\DBbb}^{p})$ and the logarithmic extension $\lambda(\Omega_{X/\DBbb}^{p}(\log))$. Both can be compared:
\begin{equation}\label{eq:def-mu-p}
    \lambda(\widetilde{\Omega}_{X/\DBbb}^{p})=\lambda(\Omega_{X/\DBbb}^{p}(\log))+\mu_{p}\cdot\Ocal([0]),
\end{equation}
for some integer $\mu_{p}$. We now describe general expressions for $\mu_{p}$ in terms of the geometry of $X_{0}$. More generally, without any assumption on the special fiber, we can reduce to the normal crossings case by an embedded resolution of singularities. One then needs to keep track of the change of the determinant of cohomology of the K\"ahler differentials under the blowing-up process. This setting will be partially treated in a second step, after the easier normal crossings case.

Suppose that the special fiber of  $f\colon X\to\DBbb$  is a normal crossings divisor of the form $X_{0}=\sum_{i=1}^r m_{i}D_{i}$, with smooth $D_{i}$. Define for $I \subseteq \{1, \ldots, r \}$, $D_I = \bigcap_{i \in I} D_i$. For $k\geq 0$ an integer, also define the codimension $k$ stratum in $X$
\begin{displaymath}
    D(k)=\bigsqcup_{\substack{I \subseteq \{1, \ldots, r \}\\ |I|=k}} D_I
\end{displaymath}
and denote by $a_k: D(k) \to X_0$ the natural map. Finally, given a cohomological complex $\Ccal^{\bullet}$, denote by $\Ccal^{\bullet\leq k}$ the b\^ete truncation of the complex up to degree $k$, namely $\ldots\to \Ccal^{k-2}\to\Ccal^{k-1}\to\Ccal^{k}$. This notation will be applied below for holomorphic de Rham complexes.

\begin{proposition}\label{prop:mu_p}
\label{prop:delta1} Assume that $f\colon X\to\DBbb$ is a projective normal crossings degeneration, and write $X_{0}=\sum_{i}m_{i}D_{i}$, where the $D_{i}$ are the reduced irreducible components, assumed to be smooth. Then
\begin{displaymath}
   \mu_{p}=(-1)^{p-1}\chi(\Omega_{X_{\infty}}^{\bullet \leq p-1})-\sum_{k=1}^{p}(-1)^{p-k}\chi(\Omega_{D(k)}^{\bullet \leq p-k}).
\end{displaymath}
In particular,
\begin{displaymath}
    \mu_1= \chi(\Ocal_{X_{\infty}})-\sum \chi(\Ocal_{D_i}).
\end{displaymath}
If moreover $f$ is semi-stable, i.e. $X_{0}$ is in addition reduced, then
\begin{displaymath}
    \mu_1 = -\sum_{k \geq 2} (-1)^{k} \chi(\Ocal_{D(k)}).
\end{displaymath}
\end{proposition}

\begin{proof}
Applying the projection formula to the definition \eqref{def:kahlerext} we find
\begin{displaymath}
    \lambda(\widetilde{\Omega}_{X/\DBbb}^p)= \sum_{k=0}^p (-1)^k \det (\Omega_\DBbb^{\otimes k} \otimes R f_* \Omega^{p-k}_{X}).
\end{displaymath}
Here $Rf_{\ast}\Omega^{p-k}_{X}$ is the whole right derived image of $\Omega_{X}^{p-k}$, which is a perfect complex on $\DBbb$. We now want to use that for a product of two perfect complexes $A,B$, we have $\det (A\otimes B) = (\rk B)(\det A)+ (\rk A)(\det B)$. The rank of $Rf_* \Omega_X^{p-k}$ can be computed on the generic fiber $X_{\infty}$, and hence equals to $\chi(\Omega_X^{p-k}|_{X_\infty})$.
From the cotangent sequence of $f$, one can derive the relation
\begin{displaymath}
    \chi(\Omega^{p-k}_{X}|_{X_{\infty}})=\chi(\Omega_{X_{\infty}}^{p-k})+\chi(\Omega_{X_{\infty}}^{p-k-1}),
\end{displaymath}
with the convention $\chi(\Omega_{X_{\infty}}^{-1})=0$. We find that the determinant K\"ahler extension can be written
\begin{equation}\label{eq:kahler-res-simpl}
\begin{split}
    \lambda(\widetilde{\Omega}_{X/\DBbb}^{p})= &\sum_{k=0}^p (-1)^k k \left(\chi(\Omega^{p-k}_{X_{\infty}})+\chi(\Omega^{p-k-1}_{X_{\infty}})\right) \Omega_{\DBbb} +\sum_{k=0}^p (-1)^k \lambda(\Omega_{X}^{p-k})\\
    =&
   \sum_{k=0}^{p-1}(-1)^{p-k}\chi(\Omega^{k}_{X_{\infty}})\Omega_{\DBbb}
   +\sum_{k=0}^p (-1)^k \lambda(\Omega_{X}^{p-k})
    =(-1)^{p}\chi(\Omega_{X_{\infty}}^{\bullet \leq p-1})\Omega_{\DBbb}
    +\sum_{k=0}^p (-1)^k \lambda(\Omega_{X}^{p-k}).
    \end{split}
\end{equation}
On the other hand, the sheaf $\Omega_{X/\DBbb}(\log)$ is already locally free, so in fact the $p$-th derived exterior power is a resolution of $\Omega_{X/\DBbb}^{p}(\log)$. Then in a similar way we find
\begin{equation} \label{eq:omegalogderived}
    \lambda(\Omega_{X/\DBbb}^p(\log)) = (-1)^{p}\chi(\Omega_{X_{\infty}}^{\bullet \leq p-1}) \Omega_{\DBbb}(\log [0])+\sum_{k=0}^p (-1)^k \lambda(\Omega_{X}^{p-k}(\log X_{0})).
\end{equation}
Comparing \eqref{eq:kahler-res-simpl} and \eqref{eq:omegalogderived}, we find
\begin{equation}\label{eq:difference-dets-1}
      \lambda(\Omega_{X/\DBbb}^{p}(\log))-\lambda(\widetilde{\Omega}_{X/\DBbb}^{p})=
       (-1)^{p}\chi(\Omega_{X_{\infty}}^{\bullet \leq p-1})\Ocal([0])
      +\sum_{k=0}^p (-1)^k \left(\lambda(\Omega_{X}^{p-k}(\log X_{0}))-\lambda(\Omega_{X}^{p-k})\right).\\
\end{equation}

For the second term on the right hand side of \eqref{eq:difference-dets-1}, we introduce the filtration by the order of poles on $\Omega_{X}^{k}(\log)$:
\begin{displaymath}
    W_m \Omega_X^k (\log X_{0}) = \Omega_X^{k-m} \wedge \Omega_X^m (\log X_{0}),\quad m\leq k\quad\text{with}\quad \Gr^W_m \Omega_X^k (\log X_{0})\overset{\sim}{\rightarrow} (a_{m})_{\ast} \Omega_{D(m)}^{k-m}.
\end{displaymath}
We find
\begin{displaymath}
    \lambda(\Omega_{X}^{k}(\log X_{0}))-\lambda(\Omega_{X}^{k})=\sum_{m=1}^{k}\lambda((a_{m})_{\ast}\Omega_{D(m)}^{k-m})=\sum_{m=1}^{k}\chi(\Omega_{D(m)}^{k-m})\Ocal([0]),
\end{displaymath}
and then
\begin{equation}\label{eq:difference-dets-3}
    \sum_{k=0}^p (-1)^k \left(\lambda(\Omega_{X}^{p-k}(\log X_{0}))-\lambda(\Omega_{X}^{p-k})\right)=\sum_{k=1}^{p}(-1)^{p-k}\chi(\Omega_{D(k)}^{\bullet\leq p-k})\Ocal([0]).
\end{equation}
To sum up, equations \eqref{eq:difference-dets-1}--\eqref{eq:difference-dets-3} combine to
\begin{displaymath}
    \lambda(\Omega_{X/\DBbb}^{p}(\log))-\lambda(\widetilde{\Omega}_{X/\DBbb}^{p})
    =\left((-1)^{p}\chi(\Omega_{X_{\infty}}^{\bullet \leq p-1})+\sum_{k=1}^{p}(-1)^{p-k}\chi(\Omega_{D(k)}^{\bullet\leq p-k})\right)\Ocal([0]),
\end{displaymath}
hence
\begin{displaymath}
    \mu_{p}=(-1)^{p-1}\chi(\Omega_{X_{\infty}}^{\bullet \leq p-1})-\sum_{k=1}^{p}(-1)^{p-k}\chi(\Omega_{D(k)}^{\bullet \leq p-k}),
\end{displaymath}
as was to be shown. This completes the proof of the first part of the statement. The first claimed expression for $\mu_{1}$ is readily checked. In the semi-stable case, we consider the natural exact sequence
\begin{equation}\label{eq:resolutionnormalcrossing}
     0\to \Ocal_{X_0}\to \Ocal_{D(1)}\to\Ocal_{D(2)}\to\Ocal_{D(3)}\to\ldots\to \Ocal_{D(n)}\to 0.
\end{equation}
Taking Euler characteristics, we derive
\begin{displaymath}
    \chi(\Ocal_{X_0})=-\sum_{k=1}^{n}(-1)^{k}\chi(\Ocal_{D(k)}).
\end{displaymath}
By flatness of $f\colon X\to\DBbb$ we have $\chi(\Ocal_{X_{\infty}})=\chi(\Ocal_{X_{0}})$. The second claimed expression for $\mu_{1}$ follows.
\end{proof}

We now discuss a variant of Proposition \ref{prop:mu_p} for morphisms which are not necessarily in normal crossings form. Before the statement, we need a preliminary observation. If $f\colon X\to\DBbb$ is a projective normal crossings degeneration, with $X_{0}=\sum m_{i}D_{i}$ as above, we have a commutative diagram of exact sequences
\begin{equation}\label{eq:diagram-res}
    \xymatrix{
 &   &   &  & 0 \ar[d]& \\
  &          &  0 \ar[d] & 0 \ar[d] & P \ar[d]& \\
   &      0 \ar[r] & f^\ast\Omega_\DBbb\ar[r] \ar[d] & f^\ast\Omega_\DBbb (\log [0])  \ar[d] \ar[r]^-{\mathrm{res}} & \Ocal_{X_0} \ar[r] \ar[d] & 0\\
    &     0 \ar[r] &   \Omega_X \ar[r] \ar[d] & \Omega_X(\log X_{0}) \ar[d] \ar[r]^-{\mathrm{res}} &  \Ocal_{D(1)} \ar[r] \ar[d] & 0\\
     0 \ar[r] &    P  \ar[r] & \Omega_{X/\DBbb} \ar[r] \ar[d]&\Omega_{X/\DBbb}(\log )  \ar[d]\ar[r] & Q \ar[d] \ar[r] & 0\\
      &      & 0  &0    & 0 &
    }
\end{equation}
We derive the existence of a quasi-isomorphism of complexes
\begin{displaymath}
    \cone[\Omega_{X/\DBbb}\to\Omega_{X/\DBbb}(\log)]\overset{\sim}{\dashrightarrow}\cone[\Ocal_{X_0}\to\Ocal_{D(1)}].
\end{displaymath}

\begin{proposition}\label{prop:kahlerlogcomp}
Suppose $f\colon X \to \DBbb$ is a projective degeneration between complex manifolds. Write $X_{0,\ \rm{red}} = \sum D_i$, and $\widetilde{D}_i$ for a desingularization of the irreducible components $D_{i}$. Denote by $\pi: \widetilde{X} \to X$  a simple normal crossings model. Then there is a canonical isomorphism
\begin{displaymath}
    \lambda(\Omega_{X/\DBbb})\overset{\sim}{\longrightarrow}\lambda(\Omega_{\widetilde{X}/\DBbb}(\log)) + \left( \chi(\Ocal_{X_\infty}) - \sum \chi(\Ocal_{\widetilde{D}_i}) \right) \Ocal([0])
\end{displaymath}
which induces the identity on the smooth locus.
\end{proposition}
\begin{proof}
The bundle $\lambda(\Omega_{\widetilde{X}/\DBbb}(\log))$ is independent of the specific normal crossings model $\widetilde{X}$, since it is build up from lower extensions of Hodge bundles. We construct one model by applying embedded resolution of singularities to $X_0 \hookrightarrow X$. This means that $\widetilde{X}$ is obtained by a sequence of blow ups in the special fibers, of $X_i\to\DBbb$ say, along regular centers $Z_i$. Denote by $\nu: X_{i+1} \to X_i$ the blowup. In that case $Z_i$ is necessarily regularly embedded in $X_i$ and the exceptional divisor $E_i$ is a projective bundle over $Z_i$. We moreover have the exact sequence
\begin{equation}\label{eq:cottriang1}
    0 \longrightarrow \nu^\ast \Omega_{X_i/\DBbb} \longrightarrow \Omega_{X_{i+1}/\DBbb} \longrightarrow  \Omega_{E_i/Z_i} \longrightarrow 0,
\end{equation}
where the exactness on the left can be justified by a direct computation.

Recalling that $R\nu_\ast \Ocal_{{X_{i+1}}}\simeq \Ocal_{X_i}$, and taking the determinant of the cohomology, we find that \begin{displaymath}
    \lambda(\Omega_{X_{i+1}/\DBbb}) \overset{\sim}{\longrightarrow} \lambda(\Omega_{X_{i}/\DBbb})+\chi( \Omega_{E_i/Z_i}) \cdot \Ocal([0]).
\end{displaymath}
We notice that $\chi(\Omega_{E_{i}/Z_{i}})=-\chi(\Ocal_{E_{i}})$. Indeed, the Euler exact sequence for the cotangent sheaf of a projective bundle readily implies $R^{p}\nu_{\ast}\Ocal_{E_{i}}\simeq R^{p+1}\nu_{\ast}\Omega_{E_{i}/Z_{i}}$ for all $p$. Hence, for corresponding Leray spectral sequences we have
\begin{displaymath}
    \xymatrix{
    &H^{q}(E_{i}, R^{p}\nu_{\ast}\Ocal_{E_{i}})\ar@{=>}[r]\ar@{=}[d]    & H^{q+p}(E_{i},\Ocal_{E_{i}})\\
    &H^{q}(E_{i}, R^{p+1}\nu_{\ast}\Omega_{E_{i}/Z_{i}})\ar@{=>}[r]    &H^{q+p+1}(E_{i},\Omega_{E_{i}/Z_{i}}),
    }
\end{displaymath}
and we conclude by taking Euler characteristics. At the end of the sequence of blow ups, we thus find
\begin{displaymath}
    \lambda(\Omega_{\widetilde{X}/\DBbb})\overset{\sim}{\longrightarrow}\lambda(\Omega_{X/\DBbb}) - \left( \sum \chi(\Ocal_{E_i}) \right) \cdot\Ocal([0]).
\end{displaymath}
On the other hand, from \eqref{eq:diagram-res} there is a quasi-isomorphism
\begin{equation}\label{eq:logres1}
    \cone[\Omega_{\widetilde{X}/\DBbb}\rightarrow \Omega_{\widetilde{X}/\DBbb}(\log)]\
    \overset{\sim}{\dashrightarrow}\cone[\Ocal_{\widetilde{X}_0} \rightarrow \bigoplus \Ocal_{D_i^{\prime}} \oplus \bigoplus \Ocal_{E_i^{\prime}}],
\end{equation}
where the $D_{i}^{\prime}$ (resp. $E_{i}^{\prime}$) are the strict transforms of the $D_{i}$ (res. $E_{i}$). Recall that the Euler characteristics of the form $\chi(\Ocal_{D})$ are birational invariants of complex manifolds, so that $\chi(\Ocal_{D_{i}^{\prime}})=\chi(\Ocal_{\widetilde{D}_{i}})$ and $\chi(\Ocal_{E_{i}^{\prime}})=\chi(\Ocal_{E_{i}})$. We derive
\begin{displaymath}
    \lambda(\Omega_{\widetilde{X}/\DBbb}(\log ))\simeq \lambda(\Omega_{\widetilde{X}/\DBbb}) + \left( \sum \chi(\Ocal_{\widetilde{D}_i})+ \sum \chi(\Ocal_{E_{i}}) - \chi(\Ocal_{\widetilde{X}_0})\right) \cdot \Ocal([0]).
\end{displaymath}
Since the family $\widetilde{X} \to \DBbb$ is flat, we have $\chi(\Ocal_{\widetilde{X}_0}) = \chi(\Ocal_{\widetilde{X}_\infty}) = \chi(\Ocal_{X_\infty}).$ The result follows by composing the isomorphisms deduced from \eqref{eq:cottriang1} and \eqref{eq:logres1}.
\end{proof}

\begin{corollary}\label{cor:kahlerlogcomp}
     Under the hypothesis of Proposition \ref{prop:kahlerlogcomp}, if $X_0$  has at most rational singularities, then
     \begin{displaymath}
        \lambda(\Omega_{X/\DBbb})\overset{\sim}{\longrightarrow}\lambda(\Omega_{\widetilde{X}/\DBbb}(\log)).
    \end{displaymath}
     This in particular holds whenever $n \geq 2$ and $X_0$ admits at most ordinary double point singularities.
\end{corollary}
\begin{proof}
Since $X_{0}$ is connected with at most rational singularities, it is in particular normal irreducible. Hence, with notations as in Proposition \ref{prop:kahlerlogcomp}, we have $X_{0}=D$. We then observe $\chi(\Ocal_{\widetilde{D}})=\chi(\Ocal_{X_{0}})$ by the rational singularities assumption and $\chi(\Ocal_{X_{0}})=\chi(\Ocal_{X_{\infty}})$ by flatness.
\end{proof}

\subsection{Families with at most ordinary double point singularities}\label{sec:hodgeodp}
In this subsection we study the Hodge bundles of morphisms whose singular fibers have at most ordinary double point singularities.

Let $f\colon X\to\DBbb$ be a projective degeneration between complex manifolds, with fibers of dimension $n$ . We assume that the special fiber $X_{0}$ has at most ordinary double point singularities. Hence, on a neighborhood (in $X$) of a singular point of $f^{-1}(0)$, there exist holomorphic coordinates $(z_{0},\ldots, z_{n}, s)$ such that $X$ is locally given by $s=z_{0}^{2}+\ldots+z_{n}^{2}$ and $f$ is the projection to $s$. Let $p_{1},\ldots, p_{r}$ be the set of singular points. Let $\nu\colon\widetilde{X}\to X$ be the blow-up in $X$ of $p_{1},\ldots,p_{r}$ and $g\colon\widetilde{X}\to \DBbb$ the natural morphism. Then the special fiber of $g$ has normal crossings of the form $\widetilde{X}_{0}=Z+2\sum_{i=1}^{r}E_{i}$, where the $E_{i}=\nu^{-1}(p_{i})\simeq\PBbb_{\CBbb}^{n}$ are disjoint and $Z$ is the strict transform of $X_{0}$. Moreover, the intersection $W_{i}:=Z\cap E_{i}\subset E_{i}$ is isomorphic to a smooth quadric in $\PBbb^{n}_{\CBbb}$.
As an application of Corollary \ref{cor:eigenvalues}, we compare the Hodge bundles of such families before and after semi-stable reduction.

\begin{proposition}\label{prop:ODP}
Let $\rho: \DBbb \to \DBbb$, $\rho(t)=t^{\ell}$, denote a base change realizing a semi-stable reduction $h\colon Y \to \DBbb$ of $f\colon X \to \DBbb$. Then there is a natural morphism
\begin{displaymath}
    \rho^\ast R^q g_\ast \Omega^{p}_{\widetilde{X}/\DBbb}(\log) {\longrightarrow} R^q h_\ast \Omega^{p}_{Y/\DBbb} (\log).
\end{displaymath}
If $n$ is odd or $(p,q) \neq (n/2, n/2)$, this is an isomorphism. In the case  $(p,q) = (n/2, n/2)$, the cokernel is isomorphic to
\begin{displaymath}
    \left(\frac{\Ocal_{\DBbb,0}}{t\Ocal_{\DBbb,0}}\right)^{\oplus \#\sing(X_0)}
\end{displaymath}
\end{proposition}

\begin{proof}
The statement can be derived from Corollary \ref{cor:eigenvalues}, by studying  the monodromy operator on the limiting mixed Hodge structure on $H^{n}(X_{\infty})$. It is described by the Picard-Lefschetz theorem. For simplicity, assume that there is a single ordinary double point. If the relative dimension $n$ is odd, then the monodromy acting on any $H^{k}(X_{\infty})$ is unipotent, and hence $T_{s}$ is trivial on it. We now focus on the case when $n$ is even. Since the monodromy is non-trivial at most on $H^{n}(X_{\infty})$, we only need to consider $p+q=n$. We have to show that $T_{s}$ acts trivially on $\Gr^{k}_{F_{\infty}}H^{n}(X_{\infty})$ for $k\neq n$. We use the fact that $T_{s}$ is an endomorphism of mixed Hodge structures on $H^{n}(X_{\infty})$. We have
\begin{displaymath}
    \Gr_{k}^{W}H^{n}(X_{\infty})=0,\quad k\neq n
\end{displaymath}
(cf. for instance \cite[Sec. 3]{wang}). Hence, it is enough to analyze the action of $T_{s}$ on $\Gr_{F_{\infty}}^{\bullet}\Gr_{n}^{W}H^{n}(X_{\infty})$. But $F_{\infty}^{\bullet}$ induces a pure Hodge structure of weight $n$ on $\Gr_{n}^{W}H^{n}(X_{\infty})$. If $T_{s}$ acts non-trivially on $\Gr^{p}_{F_{\infty}}\Gr_{n}^{W}H^{n}(X_{\infty})$, then it acts non-trivially on $\Gr^{q}_{F}\Gr_{n}^{W}H^{n}(X_{\infty})$, $p+q=n$. Since the only non-trivial eigenvalue of $T_{s}$ on $H^{n}(X_{\infty})$ is $-1$, with multiplicity one, then necessarily $p=q=n/2$. This concludes the proof.
\end{proof}


\section{$L^{2}$  metrics }
This section revolves around the $L^{2}$ metrics and their asymptotics for one-parameter degenerations of projective varieties. We provide precise expressions for the dominant and subdominant terms of the asymptotics of the $L^{2}$ metrics on determinants of Hodge bundles. For this, we rely on the general statements in Section \ref{section:monodromy} and Schmid's metric characterisation of the limiting monodromy weight filtration \cite{schmid}. Throughout, we will freely exploit the compatibility of Schmid's limiting mixed Hodge structures with Steenbrink's geometric approach \cite{Steenbrink-limits, Steenbrink-mixedonvanishing, Navarro-Guillen}.  

\subsection{Generalities on $L^{2}$ metrics}\label{subsec:Hodge-metrics}
We recall known facts about $L^{2}$ metrics, mainly to fix notations and normalizations.

Let $X$ be a compact K\"ahler manifold of dimension $n$ and with K\"ahler form $\omega$.
On the De Rham cohomology $H^{k}(X, \CBbb)$
there is a natural $L^2$ metric: given classes $\alpha,\beta$ and harmonic representatives $\tilde\alpha, \tilde\beta$, we put
\begin{displaymath}
    \langle \alpha, \beta \rangle_{L^2} = \int_X \tilde\alpha \wedge \star \overline{\tilde\beta} = \int_X  \langle\tilde\alpha, \tilde\beta\rangle \frac{\omega^n}{n!}.
\end{displaymath}
Here $\star$ denotes the Hodge star operator and the inner product in the second integral is the induced pairing on differential forms, coming from the K\"ahler metric. On 
Dolbeault cohomology $H^{p,q}(X) = H^{q}(X, \Omega_X^p)$ we may similarly define an $L^{2}$ metric. The Hodge decomposition
\begin{displaymath}
    H^{k}(X,\CBbb)=\bigoplus_{p+q=k} H^{p,q}(X)
\end{displaymath}
is then an orthogonal sum decomposition for the $L^{2}$ metrics.

Recall that, for any integer $k \leq n$, we define the primitive cohomology subspace of $H^{k}(X,\CBbb)$ by
\begin{displaymath}
    \Hprim{k}(X,\CBbb) = \ker\left( L^{n-k+1} : H^{k}(X, \CBbb) \to H^{2n-k+2} (X, \CBbb)\right),
\end{displaymath}
where $L=\omega\wedge\bullet$ is the Lefschetz operator. Then, for any $k\geq 0$ we have the Lefschetz decomposition
\begin{displaymath}
    H^{k}(X,\CBbb)=\bigoplus_{k-n\leq 2r\leq k}L^{r}H^{k-2r}_{\prim}(X,\CBbb).
\end{displaymath}
This decomposition is orthogonal for the $L^{2}$ metric. On each piece of the decomposition, the Hodge star operator is given by the rule
\begin{equation}\label{eq:hodge-star-L}
    \star L^r \alpha = (-1)^{(k-2r)(k-2r+1)/2}\frac{r!}{(n-k+r)!} L^{n-k+r} C \alpha,
\end{equation}
where $C$ is the Weil operator acting as multiplication by $i^{p-q}$ on $H^{p,q}(X)$. Therefore, the $L^{2}$ metric on $L^{r}H^{k-2r}_{\prim}(X,\CBbb)$ reads
\begin{equation}\label{eq:hodge-star-L-bis}
    \langle L^{r}\alpha, L^{r}\beta\rangle_{L^{2}}=(-1)^{(k-2r)(k-2r+1)/2}\frac{r!}{(n-k+r)!}\int_{X}\alpha\wedge C\overline{\beta}\wedge\omega^{n-k+2r}.
\end{equation}
One also defines $\Hdprim{p}{q}(X) = \Hprim{p+q}(X, \CBbb) \cap H^{p,q}(X)$. Primitive cohomology groups can be put in families, in the setting of a K\"ahler morphism $f\colon X\to S$. The construction produces holomorphic vector bundles for which we employ similar notations, for instance $(R^{k}f_{\ast}\CBbb)_{\prim}$ or $(R^{q}f_{\ast}\Omega_{X/S}^{p})_{\prim}$. The analogue of the Lefschetz decomposition holds in this generality, as a decomposition of holomorphic vector bundles. For this we notice that the Lefschetz operator induced by the K\"ahler structure is horizontal with respect to the Gauss--Mannin connection, and hence holomorphic.

At a later point we will exploit the integral structure of the cohomology groups. For a $\ZBbb$-module of finite type $\Lambda$, we denote by $\Lambda_{\nt}$ the maximal torsion free quotient of $\Lambda$.
\begin{definition}
Let $(X,\omega)$ be a compact K\"ahler manifold. We define $\vol_{L^{2}}(H^{k}(X,\ZBbb),\omega)$ as the covolume of the lattice $H^{k}(X,\ZBbb)_{\nt}\subset H^{k}(X,\RBbb)$, with respect to the Euclidean structure induced by the $L^{2}$ metric.
\end{definition}

\begin{proposition}\label{prop:rational-volume}
Let $X$ be a complex manifold, and $\omega$ be a K\"ahler form with rational cohomology class. Then $\vol_{L^2}(H^k(X, \ZBbb),\omega)$ is a rational number.
\end{proposition}
\begin{proof}
To prove that the volume is a rational number, it is enough to prove that for a given basis $e_1, \ldots, e_N$ of $H^k(X, \mathbb{Q})$, the expression $\det \langle e_i, e_j \rangle_{L^2} $ is rational. To this end, notice that the Lefschetz decomposition can be defined over $\mathbb{Q}$, since the K\"ahler form is rational:
\begin{displaymath}
    H^k(X, \mathbb{Q}) = \bigoplus_{r} L^{r} H^{k-2r}_{\prim}(X, \mathbb{Q}),\quad H^{k}_{\prim}(X, \mathbb{Q})=\ker\left( L^{n-k+1} : H^{k}(X, \QBbb) \to H^{2n-k+2} (X, \QBbb)\right).
\end{displaymath}
We also know that this decomposition is orthogonal for the $L^{2}$ metric. We can hence restrict ourselves to considering $\det\langle v_{i},v_{j}\rangle_{L^{2}}$ for a basis  $v_{1},\ldots, v_{d}$ of $L^{r} H^{k-2r}_{\prim}(X, \mathbb{Q})$. Recalling now \eqref{eq:hodge-star-L}, we see that $\star L^{r}$ acting on $H^{k-2r}_{\prim}(X, \mathbb{Q})$ can be decomposed as

\begin{equation}\label{eq:hodgestar}
    H^{k-2r}_{\prim}(X, \mathbb{C}) \overset{\text{C}}{\longrightarrow} H^{k-2r}_{\prim}(X, \mathbb{C}) \overset{L^{n-k+r}}{\longrightarrow} H^{2n-k+r}(X, \mathbb{C}) \overset{q}{\longrightarrow} H^{2n-k+r}(X, \mathbb{C}),
\end{equation}
where $q$ is multiplication by a rational number. Furthermore, the determinant of the Weil operator $C$ is 1, since
\begin{displaymath}
    \det C = \prod_{p,q} i^{(p-q)h^{p,q}_{\prim}}
\end{displaymath}
and $h^{p,q}_{\prim}=h^{q,p}_{\prim}$. Hence, since the other operators in \eqref{eq:hodgestar} preserve the rational structure, we have in fact
\begin{displaymath}
    \det \star L^{r} H^{k-2r}_{\prim}(X, \mathbb{Q}) = \det L^{n-k+r} H^{k-2r}_{\prim}(X, \mathbb{Q}).
\end{displaymath}
The pairing $\det\langle v_{i},v_{j}\rangle_{L^{2}}$ hence reduces to the determinant of a matrix of integrals of top degree rational cohomology classes, \emph{i.e.} the determinant of a matrix with rational entries. It is therefore a rational number.
\end{proof}

\subsection{The singularities of $L^{2}$ metrics}
Let $f\colon X\to\DBbb$ be a projective K\"ahler normal crossings degenerations, with $n$-dimensional fibers. We suppose that the K\"ahler form on $X^{\times}$ is rational on fibers, \emph{e.g.} induced by a projective embedding. The Hodge bundles $R^{q}f_{\ast}\Omega_{X/\DBbb}^{p}(\log)$ are then endowed with singular $L^2$ metrics, smooth over $\DBbb^{\times}$. The determinant bundle $\det R^{q}f_{\ast}\Omega_{X/\DBbb}^{p}(\log)$ inherits a singular $L^{2}$ metric. In \cite{Peters-flatness}, building on the work of Schmid \cite{schmid}, Peters considers the asymptotic expansion of the $L^{2}$ metric, including first and second order derivatives. Nevertheless, his expression does not have a clear interpretation in terms of the limiting mixed Hodge strucutre. Our aim is to address this point for the dominant and subdominant terms. 

To formulate our results, consider the limiting mixed Hodge structure $(F^{\bullet}_{\infty}, W_{\bullet})$ on $H^{k}(X_{\infty})$ ($k=p+q$), where $F_\infty^{\bullet}$ is the limiting Hodge structure and $W_{\bullet}$ is the monodromy weight filtration (cf. Schmid \cite{schmid}). We introduce the following invariants: 
\begin{equation}\label{eq:def-alpha-p-q}
    \alpha^{p,q}=-\frac{1}{2\pi i}\tr\left(^{\ell}\log T_{s}\mid \Gr^{p}_{F_{\infty}} H^{k}(X_{\infty})\right)
\end{equation}
and
\begin{equation}\label{def:beta-p-q}
    \beta^{p,q}=\sum_{r=-k}^{k}r\dim\Gr_{F_\infty}^{p}\Gr^{W}_{k+r}H^{k}(X_{\infty}).
\end{equation}
In the course of the proof of Theorem \ref{thm:expansion-hodge} below, we need the following lemma.
\begin{lemma}\label{lemma:beta-p-q}
The invariants $\beta^{p,q}$ satisfy the following two identities:
\begin{enumerate}
    \item $\beta^{p,q}=-\beta^{q,p}$.
    \item $\beta^{p,q}= \beta^{n-q, n-p}$.
\end{enumerate}
In particular,     $\beta^{p,q}=-\beta^{n-p,n-q}$.
\end{lemma}
\begin{proof}
Since $\Gr^{W}_{k+r}H^{k}(X_{\infty})$ is a pure Hodge structure of weight $k+r$ we have 
\begin{displaymath} 
    \dim \Gr_{F_\infty}^{p} \Gr^{W}_{k+r}H^{k}(X_{\infty})  = \dim \Gr_{F_\infty}^{k+r-p}\Gr^{W}_{k+r}H^{k}(X_{\infty}).
\end{displaymath} Moreover, the limiting mixed Hodge structure is equipped with a nilpotent operator $N$ on $H^{k}(X_{\infty})$ which induces an isomorphism
$$
    N^r \colon \Gr_{F_\infty}^{p}\Gr^{W}_{k+r}H^{k}(X_{\infty}) \overset{\sim}{\to}     \Gr_{F_\infty}^{p-r}\Gr^{W}_{k-r}H^{k}(X_{\infty}).
$$
 A direct combination of these observations proves the first point. 

The second point follows from the Lefschetz isomorphism, which induces isomorphisms (cf. Schmid \cite[Theorem 6.16]{schmid})
\begin{displaymath}
    L^{n-k}\colon \Gr_{F_\infty}^{p}\Gr^{W}_{k+r}H^{k}(X_{\infty})\overset{\sim}{\to} \Gr_{F_\infty}^{p+n-k}\Gr^{W}_{r+2n-k}H^{2n-k}(X_{\infty}).
\end{displaymath}
\end{proof}
\begin{theorem}\label{thm:expansion-hodge}
Let $\sigma$ be a holomorphic trivialization of $\det R^{q}f_{\ast}\Omega_{X/\DBbb}^{p}(\log)$. Then we have a real analytic asymptotic expansion for its $L^{2}$ norm
\begin{displaymath}
    h_{L^{2}}(\sigma(t),\sigma(t))=|t|^{2\alpha^{p,q}}g(t)\sum_{j=0}^{\infty}c_{j}\ (\log|t|^{-1})^{\beta^{p,q}-j},
\end{displaymath}
where: 1)  $g$ is real analytic in a neighborhood of $0$, with $g(0)\neq 0$ and 2) $c_{j}\in\RBbb$ and $c_{0}\neq 0$. In particular, if $T_{s}$ acts trivially on $\Gr^{p}_{F_{\infty}} H^{p+q}(X_{\infty})$, then the $L^{2}$ metric on $\det R^{q}f_{\ast}\Omega_{X/\DBbb}^{p}(\log)$ has at worst a logarithmic singularity at the origin, and is good in the sense of Mumford.
\end{theorem}

\begin{remark}
    As direct consequence of the theorem, we see that 
    \begin{displaymath}
        \log h_{L^2} = \alpha^{p,q} \log|t|^2 + \beta^{p,q} \log \log|t|^{-1}+ C + O\left(\frac{1}{\log|t|}\right),
    \end{displaymath}
    for a constant $C$ and where $\beta^{p,q} \in [-k h^{p,q}, k h^{p,q}]$. This generalizes Theorem A of \cite{cdg}, which was announced for $(p,q)=(n,0)$ and under the hypothesis that $h^{n,0}=1$.
\end{remark}

\begin{proof}[Proof of Theorem \ref{thm:expansion-hodge}]
First of all, we reduce to the semi-stable case: by Corollary \ref{cor:eigenvalues}, the change of the $L^{2}$ metric under semi-stable reduction is accounted for by the term $|t|^{2 \alpha^{p,q}}$.

For the rest of the argument, we can hence assume that $f \colon X \to \DBbb$ is semi-stable, so in particular the monodromy operator on $H^{k}(X_{\infty})$ is unipotent. Then a real analytic expansion of the form
\begin{displaymath}
    h_{L^{2}}(\sigma(t),\sigma(t))=g(t)\sum_{j=0}^{\infty}c_{j}\ (\log|t|^{-1})^{\beta-j}
\end{displaymath}
and satisfying properties 1) and 2) follows from \cite[Prop. 2.2.1]{Peters-flatness}, since the canonical isomorphism
\begin{displaymath}
    \det R^{q}f_{\ast}\Omega_{X/\DBbb}^{p}(\log)\simeq\det\Fcal_{\log}^{p}\otimes (\det\Fcal_{\log}^{p+1})^{\vee}
\end{displaymath}
is an isometry for the $L^{2}$ metrics. We notice that \emph{loc. cit.} is formulated within the polarized setting, \emph{i.e.} it initially applies to the primitive counterparts of the Hodge bundles $\Fcal^{p}_{\log}$. In general, one can reduce to the polarized case via the Lefschetz decomposition, since the Lefschetz operator is an isometry for the $L^{2}$ metrics. The next step will be to show that $\beta=\beta^{p,q}$. 

We claim it is enough to establish the weaker inequality $\beta \leq \beta^{p,q}$, namely
\begin{equation} \label{pq-estimate}
    h_{L^2}(\sigma(t), \sigma(t)) = O\left((\log|t|)^{\beta^{p,q}}\right).
\end{equation}
Indeed, suppose this estimate is satisfied for all $p,q$. Let $\sigma'$ be a trivializing section of \linebreak $\det R^{n-q} f_\ast \Omega_{X/\DBbb}^{n-p}(\log)$. Then we have the estimate 
\begin{equation}\label{pq-estimateI}
    h_{L^2}(\sigma'(t), \sigma'(t)) = O\left((\log|t|)^{\beta^{n-p,n-q}}\right) = O\left((\log|t|)^{-\beta^{p,q}}\right),
\end{equation}
where we applied Lemma \ref{lemma:beta-p-q}. These norms can be compared, since Serre duality induces an $L^2$ isometry
\begin{equation}\label{pq-estimateII}
    \det R^{q}f_{\ast}\Omega_{X/\DBbb}^{p}(\log) \simeq \det R^{n-q} f_\ast \Omega_{X/\DBbb}^{n-p}(\log)^\vee.
\end{equation}
Here we implicitly used that $\det \Omega_{X/\DBbb}(\log) = \omega_{X/\DBbb}$, since we are in the semi-stable setting. Combining \eqref{pq-estimateI} and \eqref{pq-estimateII} we find the reverse inequality $\beta_{p,q} \leq \beta$, and we thus conclude by \eqref{pq-estimate}.

We now proceed to prove \eqref{pq-estimate}. For the discussion, we rely on Section \ref{section:monodromy}. First of all, we fix a subset $E = \{\ebold_1, \ldots, \ebold_h\}$ of  $F^p_\infty H^k(X_\infty)$ which projects to a basis in $\Gr_{F_\infty}^p H^k(X_\infty)$ and which is adapted to the weight type filtration 
\begin{displaymath}
    W_\ell \Gr_{F_\infty}^p H^k(X_\infty) = W_\ell \cap F_\infty^p /W_\ell \cap F_\infty^{p+1}.
\end{displaymath}
In other words, for each $\ell$, there is a subset $E_\ell$ of $E$ such that the elements of $E_\ell$ are in \linebreak $W_\ell \cap F_\infty^p  \setminus W_{\ell-1}\cap F_\infty^p$ and project to a basis of $\Gr_\ell^W \Gr_{F_\infty}^p H^k(X_\infty)$. 

Secondly, lift the elements $E$ to holomorphic sections $\lbrace\widetilde{\sigma}_{j}\rbrace$ of $\Fcal^p_{\log}$. They project to a local holomorphic frame $\lbrace\sigma_{j}\rbrace$ of $\Fcal^{p}_{\log}/\Fcal^{p+1}_{\log}=R^{q}f_{\ast}\Omega^{p}_{X/\DBbb}(\log)$, by Nakayama's lemma. Also, introduce the twisted sections $\widetilde{\ebold}_j(t)= e^{-2\pi i N \tau} \ebold_j(\tau)$, for $\tau\in\HBbb$ (cf. \eqref{eq:twisted}). These can be identified with holomorphic sections of $R^{k}f_{\ast}\Omega^{\bullet}_{X/\DBbb}(\log)$. Under this identification, we have the equality $\widetilde{\sigma}_{j}(0)=\widetilde{\ebold}_{j}(0)$ (cf. \eqref{eq:psi}, \eqref{eq:steenbrink-psi} and \eqref{eq:Finfty}). Therefore,
\begin{displaymath}
    \widetilde{\sigma}_{j}(t)-\widetilde{\ebold}_{j}(t)\in t\cdot \Gamma(\DBbb, R^{k}f_{\ast}\Omega^{\bullet}_{X/\DBbb}(\log)).
\end{displaymath}
Together with \cite[Chap. II, Prop. 25]{griffiths-topics}, we derive for the $L^{2}$ norm
\begin{displaymath}
    \|\widetilde{\sigma}_{j}(t)-\widetilde{\ebold}_{j}(t)\|_{L^{2}}^{2}=O\left(|t|(\log|t|)^{b}\right),
\end{displaymath}
for some integer $b$. By Schmid's theorem \cite[Thm. 6.6]{schmid}, adapted to the present setting by Zucker \cite[Prop. 3.9]{zucker-L2}, we derive
\begin{equation}\label{eq:normsigmaI}
    \|\widetilde{\sigma}_{j}(t)\|_{L^{2}}^{2}=O\left( (\log|t|)^{\ell-k}\right),
\end{equation}
where $\ell$ is such that $\ebold_{j}\in E_{\ell}$. Since the $L^2$ norm on $\Fcal^{p}_{\log}/\Fcal^{p+1}_{\log}$ is the quotient norm of the $L^2$ norm on $\Fcal^p_{\log}$, we have:
\begin{equation}\label{eq:normsigmaII}
    \|\sigma_{1}\wedge\ldots\wedge\sigma_{h}\|^{2}_{L^{2}}(t)
    \leq\prod_j \|\sigma_{j}\|_{L^{2}}^{2}(t)
    \leq\prod_j \|\widetilde{\sigma}_{j}\|_{L^{2}}^{2}(t).
\end{equation}
Combining \eqref{eq:normsigmaI} and \eqref{eq:normsigmaII}, together with $\Gr_\ell^W \Gr_{F_\infty}^p H^k(X_\infty)\simeq \Gr_{F_\infty}^p \Gr_\ell^W H^k(X_\infty)$, we conclude the claimed estimate \eqref{pq-estimate}
\end{proof}
\begin{remark}
As for Theorem \ref{thm:eigenvalues} and Corollary \ref{cor:eigenvalues}, there is a counterpart of Theorem \ref{thm:expansion-hodge} for more general degenerations of Hodge structures, and in particular for the upper extension $^{u}\Fcal^{\bullet}$ (and their graded quotients) in the geometric case. In this situation, only the exponent $\alpha^{p,q}$ needs to be changed to the corresponding elementary exponent for the upper extensions. The rest of the asymptotic expansion remains the same, as it is determined after semi-stable reduction.
\end{remark}

\section{BCOV metrics and invariant for Calabi--Yau varieties}\label{sec:BCOV-bundle}
In this section we recall the construction of the BCOV bundle following Fang--Lu--Yoshikawa \cite{FLY}. It is named after Bershadsky--Cecotti--Ooguri--Vafa \cite{bcov}, who developed a mostly conjectural technique for computing ``higher loop string amplitudes''. For a K\"ahler family of Calabi--Yau manifolds, the bundle can be endowed with a Quillen type metric, independent of the particular choice of K\"ahler structure. In relative dimension 3, Fang--Lu--Yoshikawa could extract from this Quillen metric an important invariant of Calabi--Yau $3$-folds, called the BCOV invariant. This is a suitable normalization of a combination of holomorphic analytic torsions. It is actually this invariant, rather than the original quantity in \cite{bcov}, that is expected to fulfill the predictions in \emph{loc. cit.} in connection with mirror symmetry. The case of the mirror quintic family was successfully solved in \cite{FLY}. The analogous of this conjectural program in dimension 4 has been proposed by Klemm-Pandharipande \cite{Klemm-Pandharipande}, and further studied in dimension 5 by Pandharipande--Zinger \cite{PandiZinger}. However a right counterpart of the BCOV invariant in arbitrary dimension, independent of the K\"ahler structure, was still missing. Filling this gap is the ultimate purpose of this section.

In this section, we wish to distinguish the dualizing (or canonical) sheaf, from K\"ahler forms. We herein adopt the notation $K_X$ for the canonical sheaf of a complex manifold $X$, and similarly for the relative setting.

\subsection{The BCOV line bundle and its Quillen-BCOV metric}\label{subsec:Quillen-BCOV}
Let $f\colon X \to S$ be a K\"ahler morphism whose fibers are connected Calabi--Yau manifolds of dimension $n$. From the K\"ahler structure, the relative cotangent bundle $\Omega_{X/S}$ inherits an induced smooth hermitian metric.
\begin{definition}
The BCOV bundle of the family $f\colon X\to S$ is the line bundle on $S$
\begin{displaymath}
    \lambda_{BCOV}(X/S):=\bigotimes_{p}  \lambda(\Omega^p_{X/S})^{(-1)^p p}
\end{displaymath}
\end{definition}
From the choice of K\"ahler structure and the induced metrics on the powers $\Omega_{X/S}^{p}$, the line bundle $\lambda_{BCOV}$ carries a $L^{2}$ and a Quillen type metric, that we denote $h_{L^{2}}$ and $h_{Q}$, respectively. Let us momentarily assume that $S$ is reduced to a point. Hence we deal with a single Calabi--Yau manifold $X$ endowed with a K\"ahler hermitian metric $h_{X}$ on $T_{X}$. As it is customary in the Quillen metric literature (see \emph{e.g.} \cite[Sec. 4]{Gillet-Soule:ARR}), we work with the normalized K\"ahler form
\begin{displaymath}
    \omega=\frac{i}{2\pi}\sum_{j,k}h_{X}\left(\frac{\partial}{\partial z_j},\frac{\partial}{\partial z_k}\right)dz_{j}\wedge d\ov{z}_{k}.
\end{displaymath}
Henceforth, $L^{2}$ metrics and volumes will be computed with respect to this normalized K\"ahler form. For instance, the volume of $X$ is
\begin{displaymath}
    \vol(X,\omega)=\int_{X}\frac{\omega^{n}}{n!}.
\end{displaymath}
Following \cite[Sec. 4]{FLY} we put
\begin{equation*}\label{eq:A-constant}
    A(X,\omega)=\vol(X,\omega)^{\frac{\chi(X)}{12}}\exp\left[-\frac{1}{12}\int_{X}\log\left(\frac{i^{n^{2}}\eta\wedge\ov{\eta}}{\omega^{n}/n!}\frac{\vol(X,\omega)}{\|\eta\|_{L^{2}}^{2}}\right)c_{n}(X,\omega)\right],
\end{equation*}
where $\eta$ is a holomorphic trivialization of $\Omega_{X}^{n}$ and $c_{n}(X,\omega)$ is the Chern--Weil representative of $c_{n}(T_{X})$ associated to the K\"ahler metric. For a Ricci-flat K\"ahler metric, the factor $A(X,\omega)$ simplifies to
\begin{displaymath}
    A(X,\omega)=\vol(X,\omega)^{\frac{\chi(X)}{12}}.
\end{displaymath}
In the general family setting $f\colon X\to S$, we denote by $\omega_{s}$ the normalized K\"ahler form on the fibers $X_{s}$. Then the function $s\mapsto A(X_{s},\omega_{s})$ is clearly smooth, and will be denoted $A(X/S,\omega)$.
\begin{definition}\label{def:Q-bcov}
The Quillen-BCOV metric on $\lambda_{BCOV}$ is the smooth hermitian metric
\begin{displaymath}
    h_{Q,BCOV}:=A(X/S,\omega)\cdot h_{Q}.
\end{displaymath}
We refer to the pair $(\lambda_{BCOV}, h_{Q,BCOV})$ as the BCOV hermitian line bundle.
\end{definition}

Recall the definition of the Weil--Petersson form $\omega_{WP}=c_{1}(f_{\ast}K_{X/S},h_{L^{2}})$ (cf. \cite{tian}). The curvature of the BCOV hermitian line bundle was computed in \cite[Thm. 4.9]{FLY}:
\begin{proposition}\label{prop:Q-independent}
The curvature of $(\lambda_{BCOV}, h_{Q,BCOV})$ is given by $\frac{\chi(X_{\infty})}{12} \omega_{WP}$. In particular the BCOV hermitian line bundle is independent of the choice of K\"ahler structure.
\end{proposition}

\subsection{The $L^2$-BCOV metric}\label{subsec:L2-BCOV}
In this section, we work with general complex manifolds, not necessarily of Calabi--Yau type. We define a renormalized $L^2$ norm on the BCOV bundle. It has the feature that it is independent of the choice of K\"ahler form  (cf. Proposition \ref{prop:L2-independent}).

\begin{definition}\label{def:L2-bcov}
\begin{enumerate}
    \item Let $(X,\omega)$ be a compact complex K\"ahler manifold of dimension $n$. We define
    \begin{equation}\label{eq:def-B}
        B(X,\omega):=\prod_{k=1}^{2n} \vol_{L^2}(H^k(X, \ZBbb),\omega)^{(-1)^{k+1}k/2 }.
    \end{equation}
    Here we adopt the convention that $\vol_{L^2}(H^k(X, \mathbb{Z}),\omega)=1$ if $b_k(X)=0$.
    \item Let $f\colon X \to S$ be a K\"ahler morphism with K\"ahler structure form $\omega$, whose fibers are compact complex manifolds of dimension $n$. We define a function $B(X/S,\omega)\in\Ccal^{\infty}(S)$ by $B(X/S,\omega)(s)=B(X_{s},\omega_{s})$.
    \item The $L^2$-BCOV metric, or rescaled $L^{2}$ metric, on $\lambda_{BCOV}$ is the $\mathcal{C}^\infty$ metric
    \begin{displaymath}
        h_{L^2,BCOV} = B(X/S, \omega)\cdot h_{L^2}.
    \end{displaymath}
\end{enumerate}
\end{definition}

\begin{remark}
\begin{enumerate}[(i)]
    \item The definition drives some inspiration from Kato's formalism of heights of motives \cite{Katoheight}, see specially paragraph \textbf{1.3} in \emph{loc. cit.}
    \item Since Poincar\'e duality is a unimodular pairing, it is not difficult to prove that 
\begin{displaymath} 
    \vol_{L^2}(H^k(X, \mathbb{Z}),\omega) \vol_{L^2}(H^{2n-k}(X, \mathbb{Z}),\omega)= 1.
\end{displaymath} 
    Hence the product in the definition of $B(X,\omega)$ can be written more succinctly as
\begin{equation*}\label{eq:Bnfold}
    B(X, \omega) = \prod_{k=0}^{n-1} \vol_{L^2}(H^k(X, \mathbb{Z}),\omega)^{(-1)^{k}(n-k)}.
\end{equation*}
In particular for a (simply connected) Calabi--Yau 3-fold $X$ with a K\"ahler form $\omega$, we find that
\begin{equation}\label{eq:B3fold}
    B(X, \omega) = \vol_{L^2}(H^0(X, \mathbb{Z}),\omega)^{3} \vol_{L^2}(H^2(X, \mathbb{Z}),\omega).
\end{equation}
\end{enumerate}
\end{remark}

\begin{proposition}\label{prop:L2-independent}
The $L^2$-BCOV metric  $h_{L^2,BCOV}$ is independent of the K\"ahler structure.
\end{proposition}

\begin{proof}
We can check the statement pointwise, and hence we may work with a single compact K\"ahler manifold $(X,\omega)$. For each $k$, inspired by an identity from \cite[1.3] {Katoheight} we define $$L(H^k) =\sum_{p+q=k} p \det H^{q}(X, \Omega_X^p),$$ where as usual we adopt additive notations for tensor products. Consider also the complex conjugate line $$\overline{L(H^k)} = \sum_{p+q=k} p \overline{\det H^{q}(X, \Omega_X^p)} =  \sum_{p+q=k} (k-p) {\det H^{q}(X, \Omega_X^{p})}.$$ The $L^2$ metric induces a metric on $L(H^k)$, as well as $\overline{L(H^k)}$. Since complex conjugation \linebreak $L(H^k) \to \overline{L(H^k)}$ is an isometry (as real vector spaces), we have an $L^2$ isometry
\begin{displaymath}
    2 L(H^k) = L(H^k) + \overline{L(H^k)} = \sum_{p+q=k} k \det H^p(X, \Omega^q) = k \det H^k(X, \mathbb{C}).
\end{displaymath}
Up to sign, the line $\det H^k(X, \mathbb{C})$ has a natural element determined by the integral structure, namely $e_1 \wedge \ldots \wedge e_N$ for a basis $e_1,\ldots, e_N$ of $H^k(X, \mathbb{Z})_{\nt}$. Dividing by the norm of this section we find that the right hand side, and hence the left hand side, don't depend on the K\"ahler structure. More precisely the $L^2$ metric on $L(H^k)$ multiplied by $\vol_{L^2}(H^k(X, \mathbb{Z}), \omega)^{-k/2}$ is independent of the choice of K\"ahler form. Since the BCOV bundle is clearly given by
\begin{displaymath}
    \sum_{0\leq k\leq 2n} (-1)^k L(H^k),
\end{displaymath}
the proposition follows.
\end{proof}

\subsection{The BCOV invariant for Calabi--Yau $n$-folds}\label{subsec:BCOV-invariant}

Let now $X$ be a compact Calabi--Yau $n$-fold. Let $\omega$ be a K\"ahler metric. The vector bundles $\Omega_{X}^{p}$ inherit a smooth hermitian metric. The ``virtual'' vector bundle
\begin{displaymath}
    \bigoplus_{p}(-1)^{p}p\Omega_{X}^{p}
\end{displaymath}
has a well-defined holomorphic analytic torsion depending on $\omega$ and written $\Tcal_{BCOV}(X,\omega)$. It also carries the metrics $h_{Q,BCOV}$ and $h_{L^{2},BCOV}$.
\begin{definition}\label{deff:bcovinv} Let $X$ be a Calabi--Yau $n$-fold. The BCOV invariant of $X$ is the real number given by
\begin{displaymath}
    \tau_{BCOV}(X) = h_{Q, BCOV}/h_{L^2, BCOV}
\end{displaymath}
In other words, for any auxiliary K\"ahler form $\omega$,
\begin{equation*}\label{def:bcovinv}
    \tau_{BCOV}(X) = \frac{A(X, \omega)}{B(X, \omega)} \mathcal{T}_{BCOV}(X, \omega),
\end{equation*}
where $A(X, \omega)$ and $B(X, \omega)$ are as in definitons \ref{def:Q-bcov} and \ref{def:L2-bcov}.
\end{definition}
The terminology invariant is justified by the following proposition.
\begin{proposition}
The BCOV invariant $\tau_{BCOV}$ depends only on the complex structure of $X$.
\end{proposition}
\begin{proof}
This is the combination of Proposition \ref{prop:Q-independent} and Proposition \ref{prop:L2-independent}.
\end{proof}
Our definition of the BCOV invariant for Calabi--Yau $n$-folds is an $n$-fold generalization of the BCOV invariant defined by Fang--Lu--Yoshikawa for Calabi--Yau $3$-folds \cite[Definition 4.13]{FLY}. Indeed, their invariant is given by
\begin{displaymath}
    \vol(X, \omega)^{-3} \vol_{L^2}(H^2(X, \mathbb{Z}),\omega)^{-1} A(X, \omega) \mathcal{T}_{BCOV}(X, \omega),
\end{displaymath}
which by \eqref{eq:B3fold} coincides with our construction, since by definition $\vol(X, \omega) = \vol_{L^2}(H^0(X, \mathbb{Z}), \omega)$.

Let us now discuss the differential equation satisfied by the BCOV invariant for families. Let $f\colon X\to S$ be a K\"ahler morphism between complex manifolds, whose fibers are $n$-dimensional Calabi--Yau varieties. Then $s\mapsto \log\tau_{BCOV}(X_{s})$ yields a smooth function on the space of parameters. Endow the Hodge bundles of $f$ with the associated $L^{2}$ metrics. After Fang--Lu \cite{FangLu}, we define the $k$-th Hodge form on $S$ as the following combination of Chern--Weil forms:
\begin{equation}\label{eq:def-hodge-form-1}
    \omega_{H^{k}}=\sum_{p=0}^{k} c_{1}(\Fcal^{p}R^{k}f_{\ast}\Omega_{X/S}^{\bullet}, h_{L^{2}}).
\end{equation}
By Griffiths' computation of the curvature of Hodge bundles, this is known to be a semi-positive $(1,1)$ form on $S$, cf. \emph{loc. cit.}. It can equivalently be written
\begin{equation}\label{eq:def-hodge-form-2}
    \omega_{H^{k}}=\sum_{p=0}^{k}p c_{1}(R^{k-p}f_{\ast}\Omega^{p}_{X/S}, h_{L^{2}})\in c_1(L(H^k)).
\end{equation}

\begin{proposition}\label{prop:ddc-log-B}
Let $f\colon X\to S$ be a K\"ahler morphism between connected complex manifolds, whose fibers are Calabi--Yau manifolds of dimension $n$. There is an equality of differential forms on $S$
\begin{equation}\label{eq:ddc-log-B}
    \begin{split}
        dd^{c}\log\tau_{BCOV}&=\sum_{k=0}^{2n}(-1)^{k}\omega_{H^{k}}-\frac{\chi}{12}\omega_{WP}
    \end{split}
\end{equation}
where $\chi$ is the topological Euler characteristic of any fiber of $f$. In particular, the alternating sum of Hodge forms does not depend on the particular choice of K\"ahler structure.
\end{proposition}
\begin{proof}
By the curvature formula for the Quillen-BCOV metric \cite[Thm. 4.9]{FLY}, we just need to prove that $dd^{c}\log B(X/S,\omega)=0$, where $\omega$ gives the K\"ahler structure. Equivalently, according to \eqref{eq:def-B}, we need to show that
\begin{displaymath}
    \sum_{k=0}^{2n}(-1)^{k}k\ c_{1}(R^{k}f_{\ast}\CBbb,h_{L^{2}})=0.
\end{displaymath}
First we recall that the fiberwise primitive decomposition provides an identity of local systems
\begin{equation}\label{eq:rel-dec-prim}
    R^{k}f_{\ast}\CBbb=\bigoplus_{2k-2n\leq 2r\leq k}L^{r}(R^{k-2r}f_{\ast}\CBbb)_{\prim},
\end{equation}
where $L$ is the relative Lefschetz operator induced by the K\"ahler structure. The direct sum is orthogonal with respect to the $L^{2}$ metrics. Furthermore, we know that for $k\leq n$, $L^{n-k}$ induces an isomorphism of local systems
\begin{displaymath}
    L^{n-k}\colon R^{k}f_{\ast}\CBbb\overset{\sim}{\longrightarrow} R^{2n-k}f_{\ast}\CBbb.
\end{displaymath}
On each piece of the primitive decomposition \eqref{eq:rel-dec-prim}, the operator $L^{n-k}$ acts as an isometry up to constant, as equation \eqref{eq:hodge-star-L-bis} reveals. We infer from these remarks that
\begin{displaymath}
    c_{1}(R^{k}f_{\ast}\CBbb,h_{L^{2}})=c_{1}(R^{2n-k}f_{\ast}\CBbb,h_{L^{2}}).
\end{displaymath}
On the other hand, by Poincar\'e duality, $(R^{k}f_{\ast}\CBbb,h_{L^{2}})$ and $((R^{2n-k}f_{\ast}\CBbb)^\vee,h^\vee_{L^{2}})$ are isometric.
This concludes the proof.

\end{proof}
\begin{remark}
\begin{enumerate}[(i)]
    \item If the K\"ahler structure is fiberwise rational, then the function $B(X/S,\omega)$ is constant on $S$ by Proposition \ref{prop:rational-volume}. The existence of such K\"ahler forms is, locally over $S$, guaranteed for projective morphisms. However Proposition \ref{prop:ddc-log-B} holds in the generality of K\"ahler fibrations.
    \item In the physics literature, equation \eqref{eq:ddc-log-B} is referred to as the \emph{holomorphic anomaly equation of $F_{1}$}, see \cite[Eq. (3.10)]{bcov}. In relative dimension 3, it is part of an infinite system of differential equations relating some partition functions $F_{g}$, where $g\geq 0$ runs over all the possible genuses of compact Riemann surfaces.
\end{enumerate}
\end{remark}

\subsection{Triviality of the BCOV invariant for special geometries}

Recall the Beauville-Bomogolov classification, that any Calabi--Yau variety is an \'etale quotient of a product $T \times V \times H$, where $T$ is an abelian variety, $V$ is a strict Calabi--Yau variety, and $H$ is a hyperk\"ahler variety. Recall that the later means a K\"ahler manifold whose $H^{2,0}$ is spanned by a holomorphic symplectic form. The facts we need regarding hyperk\"ahler manifolds can be found in \cite{Hitchin}.

We are grateful to Ken--Ichi Yoshikawa for sharing with us his argument for the triviality of the BCOV torsion for hyperk\"ahler variety equipped with a Ricci flat metric.

\begin{proposition}\label{prop:ken-ichi}
If $f\colon X \to S$ is a projective morphism of complex analytic spaces, whose fibers are either abelian varities of dimension at least 2 or hyperk\"ahler varieties, then the function $s\mapsto\taubcov{X_s}$ is locally constant on $S$.
\end{proposition}

\begin{proof}
Since the statement is local over the base and only depends on $S_{\rm{red}}$, we can first desingularize $S$ and  further assume that $S$ is a polydisc. By Proposition \ref{prop:rational-volume} it suffices to show that the BCOV torsion $s\mapsto\Tcal_{BCOV}(X_{s},\omega_{s})$ is constant if we compute it with respect to a Ricci flat metric $\omega$ in the K\"ahler class providing a projective embedding. Even more, we claim that $\Tcal_{BCOV}(X_{s},\omega_{s})=1$. We may thus assume that $S$ is reduced to a point, and work with a single variety $X$. The case of abelian varieties is actually well-known, and we refer for instance to the remark in \cite[p. 154]{Berth}.


Now for the hyperk\"ahler case. Let $n$ be the dimension of $X$. Recall that $n$ is necessarily even. The BCOV torsion decomposes as a product,
\begin{displaymath}
    \Tcal_{BCOV}(X,\omega)=\prod_{p=1}^{n}T(\Omega_{X}^{p},\omega)^{(-1)^{p}p},
\end{displaymath}
where $T(\Omega_{X}^{p},\omega)$ is the holomorphic analytic torsion of $\Omega_{X}^{p}$ endowed with the hermitian metric induced by $\omega$, and computed with respect to the K\"ahler structure $\omega$. The holomorphic analytic torsion of a hermitian vector bundle $E$ satisfies (see for instance \cite[Thm. 1.4]{GilletSouleTodd})
\begin{equation}\label{eq:torsion-serre}
     T(E, \omega) = T(E^\vee \otimes K_X, \omega)^{(-1)^{n+1}},
\end{equation}
where $K_X$ is equipped with the metric coming from the Ricci flat metric on $T_X$. For the Ricci-flat metric, there is an isometry $K_X\simeq\Ocal_{X}$, where the latter is equipped with the trivial metric. Indeed, if $\eta$ is a holomorphic symplectic 2 form, then $\eta^{n/2}$ is a holomorphic trivialization of $K_{X}$. Furthermore, Ricci flatness implies that the pointwise norm of $\eta^{n/2}$ is constant. A suitable rescaling of $\eta$ thus provides the claimed isometry. From \eqref{eq:torsion-serre} we then infer $T(E, \omega) = T(E^\vee, \omega)^{-1}$. Therefore, to conclude, it will be enough to prove that the symplectic holomorphic form induces a holomorphic isometry $\Omega^p_X \simeq \wedge^p T_X \simeq (\Omega_X^{p})^\vee$. This is the content of the following lemma, probably well-known to the specialists, which we state separately.
\end{proof}

\begin{lemma}
Let $X$ be a hyperk\"ahler manifold, with non-degenerate holomorphic symplectic form $\eta$. Then the induced isomorphism  $\wedge^p T_X \simeq \Omega_X^p$ is, up to a constant,  an isometry.
\end{lemma}
\begin{proof}
Since $X$ is hyperk\"ahler, the Ricci flat metric defines a K\"ahler form with respect to three orthogonal parallel complex structures, $I, J, K$, satisfying $I^2 = J^2 = K^2 = IJK = -1$. After renormalization by a scalar, it is possible to write
\begin{equation} \label{eq:etadefn}
    \eta(\bullet, \bullet)= \frac{g(J\bullet, \bullet) + i g(K\bullet, \bullet)}{2}.
\end{equation}
We prove that contracting $\eta$ by a holomorphic tangent vector provides an isometry $T_X \simeq \Omega_X$, where the holomorphic tangent and cotangent bundles are given by the  complex structure $I$. The general statement follows. If $v \in T_X$, then $I v = i v$, and so $v = \xi - i I \xi$ where $\xi = \Real{v}.$ A direct computation shows that $\|v \|^2 = 2\|\xi \|^2$. On the other hand, by \eqref{eq:etadefn} one can conclude $\eta(v,\bullet) =  g(J \xi + i K \xi, \bullet)$. By definition, $\|g(x, \bullet)\|^2 = \langle g(x, \bullet ), \overline{g(x, \bullet)} \rangle =  \langle x, \overline{x} \rangle  $, and hence
\begin{displaymath}
    \|\eta(v)\|^2 =\langle J \xi + i K \xi , J \xi - i K \xi  \rangle = \langle J\xi, J\xi \rangle + \langle  K \xi, K \xi \rangle =  2 \| \xi\|^2=\|v\|^{2}.
\end{displaymath}
Here we have used that $J$ and $K$ preserve the metric. This concludes the proof.
\end{proof}
\begin{remark}\label{rem:bcovk3ab}
 If $X$ an abelian variety of dimension at least 2 or a $K3$ surface, it is not difficult to prove that $\taubcov{X}$ is in fact equal to 1. We ignore whether this continues to be true for higher dimensional hyperk\"ahler manifolds. The difficulty lies in having enough relations between the cohomology groups to conclude that $A(X, \omega)=B(X,\omega)$ for Ricci flat metrics. 
\end{remark}

\section{General asymptotics of the BCOV invariant}
In this section we investigate the singular behaviour of the BCOV invariant along one-parameter degenerations of Calabi--Yau manifolds. 

After initially comparing, in the first subsection, the various extensions of the BCOV bundle, we go on to establish the logarithmic behavior of the BCOV invariant along one-parameter degenerations. There we provide a closed formula for general normal crossings projective degenerations of Calabi--Yau varieties (cf. Theorem \ref{thm:general-bcov}). This can be recast as providing the boundary conditions of the holomorphic anomaly equation (cf. Proposition \ref{prop:ddc-log-B}), and it has proven to be key to the proof of the known cases of the BCOV conjecture, cf. \cite{FLY}. We then proceed to determine the subdominant term of the BCOV invariant in terms of limiting mixed Hodge structures (cf. Proposition \ref{prop:loglogbcov}). These statements will be the point of departure for the computations for special geometries in Section \ref{sec:specialgeom}.

\subsection{K\"ahler and logarithmic extensions of the BCOV bundle}
Let $f\colon X\to\DBbb$ be a projective degeneration between complex manifolds. The BCOV line bundle $\lambda_{BCOV}(X^{\times}/\DBbb^{\times})$ affords a natural extension to $\DBbb$, the so-called K\"ahler extension:
\begin{definition}
The K\"ahler extension of the BCOV bundle $\lambda(X^{\times}/\DBbb^{\times})$ is defined as
\begin{displaymath}
    \widetilde{\lambda}_{BCOV}=\bigotimes_{p}\lambda(\widetilde{\Omega}_{X/\DBbb}^{p})^{(-1)^{p}p},
\end{displaymath}
where we recall that $\lambda(\widetilde{\Omega}_{X/\DBbb}^{p})$ is the K\"ahler extension of $\lambda(\Omega_{X^{\times}/\DBbb^{\times}}^{p})$, cf. Definition \ref{def:kahler-extension}.
\end{definition}
If the singular fiber $X_{0}$ has normal crossings, then there is another natural extension: the logarithmic extension.
\begin{definition}
Let $f\colon X\to\DBbb$ be a projective normal crossings degeneration. The logarithmic extension of the BCOV bundle is defined as
\begin{displaymath}
    \lambda_{BCOV}(\log)=\bigotimes_{p}\lambda(\Omega_{X/\DBbb}^{p}(\log))^{(-1)^{p}p}.
\end{displaymath}
\end{definition}
In our previous work \cite{cdg}, the singularities of the Quillen-BCOV metric were formulated in terms of the K\"ahler extension of the BCOV bundle. In contrast, the degeneration of the $L^{2}$-BCOV metric is well understood for the logarithmic extension, thanks to Theorem \ref{thm:expansion-hodge}. A comparison of both extensions is needed in order to extract the singularities of the BCOV invariant. We define $\mu_{BCOV}$ as the integer realizing this comparison, assuming that $f\colon X\to\DBbb$ is a normal crossings projective degeneration:
\begin{displaymath}
    \widetilde{\lambda}_{BCOV}=\lambda_{BCOV}(\log)+\mu_{BCOV}\Ocal([0]).
\end{displaymath}
Hence, recalling the definition \eqref{eq:def-mu-p} of the integers $\mu_{p}$, we have
\begin{displaymath}
    \mu_{BCOV} = \sum_{p=0}^{n}p(-1)^{p}\mu_{p}.
\end{displaymath}
We now exploit the explicit computation of $\mu_{p}$ provided by Proposition \ref{prop:mu_p}, in order to give an expression for $\mu_{BCOV}$. We make use of the notation introduced in section \ref{subsec:comparison}.
\begin{proposition}\label{prop:mu-bcov}
Let $f\colon X\to\DBbb$ be a projective normal crossings degeneration, and write\linebreak  $X_{0}=\sum m_{i} D_{i}$. Assume that the smooth fibers are Calabi--Yau $n$-folds. Let $d(k)=\dim D(k)=n-k+1$. Then
\begin{displaymath}
       \begin{split}
          \mu_{BCOV}=&-\frac{(9n+5)n}{24}\chi(X_{\infty}) -\sum_{k=1}^{n}(-1)^{k}\frac{(9n+3k+2)d(k)}{24}\chi(D(k))\\
        &\quad \quad
             -\frac{(-1)^{n}}{12}\sum_{k=1}^{n}\int_{D(k)}c_{1}(\Omega_{D(k)})c_{n-k}(\Omega_{D(k)}).
         \end{split}
\end{displaymath}
\end{proposition}
\begin{proof}
Recall the expression we obtained for $\mu_p$ in Proposition \ref{prop:mu_p}:
\begin{displaymath}
   \mu_{p}=(-1)^{p-1}\chi(\Omega_{X_{\infty}}^{\bullet \leq p-1})-\sum_{k=1}^{p}(-1)^{p-k}\chi(\Omega_{D(k)}^{\bullet \leq p-k}).
\end{displaymath}
Therefore
\begin{equation}\label{eq:mu-bcov-1}
    \begin{split}
      \mu_{BCOV}=&-\sum_{p=1}^{n}p\chi(\Omega_{X_{\infty}}^{\bullet\leq p-1})-\sum_{p=1}^{n}\sum_{k=1}^{p}(-1)^{k}p\chi(\Omega_{D(k)}^{\bullet \leq p-k})\\
        =&-\sum_{p=1}^{n}p\chi(\Omega_{X_{\infty}}^{\bullet\leq p-1})-\sum_{k=1}^{n}(-1)^{k}\sum_{p=k}^{n}p\chi(\Omega_{D(k)}^{\bullet \leq p-k}).
    \end{split}
\end{equation}
For an integer $d\geq 0$, define $S_{d}=0+\ldots+d$. Then
\begin{equation}\label{eq:mu-bcov-2}
    \begin{split}
      \sum_{p=1}^{n}p\chi(\Omega_{X_{\infty}}^{\bullet\leq p-1})&=\sum_{j=0}^{n}(-1)^{j}(S_{n}-S_{j})\chi(\Omega_{X_{\infty}}^{j})
      =S_{n}\chi(X_{\infty})-\sum_{j=0}^{n}(-1)^{j}S_{j}\chi(\Omega^{j}_{X_{\infty}})\\
      &=\frac{n^{2}+n}{2}\chi(X_{\infty})-\sum_{j=0}^{n}(-1)^{j}\frac{j^{2}+j}{2}\chi(\Omega^{j}_{X_{\infty}}).
     \end{split}
\end{equation}
There are formulas for such sums. We will invoke them below. Let's first tackle the remaining terms in \eqref{eq:mu-bcov-1}. They have a similar structure:
\begin{equation}\label{eq:mu-bcov-3}
    \begin{split}
        \sum_{k=1}^{n}(-1)^{k}\sum_{p=k}^{n}p\chi(\Omega_{D(k)}^{\bullet \leq p-k})&=\sum_{k=1}^{n}(-1)^{k}\sum_{p=0}^{n-k}(p+k)\chi(\Omega_{D(k)}^{\bullet \leq p})\\
        &=\sum_{k=1}^{n}(-1)^{k}\sum_{p=0}^{n-k}p\chi(\Omega_{D(k)}^{\bullet \leq p}) + \sum_{k=1}^{n}(-1)^{k}k\sum_{p=0}^{n-k}\chi(\Omega_{D(k)}^{\bullet \leq p})\\
        &=\sum_{k=1}^{n}(-1)^{k}\sum_{p=1}^{d(k)}p\chi(\Omega_{D(k)}^{\bullet\leq p-1})+\sum_{k=1}^{n}(-1)^{k}(k-1)\sum_{p=0}^{n-k}\chi(\Omega_{D(k)}^{\bullet \leq p}),
    \end{split}
\end{equation}
where we used that $d(k)=\dim D(k)=n-k+1$.
The first term in \eqref{eq:mu-bcov-3} is computed as in \eqref{eq:mu-bcov-2}:
\begin{equation}\label{eq:mu-bcov-5}
    \sum_{k=1}^{n}(-1)^{k}\sum_{p=1}^{d(k)}p\chi(\Omega_{D(k)}^{\bullet\leq p-1})=\sum_{k=1}^{n}(-1)^{k}\left(\frac{d(k)^{2}+d(k)}{2}\chi(D(k))-\sum_{j=0}^{d(k)}(-1)^{j}\frac{j^{2}+j}{2}\chi(\Omega^{j}_{D(k)})\right).
\end{equation}
For the second term in \eqref{eq:mu-bcov-3}, we observe
   $ \sum_{p=0}^{n-k}\chi(\Omega_{D(k)}^{\bullet \leq p})=d(k)\chi(D(k))-\sum_{p=1}^{d(k)}(-1)^{p}p\chi(\Omega_{D(k)}^{p})$,
hence
\begin{equation}\label{eq:mu-bcov-15}
\begin{split}
        \sum_{k=1}^{n}(-1)^{k}(k-1)\sum_{p=0}^{n-k}\chi(\Omega_{D(k)}^{\bullet \leq p})=&\sum_{k=1}^{n}(-1)^{k}(k-1)(n-k+1)\chi(D(k))\\
         &-\sum_{k=1}^{n}(-1)^{k}(k-1)\sum_{p=1}^{d(k)}(-1)^{p}p\chi(\Omega_{D(k)}^{p}).
\end{split}
\end{equation}
To simplify \eqref{eq:mu-bcov-2}--\eqref{eq:mu-bcov-15}, we infer from \cite[Lemma 4.6]{cdg}:
\begin{eqnarray}\label{eq:mu-bcov-5-bis}
    \sum_{j=1}^{d(k)}(-1)^{j}j\chi(\Omega_{D(k)}^{j})&=&(-1)^{d(k)}\frac{d(k)}{2}\int_{D(k)}c_{d(k)}(\Omega_{D(k)})=\frac{d(k)}{2}\chi(D(k)),
\\
 \label{eq:mu-bcov-16}
 \sum_{j=1}^{d(k)}(-1)^{j}\frac{j^{2}+j}{2}\chi(\Omega_{D(k)}^{j})&=&\frac{d(k)(3d(k)+7)}{24}\chi(D(k))
 \\
 \nonumber
 &&+\frac{(-1)^{d(k)}}{12}\int_{D(k)}c_{1}(\Omega_{D(k)})c_{d(k)-1}(\Omega_{D(k)}).
\end{eqnarray}
A similar expression holds for $X_{\infty}$ instead of $D(k)$. Taking into account that $c_{1}(\Omega_{X_{\infty}})=0$, it reads
\begin{equation}\label{eq:mu-bcov-7}
    \sum_{j=1}^{n}(-1)^{j}\frac{j^{2}+j}{2}\chi(\Omega_{X_{\infty}}^{j})=\frac{n(3n+7)}{24}\chi(X_{\infty}).
\end{equation}
Also, thanks to \eqref{eq:mu-bcov-5-bis}, equation \eqref{eq:mu-bcov-5} simplifies to
\begin{equation}\label{eq:mu-bcov-8}
   \sum_{k=1}^{n}(-1)^{k}\sum_{p=1}^{d(k)}p\chi(\Omega_{D(k)}^{\bullet\leq p-1})= \sum_{k=1}^{n}(-1)^{k}\frac{(k-1)(n-k+1)}{2}\chi(D(k))
\end{equation}
To conclude, it suffices to plug \eqref{eq:mu-bcov-5-bis}--\eqref{eq:mu-bcov-8} into \eqref{eq:mu-bcov-2}--\eqref{eq:mu-bcov-15} and then adding up.
\end{proof}

\subsection{Logarithmic behaviour of the BCOV invariant}\label{subsec:asymptotic-bcov-dim-n}

Recall from the introduction that there is a strong motivation for finding the asymptotic  behaviour of the invariant. For the sake of motivation, consider the special case of a K\"ahler degeneration of Calabi--Yau $3$-folds $f: X \to \DBbb$ admitting a single ordinary double point over the origin, Fang--Lu--Yoshikawa \cite[Theorem 8.2]{FLY} proved that
\begin{equation*}\label{eq:TAUFLYODP}
    \log \tau_{BCOV}(X_t) = \frac{1}{6}\log|t|^2 + O(\log\log|t|^{-1}).
\end{equation*}
This was used in \emph{loc. cit.} in order to prove an instance of genus 1 mirror symmetry for the quintic $3$-folds.

More recently, in \cite{LiuXia}, Liu and Xia study systematically the logarithmic behavior of the BCOV invariant of \cite{FLY}. More precisely they study the limits
\begin{equation*}\label{def:logdiv}
   \kappa_f := \lim_{t \to 0} \frac{\log \taubcov{X_t}}{\log|t|^2}
\end{equation*}
where $t$ is a local parameter around $0 \in \DBbb$ for a projective K\"ahler degeneration $X \to \DBbb$ of Calabi--Yau 3-folds. It follows from \cite[Theorem 9.1]{FLY}  that the limit exists and is a real number. In \cite[Thm. 0.1]{yoshikawa5}, Yoshikawa proved it is a rational number. In \cite{LiuXia} formulas are obtained for sums of local contributions, for a compact one-dimensional base $S$ of a generically smooth family of Calabi--Yau 3-folds $X \to S$. They conjecture that these formulas localize in a precise sense under some specific conditions. In this subsection we address the corresponding matters in arbitrary dimension.

Before we proceed with the statement of the first theorem, we need to recall a definition from \cite{cdg}. Given a projective degeneration $f\colon X\to \DBbb$ between complex manifolds, whose smooth fibers are Calabi--Yau manifolds, we have an injective morphism of line bundles
\begin{displaymath}
    \mathbf{ev}\colon f^{\ast}f_{\ast}K_{X/\DBbb}\hookrightarrow K_{X/\DBbb}.
\end{displaymath}
We may equivalently see $\mathbf{ev}$ as a global section of $K_{X/\DBbb}\otimes (f^{\ast}f_{\ast}K_{X/\DBbb})^{-1}$.

\begin{definition}
We denote $B=\mathrm{div}\ \mathbf{ev}$. We say that $f\colon X\to\DBbb$ is a Kulikov family if $B=0$.
\end{definition}

The divisor $B$ is effective. Because the smooth fibers have trivial canonical bundle, $B$ is supported on the special fiber $X_{0}$. To sum up, we have
\begin{displaymath}
    K_{X/\DBbb}=f^{\ast}f_{\ast}K_{X/\DBbb}\otimes\Ocal(B),\quad |B|\subseteq X_{0}.
\end{displaymath}

\begin{theorem}\label{thm:general-bcov}
Let $f\colon X\to\DBbb$ be a germ of a projective degeneration between complex algebraic varieties, whose smooth fibers are Calabi--Yau manifolds. Then the limit
\begin{equation}\label{eq:sing-bcov}
    \kappa_{f}:=\lim_{t\to 0}\frac{\log\tau_{BCOV}(X_{t})}{\log|t|^{2}}
\end{equation}
exists and is a rational number. Moreover, if $X$ is non-singular and the special fiber $X_{0}=\sum m_{i}D_{i}$ has normal crossings, then
\begin{equation}\label{eq:expression-kappa}
    \begin{split}
        \kappa_{f}&=\frac{3n+1}{12}(\chi(X_{\infty})-\chi(X_{0}))+\sum_{k=1}^{n+1}(-1)^{k}\frac{(k-1)(3k+6n+2)}{24}\chi(D(k))\\
        &\quad -\frac{(-1)^{n}}{12}\int_{B}c_{n}(\Omega_{X})-\sum_{k=1}^{n}\frac{(-1)^{n}}{12}\int_{D(k)}c_{1}(\Omega_{D(k)})c_{n-k}(\Omega_{D(k)})\\
        &\quad \quad -\frac{\alpha}{12}\chi(X_{\infty})-\sum_{0\leq p,q\leq n}(-1)^{p+q}p\alpha^{p,q},
    \end{split}
\end{equation}
where $\alpha=\frac{1}{2\pi i}\tr\left(^{u}\log T_{s}\mid \Gr_{F_{\infty}}^{n}H^{n}(X_{\infty})\right)$.
\end{theorem}

\begin{remark}
\begin{enumerate}[(i)]
    \item The restriction to \emph{germs} of projective morphisms of algebraic varieties stems from the application of our previous work \cite{cdg}. In turn \emph{loc. cit.} relies on Yoshikawa's theorem on the singularities of Quillen metrics \cite{yoshikawa}, where the compacteness of the base of the fibration is needed. The germ assumption is a technical one, satisfied in most applications. It should be possible to remove this assumption from a functorial lifting of the Grothendieck--Riemann--Roch theorem to the level of line bundles. An example of this in the case of curves is provided by \cite{ErikssonQuillen}, where the first author applies Deligne's Riemann--Roch isomorphism to obtain the singularities of the Quillen metric.
    \item The limit $\kappa_{f}$ depends only on the restriction of $f$ to $\DBbb^{\times}$, although the existence of a model over $\DBbb$ is needed.
    \item Notice that the first sum in the expression of $\kappa_{f}$ runs from $k=1$ to $n+1$, in contrast to the first sum in the expression for $\mu_{BCOV}$ (Proposition \ref{prop:mu-bcov}), which runs from $k=1$ to $n$. Observe that $D(n+1)$ is zero dimensional, and that $\chi(D(n+1))=\# D(n+1)$ could be non-trivial.
    \item The integral $\int_{B}c_{n}(\Omega_{X})$ can be worked out explicitly in terms of the geometry/topology of the components of the special fiber $X_{0}$, and their incidence relations. Instead of giving a cumbersome general expression, we refer the reader to the proof of Theorem \ref{thm:bcov-odp-dim-n} for an example of use.
\end{enumerate}
\end{remark}

\begin{proof}[Proof of Theorem \ref{thm:general-bcov}]
Because $\kappa_{f}$ only depends on the restriction of $f$ to $\DBbb^{\times}$, it is enough to prove the second assertion. By resolution of singularities we can suppose that $X$ is smooth and $X_{0}=\sum n_{i}D_{i}$ is a divisor with normal crossings. Also we may take an auxiliary K\"ahler structure $\omega$ induced by a projective embedding, with respect to which we compute $L^{2}$ metrics and volumes.

Let $\widetilde{\lambda}_{BCOV}$ be the K\"ahler extension of the BCOV bundle. Let $\sigma$ be a trivializing holomorphic section. In \cite[Cor. 4.9]{cdg} we showed that
\begin{equation}\label{eq:general-bcov-1}
    \begin{split}
         \log h_{Q, BCOV}(\sigma,\sigma)=&\left\lbrace\frac{9n^{2}+11n+2}{24}(\chi(X_{\infty})-\chi(X_{0}))
         -\frac{\alpha}{12}\chi(X_{\infty})
          +\frac{(-1)^{n+1}}{12}\int_{B}c_{n}(\Omega_{X})\right\rbrace\log|t|^{2}\\
          &+o(\log|t|).
    \end{split}
\end{equation}
Notice that in \emph{loc. cit.} the last term in the asymptotics was written as $\int_{B}c_{n}(\Omega_{X/\DBbb})$, which equals $\int_{B} c_{n}(\Omega_{X})$ as stated above. Recall now the relation $\widetilde{\lambda}_{BCOV}=\lambda_{BCOV}(\log)+\mu_{BCOV}\Ocal([0])$. This means that if we are given trivializing sections $\sigma$ as above, and $\sigma^{\prime}$ of $\lambda_{BCOV}(\log)$, then the relation between the respective Quillen-BCOV square norms is
\begin{equation*}\label{eq:general-bcov-2}
    \log h_{Q,BCOV}(\sigma,\sigma)=\log h_{Q,BCOV}(\sigma^{\prime},\sigma^{\prime})-\mu_{BCOV}\log|t|^{2}+\text{continuous}.
\end{equation*}
We can take $\sigma^{\prime}$ of the form
\begin{displaymath}
    \sigma^{\prime}=\bigotimes_{p,q}(\sigma^{(p,q)})^{(-1)^{p+q}p},
\end{displaymath}
where $\sigma^{(p,q)}$ trivializes $\det R^{q}f_{\ast}\Omega^{p}_{X/\DBbb}(\log)$. Then, by Theorem \ref{thm:expansion-hodge} and because the $L^{2}$ volumes $\vol_{L^{2}}(H^{k}(X_{t},\ZBbb),\omega)$ stay constant (choice of rational K\"ahler structure and Proposition \ref{prop:rational-volume}), we find
\begin{equation}\label{eq:general-bcov-3}
    \log h_{L^{2}, BCOV}(\sigma^{\prime},\sigma^{\prime})=\left(\sum_{p,q}(-1)^{p+q}p\alpha^{p,q}\right)\log |t|^{2}+O(\log\log|t|^{-1}).
\end{equation}
But $\log\tau_{BCOV}(X_{t})=\log h_{Q,BCOV}(\sigma^{\prime},\sigma^{\prime})-\log h_{L^{2},BCOV}(\sigma^{\prime},\sigma^{\prime})$. Therefore the conclusion is achieved by combining equations \eqref{eq:general-bcov-1}--\eqref{eq:general-bcov-3} with Proposition \ref{prop:mu-bcov}, together with the relation
\begin{equation}\label{eq:chi-X0}
    \chi(X_{0})=-\sum_{k=1}^{n+1}(-1)^{k}\chi(D(k))=-\sum_{k=1}^{n}(-1)^{k}\chi(D(k))+(-1)^{n}\chi(D(n+1)).
\end{equation}
\end{proof}

\begin{corollary}
Let $f\colon X\to\DBbb$ be a germ of normal crossings projective degeneration of algebraic varieties. Write $X_0=\sum_{i=1}^{r} m_{i} D_{i}$ and define $M=\operatorname{lcm}(m_1,\ldots, m_r)$. Then
\begin{displaymath}
    12M\kappa_{f}\in\ZBbb.
\end{displaymath}
If all the monodromies are unipotent, then $12\kappa_{f}\in\ZBbb$.
\end{corollary}
\begin{proof}
The first statement follows from the following observations. First, $\operatorname{lcm}(m_{1},\ldots,m_{r})$ kills the semi-simple part of the monodromy endomorphisms acting on the cohomology groups $H^{k}(X_{\infty})$, and hence all the $M\alpha$ and $M\alpha^{p,q}$ are integers. Second, the numerators $(k-1)(3k+6n+2)$ in the second sum in \eqref{eq:expression-kappa} are always even integers. The second claim follows analogously, using that the $\alpha$ and $\alpha^{p,q}$ vanish.
\end{proof}

\subsection{The subdominant term in the asymptotics of the BCOV invariant}
Let $f\colon X\to\DBbb$ be a germ of a projective degeneration between complex algebraic varieties, whose smooth fibers are Calabi--Yau manifolds. In the previous subsection, we provided a general expression for the leading term of the asymptotic behaviour of the function $t\mapsto\log\taubcov{X_{t}}$. The following statement describes the subdominant term.
\begin{proposition}\label{prop:loglogbcov}
The assumptions being as above, we have an asymptotic expansion
\begin{displaymath}
    \log\taubcov{X_{t}}=\kappa_{f}\log|t|^{2}+\varrho_{f}\log\log|t|^{-1}+\text{continuous},
\end{displaymath}
where $\varrho_{f}$ is given in terms of the limiting Hodge structures by
\begin{displaymath}
    \varrho_{f}=\frac{\chi(X_{\infty})}{12}\beta^{n,0}-\sum_{p,q} (-1)^{p+q}p\beta^{p,q}.
\end{displaymath}
\end{proposition}
\begin{proof}
It is enough to combine the definition of the BCOV invariant (Definition \ref{deff:bcovinv}), together with \cite[Prop. 4.2]{cdg}, Theorem \ref{thm:general-bcov} and Theorem \ref{thm:expansion-hodge}.
\end{proof}

An immediate consequence of the proposition is the existence of the following limit:
\begin{equation}\label{eq:tau-bcov-lim}
    \tau_{BCOV,\ \lim}=\lim_{t\to 0}\frac{\taubcov{X_{t}}}{|t|^{2\kappa_{f}}(\log|t|^{-1})^{\varrho_{f}}}\in\RBbb_{>0}.
\end{equation}
This limit actually depends on the choice of coordinate $t$ on $\DBbb$. Under a change of coordinate $t\mapsto\lambda(t)$, the limit changes to $|\lambda'(0)|^{2\kappa_{f}}\cdot\tau_{BCOV,\ \lim}$. This can be restated by saying that $\tau_{BCOV,\ \lim}$ defines a hermitian metric on the $\QBbb$-complex line $(\omega_{\DBbb,0})^{\otimes\kappa_{f}}$. By construction, this metric only depends on the restriction of $f\colon X\to\DBbb$ to $\DBbb^{\times}$, but not on the special fiber $X_{0}$.

\begin{remark}
\begin{enumerate}
    \item As an application of Lemma \ref{lemma:beta-p-q}, the expression for $\varrho_{f}$ can equivalently be written as
\begin{displaymath}
    \varrho_{f}=\frac{\chi(X_{\infty})}{12}\beta^{n,0}+2\sum_{\substack{p+q<n\\ q< p}}(-1)^{p+q}(q-p)\beta^{p,q}+(-1)^{n}\sum_{\substack{p+q=n\\ q<p}}(q-p)\beta^{p,q}.
\end{displaymath}
    \item For maximally unipotent degenerations of Calabi--Yau varieties as above, the mirror map provides a (quasi-)canonical coordinate on $\DBbb$. This choice of coordinate induces a trivialization of $\omega_{\DBbb,0}$, so that $\tau_{BCOV,\ \lim}$ becomes an unambigously defined quantity. 
    \item For open Calabi--Yau manifolds with cyclindrical ends, Conlon--Mazzeo--Rochon \cite{CMR} have defined an avatar of the Quillen-BCOV metric. It would be interesting to explore the connection between their work and our limiting invariant $\tau_{BCOV,\ \lim}$. 
\end{enumerate}
\end{remark}

\subsection{Extensions of the holomorphic anomaly equation}
Let $f\colon X\to S$ be a flat projective morphism of compact connected complex manifolds, with $\dim S=1$. We let $S^{\times}\subseteq S$ be the locus of the regular values of $f$, which is necessarily Zariski open and non-empty. Write $\lbrace P_{1},\ldots, P_{r}\rbrace$ for the complement $S\setminus S^{\times}$. If the smooth fibers of $f$ are Calabi--Yau manifolds, then they have a well-defined BCOV invariant. By Theorem \ref{thm:general-bcov} the smooth function $s\mapsto\log\tau_{BCOV}(X_{s})$ on $S^{\times}$ extends to a locally integrable function on $S$. For every point $P_{i}$, the singularity of $\log\tau_{BCOV}(X_{s})$ is of logarithmic type, with an attached coefficient $\kappa_{f}(P_{i})$ defined as in \eqref{eq:sing-bcov}. In general, for a locally integrable differential form $\theta$ on $S$, we denote by $[\theta]$ the current of integration against $\theta$. Hence $[\log\tau_{BCOV}]$ is defined. Also the current extension $[dd^{c}\log\tau_{BCOV}]$ can be defined:

\begin{lemma}
The differential form $dd^{c}\log\tau_{BCOV}$ on $S^{\times}$ is locally integrable on the whole of $S$. Moreover, we have the equality of currents
\begin{equation}\label{eq:dirac-bcov}
    dd^{c}[\log\tau_{BCOV}]-[dd^{c}\log\tau_{BCOV}]=\sum_{i=1}^{r}\kappa_{f}(P_{i})\cdot \delta_{P_{i}}.
\end{equation}
\end{lemma}
\begin{proof}
For the proof one needs a complement to the analysis of the subdominant term in Proposition \ref{prop:loglogbcov}, in order to include its $d$ and $dd^{c}$ derivatives. On the one hand, we need a control on the remainder of the asymptotics of the Quillen-BCOV metrics \eqref{eq:general-bcov-1}. This was addressed in \cite[Prop. 4.2]{cdg}: the remainder, together with its $d$ and $dd^c$ derivatives, is modeled on $\log\log|t|^{-1}$ and its $d$ and $dd^{c}$ derivatives. On the other hand, we have a similar property for the $L^{2}$ metrics on Hodge bundles for a choice of fiberwise rational K\"ahler structure, by Theorem \ref{thm:expansion-hodge}. One easily concludes from these facts that $dd^{c}\log\tau_{BCOV}$ has at worst Poincar\'e growth, and is in particular locally integrable. Also \eqref{eq:dirac-bcov} follows from the indicated behaviour of the remainder term and by Theorem \ref{thm:general-bcov}, by a standard evaluation of the current $dd^{c}[\log\tau_{BCOV}]-[dd^{c}\log\tau_{BCOV}]$ on test functions.
\end{proof}

Let us now choose a K\"ahler structure on $f$, fiberwise rational on the smooth locus. We let $\omega_{H^{k}}$ be the corresponding Hodge forms, defined on $S^{\times}$ by \eqref{eq:def-hodge-form-1}--\eqref{eq:def-hodge-form-2}. Denote as before $\omega_{WP}$ the Weil--Petersson form on $S^{\times}$. By Theorem \ref{thm:expansion-hodge}, the differential forms $\omega_{H^{k}}$ and $\omega_{WP}$ are locally integrable and have at most Poincar\'e growth singularities on $S$. Hence they define currents $[\omega_{H^{k}}]$ and $[\omega_{WP}]$ by integration. These are all closed semi-positive currents, and in particular they have well-defined (current) cohomology classes on $S$ denoted $\{\omega_{H^{k}}\}$ and $\{\omega_{WP}\}$ for these cohomology classes. The cohomology classes of $\delta_{P_{i}}$ are identified to the cycle cohomology classes of the points $P_{i}$, denoted $\{P_{i}\}$.

\begin{proposition}
We have an equality in $H^{1,1}(S)$
\begin{equation*}
    \frac{\chi}{12}\{\omega_{WP}\}-\sum_{k=1}^{2n}(-1)^{k}\{\omega_{H^{k}}\}=\sum_{i}\kappa_{f}(P_{i})\{P_{i}\},
\end{equation*}
where $\chi$ is the topological Euler characteristic of a general smooth fiber of $f$. Consequently, there is a relation
\begin{displaymath}
    \frac{\chi}{12}\int_{S}\omega_{WP}-\sum_{k=1}^{2n}(-1)^{k}\int_{S}\omega_{H^{k}}=\sum_{i=1}^{r}\kappa_{f}(P_{i}).
\end{displaymath}
\end{proposition}
\begin{proof}
The first equation is a combination of \eqref{eq:dirac-bcov} and the differential equation in Proposition \ref{prop:ddc-log-B}, applied over $S^{\times}$ and extended to the level of currents by integration. The second claim is obtained by integration over $S$.
\end{proof}

\section{Asymptotics for special geometries }\label{sec:specialgeom}

The general discussion in the previous section can be bolstered if further assumptions on the degeneration $ \colon X \to S$ are imposed. We initially treat the cases of  semi-stable minimal (or Kulikov) degenerations (cf. Proposition \ref{cor:liuxiahigherdim}) as well as for ordinary double point singularities (cf. Theorem \ref{thm:bcov-odp-dim-n}). For the small dimensions 3 and 4, we can give refined statements for general degenerations (cf. Theorem \ref{thm:LiuXia-general} and Theorem \ref{thm:LiuXia-general4}). As a corollary we will then deduce the conjecture of Liu-Xia of \cite[Conj. 0.5]{LiuXia}.  We conclude the article with some algebraic geometric applications of our results, summarized here as constraints to the existence of particular degenerations of Calabi--Yau varieties. 

\subsection{Kulikov degenerations}
In this subsection we consider a semi-stable germ of a projective degeneration  $f: X \to  \DBbb$  of smooth algebraic varieties. Furthermore, we suppose that $f$ is Kulikov,  \emph{i.e.} $B=0$. We write $X_{0}=\sum_{i}D_{i}$.
\begin{proposition}\label{cor:liuxiahigherdim}
With the above assumptions and notations, we have
\begin{equation}\label{eq:kappa-kulikov}
    \kappa_{f}=\sum_{k=1}^{n+1}(-1)^{k}\frac{k(k-1)}{24}\chi(D(k)).
\end{equation}
\end{proposition}
\begin{proof}
First, from Theorem \ref{thm:general-bcov} it easily follows the intermediate expression 
\begin{equation}\label{eq:kappa-kulikov-I}
    \kappa_{f}=\sum_{k=1}^{n+1}(-1)^{k}\frac{k(k-1)}{8}\chi(D(k))
       -\sum_{k=1}^{n}\frac{(-1)^{n}}{12}\int_{D(k)}c_{1}(\Omega_{D(k)})c_{n-k}(\Omega_{D(k)}).
\end{equation}
Indeed, the semi-stable case, the monodromy is unipotent, and therefore $\alpha=\alpha^{p,q}=0$ for all $p,q$. A standard computation, see \emph{e.g.} \cite[Lemma 1.4]{SaitoTakeshiBored}, shows that if $D_i^\circ = D_i \setminus  \bigcup_{j \neq i} D_{ij}$, then
\begin{equation}\label{eq:euler-saito}
    \chi(X_\infty) = \sum \chi(D_i^\circ) =  \sum_{k=1}^{n+1}(-1)^{k+1}k\chi(D(k)).
\end{equation}
Combining with equation \eqref{eq:chi-X0}, we derive
\begin{displaymath}
    \chi(X_{\infty})-\chi(X_{0})=\sum_{k=1}^{n+1}(-1)^{k+1}(k-1)\chi(D(k)).
\end{displaymath}
Now it is enough to plug this relation in the general expression for $\kappa_{f}$ \eqref{eq:expression-kappa}, together with the vanishing of $B$ and $\alpha$, $\alpha^{p,q}$, to conclude with \eqref{eq:kappa-kulikov-I}.

We next proceed to simplify the contribution of the integrals in \eqref{eq:kappa-kulikov-I}. For this, we establish a recursion relating the integral over $D(k)$ to the integral over $D(k+1)$. Enumerate the components of $X_{0}$ as $D_{1},\ldots, D_{r}$. Let $I\subseteq \lbrace 1,\ldots,r\rbrace$ be a multi-index subset of order $k$. Accordingly, define $D_{I}=\cap_{i\in I}D_{i}$ and $B_{I}=D_{I}\cap(\cup_{j\not\in I} D_{j})$. The Kulikov assumption, together with the triviality (as a Cartier divisor) of $X_{0}=\sum_{i} D_{i}$ and the adjunction formula guarantee that the pairs $(D_{I},B_{I})$ are $\log$-Calabi--Yau, \emph{i.e.} $K_{D_{I}}+B_{I}=0$. Combining this with the conormal exact sequence for the inclusions $D_{I\cup\lbrace j\rbrace}\hookrightarrow D_{I}$  $(j\not\in I)$, we obtain the equalities
\begin{eqnarray*}
    \int_{D(k)}c_{1}(\Omega_{D(k)})c_{n-k}(\Omega_{D(k)}) &=&-\sum_{|I|=k} \int_{D_I}[B_I] \cap c_{n-k}(\Omega_{D(k)}) \\ &=& -(k+1)\int_{D(k+1)} c_{n-k}(\Omega_{D(k+1)}) + \sum_{\substack{|I|=k\\ j\not\in I}} \int_{D_{I\cup\lbrace j\rbrace}} [D_j]_{\mid_{D_{I\cup \{j\}}}} c_{n-k-1}(\Omega_{D_{I \cup \{ j\} })}) \\ &=& (-1)^{n-k+1} (k+1) \chi(D(k+1))+ \int_{D(k+1)}c_{1}(\Omega_{D(k+1)})c_{n-k-1}(\Omega_{D(k+1)}).
\end{eqnarray*}
In the last equality we used that $\int_{D_I} c_{n-k}(\Omega_{D_I}) = (-1)^{n-k}\chi(D_I)$ and $K_{D_I} = -B_I = \sum_{i \in I} [D_i]|_{D_I}$. The result follows by applying this recursion to \eqref{eq:kappa-kulikov-I}.
\end{proof}
\begin{remark}
In the case of K3 surfaces, an application of adjunction shows that the right hand side of \eqref{eq:kappa-kulikov} vanishes. This is in agreement with the constancy of the BCOV invariant established in Proposition \ref{prop:ken-ichi}, in fact equal to 1 (cf. Remark \ref{rem:bcovk3ab}), which also implies the vanishing of $\kappa_{f}$.
\end{remark}

\subsection{Ordinary double point singularities}
In this subsection, let $f\colon X\to\DBbb$ be a germ of a projective degeneration between smooth algebraic varieties, with general Calabi--Yau fibers of dimension $n$. Suppose that $X_{0}$ admits at most ordinary double point singularities. 
\begin{theorem}\label{thm:bcov-odp-dim-n}
With the above assumptions and notations, we have
\begin{displaymath}
    \kappa_{f}=\frac{n+1}{24}\#\sing(X_{0})\quad\text{ and } \quad \varrho_f = \#\sing(X_{0}) \quad\text{if } n \text{ is odd},
\end{displaymath}
or
\begin{displaymath}
    \kappa_{f}=-\frac{n-2}{24}\#\sing(X_{0}) \quad\text{ and } \quad \varrho_f = 0\quad\text{if } n \text{ is even},
\end{displaymath}
where $\#\sing(X_{0})$ denotes the number of singular points in the fiber $X_{0}$.
\end{theorem}
For later use we record the following lemma, which follows from the conormal exact sequence for the cotangent bundle, and the Euler sequence on $\PBbb^n_{\CBbb}$ restricted to $W$:
\begin{lemma}\label{lemma:ints-odp-1}
Let $W$ be an irreducible degree $d$ smooth hypersurface in $\PBbb^{n}_{\CBbb}$. Then
\begin{displaymath}
    \int_{W}c_{1}(\Ocal_{W}(1))c_{n-2}(\Omega_{W})=\frac{(-1)^{n-1}}{d}\chi(W)+(-1)^{n}\frac{n(n+1)}{2}.
\end{displaymath}
\end{lemma}

For the following lemma, the reader is advised to review the description of the blow up of $X$ along the ordinary double point singularities \textsection\ref{sec:hodgeodp}. Also, we notice that for the morphism $f$, the divisor of the evaluation map $\mathbf{ev}$ is trivial, hence $K_{X}=\Ocal_{X}$ is trivial as well: $f$ is a Kulikov family.

\begin{lemma}\label{lemma:ints-odp-2}
Let $\nu\colon\widetilde{X}\to X$ be the blow up of the ordinary double points in $X_{0}$. Let $Z$ be the strict transform of $X_{0}$ in $\widetilde{X}$. Then
\begin{displaymath}
    \int_{Z}c_{1}(\Omega_{Z})c_{n-1}(\Omega_{Z})= \left((-1)^{n-1}\frac{3(n-2)}{2}\chi(Q)+(-1)^{n}\frac{(n-2)n(n+1)}{2}\right)\#\sing(X_{0}),
\end{displaymath}
where $Q$ is any smooth quadric in $\PBbb_{\CBbb}^{n}$, and hence $\chi(Q)=n+(1+(-1)^{n+1})/2$.
\end{lemma}
\begin{proof}
Let $E=E_{1}+\ldots+E_{r}$ be the exceptional divisor, where $r$ is the number of ordinary double point singularities. Each $E_{i}$ is isomorphic to $\PBbb_{\CBbb}^{n}$. Since the canonical bundle $K_{X}$ is trivial, we have
\begin{displaymath}
    K_{\widetilde{X}}=\Ocal(nE).
\end{displaymath}
Hence, because $\widetilde{X}_{0}=2E+Z$, we find by the adjunction formula
\begin{displaymath}
    K_{Z}=\Ocal_{Z}((n-2)E)=\Ocal_{Z}((n-2)W),
\end{displaymath}
where $W=\sum E_{i}\cap Z$, and each $E_{i}\cap Z$ is an irreducible smooth quadric in $\PBbb^{n}_{\CBbb}$. Hence
\begin{equation}\label{eq:cherns-Z-1}
    \int_{Z}c_{1}(\Omega_{Z})c_{n-1}(\Omega_{Z})=(n-2)\int_{W}c_{n-1}(\Omega_{Z}\mid_{W}).
\end{equation}
From the conormal exact sequence of the immersion $W\hookrightarrow Z$, we derive
\begin{equation}\label{eq:cherns-Z-2}
    \begin{split}
        \int_{W} c_{n-1}(\Omega_{Z}\mid_{W})=&\int_{W}c_{n-1}(\Omega_{W})-\int_{W}c_{1}(N_{W/Z})c_{n-2}(\Omega_{W})\\
        =&(-1)^{n-1}\chi(W)-\int_{W}c_{1}(N_{W/Z})c_{n-2}(\Omega_{W}).
    \end{split}
\end{equation}
To compute the last integral in \eqref{eq:cherns-Z-2}, we notice that $N_{W/\widetilde{X}}=N_{W/Z}\oplus N_{W/E}$, because $Z$ and $E$ intersect transversally. It becomes
\begin{displaymath}
    c_{1}(N_{W/Z})=c_{1}(K_{W})-c_{1}(K_{\widetilde{X}}\mid_{W})-c_{1}(N_{W/E}).
\end{displaymath}
Together with the adjunction formula for $W\hookrightarrow E$, we infer
\begin{displaymath}
    c_{1}(N_{W/Z})=c_{1}(K_{E}\mid_{W})-c_{1}(K_{\widetilde{X}}\mid_{W}).
\end{displaymath}
Now recall that $K_{E}=\Ocal_{E}(-n-1)$ and $K_{\widetilde{X}}\mid_{W}=\Ocal_{E}(nE)\mid_{W}=\Ocal_{E}(-n)\mid_{W}$. Therefore $c_{1}(N_{W/Z})=c_{1}(\Ocal_{E}(-1)\mid_{W})$ and
\begin{equation}\label{eq:cherns-Z-3}
    \int_{W}c_{1}(N_{W/Z})c_{n-2}(\Omega_{W})=-\int_{W}c_{1}(\Ocal_{W}(1))c_{n-2}(\Omega_{W}).
\end{equation}
To conclude, we apply Lemma \ref{lemma:ints-odp-1} and add up \eqref{eq:cherns-Z-1}--\eqref{eq:cherns-Z-3}.
\end{proof}

\begin{proof}[Proof of Theorem \ref{thm:bcov-odp-dim-n}]

The computation of the subdominant term is a consequence of the description of the limiting Hodge structure of ordinary double points \cite[Example 2.15]{Steenbrink-mixedonvanishing}.

For the dominant term, we let $\widetilde{X}\to X$ be the blow up of the ordinary double point singularities, and we apply Theorem \ref{thm:general-bcov} to the projection $\widetilde{f}\colon\widetilde{X}\to\DBbb$. The special fiber $\widetilde{X}_{0}=2E+Z$ is as in the previous lemma: $Z$ is the strict transform of $X_{0}$, $E=\sum E_{i}$ is a disjoint union of $\PBbb_{\CBbb}^{n}$, and the intersections $W_{i}=Z\cap E_{i}$ are irreducible smooth quadrics in these projective spaces. We put $W=\sum W_{i}$ in $Z$. Finally, the divisor of the evaluation map $\mathbf{ev}$ for $\widetilde{f}$ is $nE$. We evaluate all the contributions to $\kappa_{\widetilde{f}}$.

For the Euler characteristic of the general fiber $\widetilde{X}_{\infty}$ we have
\begin{displaymath}
    \chi(\widetilde{X}_{\infty})=2\chi(E)+\chi(Z)-3\chi(W).
\end{displaymath}
To verify this, we first notice that the degree of $c_n(\Omega_{\widetilde{X}/\DBbb}(\log))$ on fibers is constant. On a general fiber this is $(-1)^n \chi(\widetilde{X}_{\infty})$. On the special fiber, we find it is equal to the degree of $2 c_n(\Omega_{E}(\log W)) + c_n(\Omega_{Z}(\log W))$. This can be computed through the residue exact sequence (see e.g. \eqref{eq:diagram-res}), and the result follows.
Also,
\begin{displaymath}
    \chi(\widetilde{X}_{0})=\chi(E)+\chi(Z)-\chi(W).
\end{displaymath}
Using that $\chi(E)=(n+1)\#\sing(X_{0})$ and $\chi(W)=[n+(1+(-1)^{n+1})/2]\#\sing(X_{0})$, we obtain
\begin{equation}\label{eq:bcov-odp-gen-1}
    \frac{3n+1}{12}(\chi(\widetilde{X}_{\infty})-\chi(\widetilde{X}_{0}))=\frac{(3n+1)(-n+(-1)^{n})}{12}\#\sing(X_{0}).
\end{equation}
The next term in \eqref{eq:expression-kappa} equals
\begin{equation}\label{eq:bcov-odp-gen-2}
    \frac{3n+4}{12}\chi(W)=\frac{3n+4}{12}\left(n+\frac{1+(-1)^{n+1}}{2}\right)\#\sing(X_{0}).
\end{equation}
By the conormal exact sequence of $E\hookrightarrow\widetilde{X}$, and taking into account that $c(\Omega_{E})=(1-c_{1}(\Ocal_{E}(1)))^{n+1}$ and $\Ocal(E)\mid_{E}=\Ocal_{E}(-1)$, we have
\begin{equation}\label{eq:bcov-odp-gen-4}
    -\frac{(-1)^{n}}{12}\int_{nE}c_{n}(\Omega_{\widetilde{X}})=\left\lbrace\frac{n^{2}(n+1)}{24}-\frac{n(n+1)}{12}\right\rbrace\#\sing(X_{0}).
\end{equation}
For the following terms, we apply lemmas \ref{lemma:ints-odp-1} and \ref{lemma:ints-odp-2}. Using again $c(\Omega_{E})=(1-c_{1}(\Ocal_{E}(1)))^{n+1}$ and $K_{W}=\Ocal_{W}(-n+1)$, we find
\begin{equation}\label{eq:bcov-odp-gen-3}
    \begin{split}
        &-\frac{(-1)^{n}}{12}\left\lbrace\int_{E}c_{1}(\Omega_{E})c_{n-1}(\Omega_{E})+\int_{Z}c_{1}(\Omega_{Z})c_{n-1}(\Omega_{Z})+\int_{W}c_{1}(\Omega_{W})c_{n-2}(\Omega_{W})\right\rbrace =\\
        &\hspace{5cm} \left\lbrace -\frac{n^{2}(n+1)}{24}+\frac{2n-5}{24}\left(n+\frac{1+(-1)^{n+1}}{2}\right)\right\rbrace\#\sing(X_{0}).
    \end{split}
\end{equation}
Finally, by Proposition \ref{prop:ODP} we know: i) if $n$ is odd, then $\alpha^{p,q}=0$ for all $p,q$; ii) if $n$ is even, then $\alpha^{p,q}=0$ for $(p,q)\neq (n/2,n/2)$ and
\begin{displaymath}
    \alpha^{n/2,n/2}=\frac{1}{2}\#\sing(X_{0}).
\end{displaymath}
 Since moreover $\alpha=0$, we therefore conclude
\begin{equation}\label{eq:bcov-odp-gen-5}
    -\frac{\alpha}{12}\chi(X_{\infty})-\sum_{p,q}(-1)^{p+q}p\alpha^{p,q}=
    \begin{cases}
        0 &\quad\text{if }n\text{ is odd}\\
        -\frac{n}{4}\#\sing(X_{0})&\quad\text{if }n\text{ is even}.
    \end{cases}
\end{equation}
To complete the proof, one just needs to evaluate the sum \eqref{eq:bcov-odp-gen-1}$+\ldots+$\eqref{eq:bcov-odp-gen-5}.

\end{proof}

\subsection{Strict Calabi--Yau varieties: dimensions 3 and 4}\label{subsec:asymptotic-bcov-dim-3-4}
In the case of degenerating families of strict Calabi--Yau $3$-folds and $4$-folds we can give general results on the asymptotic behaviour of the BCOV invariant, not supposing that the central fiber is a normal crossings divisor. 

We suppose first that $f\colon X \to \DBbb$ is a germ of a projective degeneration between smooth algebraic varieties, whose smooth fibers are strict Calabi--Yau 3-folds and with special fiber $X_0 = \sum m_i D_i$, not necessarily of normal crossings.

\begin{theorem}\label{thm:LiuXia-general}
With the above assumptions and notations, 
\begin{enumerate}
\item we have
\begin{displaymath}
    \kappa_f = -\frac{1}{6}(\chi(X_\infty)- \chi(X_0)) - \left(\frac{\chi(X_{\infty})}{12}+3\right)\alpha
     +\alpha^{1,1}-\alpha^{1,2}-\sum\chi(\Ocal_{\widetilde{D_i}})+ \frac{1}{12} \int_B c_3(\Omega_{X})
\end{displaymath}
where $\widetilde{D_i}$ denotes a desingularization of each $D_i$ and $\alpha = \frac{1}{2\pi i}\left(^{\textit{u}}\log T_{s}\mid \Gr^{3}_{F_{\infty}} H^{3}(X_{\infty})\right)$.
\item If $f \colon X \to \DBbb$ has unipotent monodromies,
$$\kappa_f =-\frac{1}{6}(\chi(X_\infty)- \chi(X_0)) +\left(\chi(\Ocal_{X_\infty}) - \sum\chi(\Ocal_{\widetilde{D_{i}}})\right)+ \frac{1}{12} \int_B c_3(\Omega_{X}) $$
\end{enumerate}
\end{theorem}

The proof is given below. We first deduce an immediate corollary.


\begin{corollary}\label{cor:rational3}
     Suppose furthermore that $X_0$ has at most rational singularities.
     \begin{enumerate}
     \item Then $$\kappa_f = -\frac{1}{6}(\chi(X_\infty)- \chi(X_0)) - \alpha^{1,2}.$$
     \item If the singularities are moreover isolated, we have
     $$\kappa_f = \frac{1}{6}\mu_{f}  - \alpha^{1,2}$$
     where $\mu_f$ denotes the Milnor number of the special fiber.
     \end{enumerate}
\end{corollary}
\begin{proof}
Since $X_0$ is normal, $f: X \to \DBbb$ is automatically Kulikov, \emph{i.e.} $B=0$. If $X_0$ has rational singularities one finds that $\chi(\Ocal_{\widetilde{X_0}}) = \chi(\Ocal_{{X_0}})=\chi(\Ocal_{X_\infty})=0$ where we used the flatness of $f$ and the strict Calabi--Yau condition on smooth fibers.
Finally, for isolated singularities the monodromy acts trivially on $H^2$ so $\alpha^{1,1}=0$, and since rational and canonical singularities are equivalent for Gorenstein complex varieties (see \emph{e.g.} \cite[Corollary 11.13]{kollar:singpairs}) it follows from \cite[Proposition 2.8]{cdg} that $\alpha =0$. For isolated singularities in $X_0$ the total dimension of the vanishing cycles $-(\chi(X_\infty)- \chi(X_0))$ is the Milnor number $\mu_f$. 
\end{proof}

\begin{proof}[Proof of Theorem \ref{thm:LiuXia-general}]
By the definition of the BCOV invariant, it suffices to study the asymptotic behavior of the $L^2$- and Quillen-BCOV metrics on the K\"ahler extension of the BCOV bundle. We fix a K\"ahler metric on the total space $X$, whose K\"ahler form is fiberwise integral on the smooth locus. The asymptotic behavior of the Quillen-BCOV metric was established in \cite[Cor 4.9]{cdg} (see \eqref{eq:general-bcov-1}) and states that for a local trivialization of $\widetilde{\lambda}_{BCOV}$
\begin{equation}\label{eq:asym-Q-bcov-dim3}
    \log \|\sigma\|^2_{BCOV, Q} = \left(\frac{29}{6}(\chi(X_\infty)- \chi(X_0)) - \frac{\alpha}{12} \chi(X_\infty) + \frac{1}{12} \int_B c_3(\Omega_{X})\right) \log |t|^2 + o(\log|t|).
\end{equation}
In fact this also holds without the assumption of having normal crossings in the special fiber.

To control the asymptotic behaviour of the $L^2$-BCOV metric, we first notice that the renormalizing factors are rational numbers by Proposition \ref{prop:rational-volume} and thus constant in the family, so for the purposes of the asymptotic behaviour we can assume we are working with the non-renormalized $L^2$ metric on the BCOV bundle. We can then apply Serre duality (cf. Proposition \ref{prop:kahlercomp})  and find that there is an isometry
\begin{equation*}
    \widetilde{\lambda}_{BCOV, L^2} \simeq -3 \lambda(\Ocal_{X}) + \lambda(\Omega_{X/\DBbb}) +5c \cdot \Ocal([0])
\end{equation*}
where $\Ocal([0])$ is equipped with the trivial singular metric and $c=-(\chi(X_{\infty})-\chi(X_{0}))$ is a localized Chern class. We first compare with the logarithmic extension, whose $L^2$ norms is easier to handle. By Proposition \ref{prop:kahlerlogcomp} we have
$\lambda(\Omega_{X/\DBbb})=\lambda(\Omega_{X/\DBbb}(\log))+ \left(\chi(\Ocal_{X_\infty}) - \sum\chi(\Ocal_{\widetilde{D_{i}}}))\right)\Ocal([0])$
hence, as $X_\infty$ is a strict Calabi--Yau threefold, we find the isometry
\begin{equation}\label{eq:bcov-kahler-simplification}
        \widetilde{\lambda}_{BCOV, L^2} \simeq  -3 \lambda(\Ocal_{X}) + \lambda(\Omega_{X/\DBbb}(\log))) +\left(5c - \sum\chi(\Ocal_{\widetilde{D_{i}}}))\right) \cdot \Ocal([0]).
\end{equation}
We study individually the asymptotic behaviour of the various $L^2$ metrics. First of all, notice that all the $R^i f_\ast \Ocal_{X}$ are locally free and commute with restriction to the fibers. To see this, if $\widetilde{X}\to X$ is a birational morphism of complex manifolds such that $\tilde{f}\colon\widetilde{X}\to\DBbb$ is a projective normal crossings degeneration, then $R^{i}f_{\ast}\Ocal_{X}\simeq R^{i}\tilde{f}_{\ast}\Ocal_{\widetilde{X}}$, and the latter is locally free as it is a higher direct image of a logarithmic sheaf. It easily follows, by connectedness of the fibers, that $f_{\ast}\Ocal_{X}=\Ocal_{\DBbb}$. By Grothendieck-Serre duality, there is then an isomorphism
\begin{equation}\label{eq:expdeterminant}
    \lambda(\Ocal_{X}) \simeq  \Ocal_\DBbb - \det R^1 f_\ast \Ocal_{X}+ \det  R^2 f_\ast \Ocal_{X}- R^3 f_\ast \Ocal_{X} \simeq \Ocal_\DBbb - \det  R^1 f_\ast \Ocal_{X}+ \det  R^2 f_\ast \Ocal_{X} + f_\ast \omega_{X/\DBbb}.
\end{equation}
This is an isometry for the $L^2$ norms. The squared $L^2$ norm of 1 in the first factor is the volume of the fixed K\"ahler form on the fiber, and hence constant since it is a rational number.  For the next two terms, by Theorem \ref{thm:expansion-hodge}, for a local trivialization $\sigma_{p,q}$ of $\det R^q f_\ast \Omega_{X/\DBbb}^p(\log)$ we have  $\log \|\sigma_{p,q} \|^2_{L^2} = \alpha^{p,q} \log|t|^2 + o(\log|t|).$  Also, for a local trivialization $\eta$ of $f_\ast \omega_{X/\DBbb}$, by \cite[Theorem A]{cdg} we have
   $\log \|\eta \|_{L^{2}}^2 = -\alpha \log|t|^2 + o(\log|t|)$
so that a local trivialization $\sigma$ of $\lambda(\Ocal_X)$ satisfies
\begin{displaymath}
    \log\|\sigma\|^2_{L^2} = \left(- \alpha^{0,1}+\alpha^{0,2}-\alpha\right) \log|t|^2 + o(\log |t|).
\end{displaymath}

For the $L^2$ metric on $\lambda(\Omega_{X/\DBbb}(\log)))$,
expanding the determinant and performing the same analysis as in \eqref{eq:expdeterminant}, we find that the norm of a local trivialization $\sigma$ of $\lambda(\Omega_{X/\DBbb}(\log))$ is
\begin{displaymath}
    \log\|\sigma\|^2_{L^2} = \left(\alpha^{1,0} - \alpha^{1,1}+\alpha^{1,2} - \alpha^{1,3} \right) \log |t|^2+o(\log|t|).
\end{displaymath}
All in all, combining \eqref{eq:asym-Q-bcov-dim3}, \eqref{eq:bcov-kahler-simplification}
and the subsequent computations we find a general expression for $\kappa_f$. The assumption that a general fiber is strict Calabi--Yau assures the vanishing of several terms ($\alpha^{0,1}=\alpha^{0,2}=\alpha^{1,0}=\alpha^{1,3}=0$), giving the first part of the theorem. The case of unipotent monodromies is a simplification of the main result under this assumption.
\end{proof}

For the purposes of the below statements, we use the notation $D_{ij}, D_{ijk}$, etc. for $ D_i \cap D_j $, $D_i \cap D_j \cap D_k$ etc. We also abuse notation and write \emph{e.g.} $[D_i^2 D_j ]$ for the class $c_1(\Ocal(D_i)) \cap [D_{ij}]$ and identify top degree intersection products with their degrees. We call quadruple points those points lying on four components on a normal crossings union of three dimensional varieties. With this understood, we can then record the following lemma.

\begin{lemma}\label{lemma:combinatoricid}
Let $f\colon X \to \DBbb$ be a semi-stable projective degeneration, with strict Calabi--Yau $3$-folds smooth fibers, and write $X_0 = \sum D_i$. Denote by $Q$ the number of quadruple points on $X_0$. Then we have
\begin{displaymath}
    \sum_{k \neq i,k \neq j, i < j} c_1(\Ocal(D_k))^2 \cap [D_{ij}] = -4 Q.
\end{displaymath}
If $f$ is furthermore assumed to be Kulikov, then
\begin{displaymath}
\sum_{i < j} c_1(K_{D_{ij}})^2 \cap [D_{ij}] = 8 Q.
\end{displaymath}
\end{lemma}
\begin{proof}
First note the relation $0 = c_1(\Ocal(X_0)) = \sum c_1(\Ocal(D_i)$. Then it follows that, since $D_{ijk} = c_1(\Ocal(D_k)) \cap [D_{ij}]$,
\begin{displaymath}
    \sum c_1(\Ocal(D_i)) [D_{ijk}] = [D_i^2 D_j D_k] + [D_i D_j^2 D_k] + [D_i D_j D_k^2] + \sum_{l \not \in \{i,j,k\} }[D_{ijkl}].
\end{displaymath}
The last sum is the number of quadruple points on $D_{ijk}$. As any quadruple point appears on four different components, taking sums over all possible combinations we find the first identity. If $f$ is Kulikov, then all the components $D_{ij}$ are $\log$-Calabi--Yau so that 
$K_{D_{ij}}=-\sum_{l\not=i, l\not=j}c_1(\Ocal(D_l))$. Then,
\begin{displaymath}
\sum_{i < j} c_1(K_{D_{ij}})^2 \cap [D_{ij}] = \sum_{i < j} \left(\sum_{l\not=i, l\not=j}c_1(\Ocal(D_l))\right)^2 \cap [D_{ij}]=-4Q+2{{4}\choose{2}}Q= 8 Q.
\end{displaymath}
\end{proof}

\begin{corollary}\label{prop:LiuXia:special}
Suppose that $f: X \to  \DBbb$ is a semi-stable, Kulikov germ of a projective degeneration between smooth algebraic varieties, with strict Calabi--Yau $3$-fold generic fibers. Then, if $Q$  denotes the number of quadruple points, we have
\begin{displaymath}
    \begin{split}
         12 \kappa_f &=  \chi(D(2)) - 6Q \\
         &=12 \chi(D(1), \Ocal_{D(1)})- 2Q .
    \end{split}
\end{displaymath}
Notice that in particular $\kappa_f \equiv \frac{-Q}{6} \mod{1}$.
\end{corollary}
\begin{proof}
According to Proposition \ref{cor:liuxiahigherdim} we have
\begin{equation} \label{eq:asimplenormal} 12 \kappa_f = \chi(D(2)) - 3 \chi(D(3)) + 6 \chi(D(4)).
\end{equation}
Since $K_{D_{ijk}} + B_{ijk}$ is trivial (see the proof of the same Proposition \ref{cor:liuxiahigherdim}), we have  
\begin{displaymath} 
    \chi(D_{ijk}) = \# \left \{ \hbox{quadruple points on } D_{ijk} \right\} 
\end{displaymath}
so that $\chi(D(3)) = 4 Q$, which provides the first equality. Using the same argument and the sequence \eqref{eq:resolutionnormalcrossing} we find as $\chi(\Ocal_{X_{0}})=\chi(\Ocal_{X_{\infty}})=0$
\begin{equation*}
     \chi(\Ocal_{{D(1)}})  = \chi(\Ocal_{D(2)})- \chi(\Ocal_{D(3)}) + Q = \chi(\Ocal_{D(2)}) - Q.
\end{equation*}
This expression together with an application of the Noether formula, for the surfaces $D_{ij}$, and Lemma \ref{lemma:combinatoricid} furnishes the second formula in the corollary.
\end{proof}

\begin{remark}
The corollary implies the conjecture of Liu--Xia \cite[Conj. 0.5]{LiuXia}, simplified with the computations as in the previous lemma, 
\begin{displaymath}
    \sum_{i < j} c_1(\Ocal(D_i))c_1(\Ocal(D_j)) \cap [D_{ij}] - \left(\sum_{i < j} c_1(\Ocal(D_i))c_1(\Ocal(D_j))\right) \cap \left(\sum_{i < j} [D_{ij}] \right)  = -2Q.
\end{displaymath}
\end{remark}

In the case that $f: X \to \DBbb$ is a projective degeneration of Calabi--Yau $4$-folds, a similar proof as in Theorem \ref{thm:LiuXia-general} yields the below theorem. As in the $3$-dimensional case, for simplicity, we state it only for strict Calabi--Yau $4$-folds. It is a refinement of the case of normal crossings in Theorem \ref{thm:general-bcov}.

\begin{theorem}\label{thm:LiuXia-general4}
Suppose that $f\colon X \to \DBbb$ is a germ of a projective degeneration between smooth algebraic varieties, whose smooth fibers are strict Calabi--Yau 4-folds and with  special fiber $X_0 = \sum m_i D_i$, not necessarily of normal crossings.  Then
\begin{displaymath}
    \kappa_f = -\frac{1}{12}(\chi(X_\infty)- \chi(X_0)) + \left(-\frac{\chi(X_{\infty})}{12}+4\right)\alpha
     + 2\alpha^{1,1}-2 \alpha^{1,2} + 2 \alpha^{1,3} + 2 \left(2-\sum\chi(\Ocal_{\widetilde{D_i}}) \right)- \frac{1}{12} \int_B c_4(\Omega_{X})
\end{displaymath}
where $\widetilde{D_i}$ denotes a desingularization of $D_i.$ Here  $\alpha = \frac{1}{2\pi i}\left(^{\textit{u}}\log T_{s}\mid \Gr^{4}_{F_{\infty}} H^{4}(X_{\infty})\right)$.
\end{theorem}


The following is a straightforward corollary following the lines of Corollary \ref{cor:rational3}.
\begin{corollary}\label{cor:rational4}
Suppose moreover that $X_0$ has at most rational singularities.
     \begin{enumerate}
     \item Then $\kappa_f = -\frac{1}{12}(\chi(X_\infty)- \chi(X_0))  -2 \alpha^{1,2} + 2 \alpha^{1,3}.$
     \item\label{item:cor-rational4-isolated} If the singularities are moreover isolated, we have
     $$\kappa_f = -\frac{\mu_f}{12}  + 2 \alpha^{1,3}$$
     where $\mu_f$ denotes the Milnor number of the special fiber.
     \end{enumerate}
\end{corollary}
\begin{proof}
We just comment on the vanishing of $\alpha^{1,1}$ and $\alpha^{1,2}$ implicit in the second claim: since the singularities are isolated, the monodromy action is trivial on $H^{2}(X_{\infty})$ and $H^{3}(X_{\infty})$.
\end{proof}

\begin{remark}   For ordinary double point singularities, the corollaries \ref{cor:rational3} and \ref{cor:rational4} are compatible with Theorem \ref{thm:bcov-odp-dim-n}. Indeed, for ordinary double points the Milnor number equals the number of singular points, and we know the monodromy from the Picard-Lefschetz theorem.
\end{remark}

\subsection{Bounds on ordinary double point singularities}\label{subsec:applications-FangLu}
In \cite{FangLu}, Fang--Lu used the differential equation satisfied by the BCOV invariant to prove the non-existence of complete curves in some moduli spaces of polarized Calabi--Yau varieties. In this section, we apply Theorem \ref{thm:bcov-odp-dim-n} to improve on their work. We also remark, in Proposition \ref{prop:ODPHK}, a consequence for abelian varieties and hyperk\"ahler varieties.

Following Fang--Lu, we say that the family $f\colon X\to S$ is primitive if the Hodge forms $\omega_{H^{k}}$ vanish for all $k\neq n$.\footnote{Actually Fang--Lu define the primitive Calabi--Yau manifolds as those strict Calabi--Yau manifolds whose Kuranishi deformations have this vanishing property for the Hodge forms. Also they define primitivity by vanishing for $k < n$. In the integrally polarized case, our notion coincides with theirs as follows from Serre duality.} Also, they establish in \emph{loc. cit.}, Corollary 2.10, the following inequality:
\begin{displaymath}
    \omega_{H^{n}}\geq 2\omega_{WP}.
\end{displaymath}
While the authors work with strict Calabi--Yau manifolds, an examination of their proof shows the validity of this inequality for Calabi--Yau varieties in the broad sense. This gives:
\begin{proposition}\label{prop:fang-lu-gen}
If $f\colon X\to S$ is a primitive family of Calabi--Yau varieties as above, then
\begin{equation*}\label{eq:inequality-WP}
    (-1)^{n}\sum_{i=1}^{r}\kappa_{f}(P_{i})\leq \left(-2+(-1)^{n}\frac{\chi}{12}\right)\vol(\omega_{WP}),
\end{equation*}
where $\vol(\omega_{WP})=\int_{S}\omega_{WP}$.
\end{proposition}
It follows from the proposition that if $(-1)^{n+1}\chi>-24$ and $f$ is non-isotrivial, then $f$ has at least one singular fiber. This observation was already made by Fang--Lu \cite[Cor. 1.3]{FangLu}. We now prove variants of their result.

\begin{corollary}\label{cor:FangLu-odp-1}
In odd relative dimension $n$, assume that $\chi>-24$ and $f\colon X\to S$ as above is a non-isotrivial primitive degeneration of Calabi--Yau manifolds. Suppose furthermore that $f$ has at most ordinary double points. Let $\#\sing(X/S)$ be the total number of singular points in the fibers of $f$. Then
    \begin{displaymath}
        \#\sing(X/S)\geq \frac{48+2\chi}{n+1}.
    \end{displaymath}
\end{corollary}
\begin{proof}
The inequality results as an application of Proposition \ref{prop:fang-lu-gen} and Theorem \ref{thm:bcov-odp-dim-n}. We need to observe that $\omega_{WP}=c_{1}(f_{\ast}K_{X/S},h_{L^{2}})$ and that the $L^{2}$ metric is Mumford good on $f_{\ast}K_{X/S}$ by \cite[Thm. A \& Prop. 2.8]{cdg}. Alternatively, in this case $f_{\ast}K_{X/S}$ is a lower extension and we can apply Theorem \ref{thm:expansion-hodge} and the fact that $\alpha^{n,0}=0$. In any event, this entails
\begin{displaymath}
    \vol(\omega_{WP})=\deg f_{\ast}K_{X/S}\geq 1,
\end{displaymath}
the latter inequality being due to the non-isotriviality assumption \cite[Thm. 5.3.1]{Peters-flatness}.
\end{proof}
\begin{corollary}\label{cor:FangLu-odp-2}
In even relative dimension $n\geq 4$, assume that $\chi<24$ and $f\colon X\to S$ as above is a non-isotrovial primitive degeneration of Calabi--Yau manifolds, admitting at most ordinary double point singularities. Then
\begin{displaymath}
    \#\sing(X/S)\geq \left\lceil\frac{48-2\chi}{n-2}\right\rceil^{\rm{even}},
\end{displaymath}
where for $x\in\RBbb$, $\lceil x\rceil^{\rm{even}}$ denotes the smallest even integer $n$ with $n\geq x$.
\end{corollary}
\begin{proof}
For the inequality $\#\sing(X/S)\geq (48-2\chi)/(n-2)$, the proof goes as in Corollary \ref{cor:FangLu-odp-1}, and is left as an exercise to the reader. To conclude we apply Corollary \ref{cor:funny}.
\end{proof}

For particular geometries we have a stronger non-existence result, which can be proven by other means. We include it as an illustration of our techniques:

\begin{proposition}\label{prop:ODPHK}
If $f :  X \to S$ is a projective degeneration of abelian varieties of dimension at least 2 or hyperk\"ahler varieties of dimension at least 4, then no fibers of $f$ have only ordinary double point singularities.
\end{proposition}
\begin{proof}
This follows from the constancy of the BCOV invariant for such families (cf. Proposition \ref{prop:ken-ichi}) and Theorem \ref{thm:bcov-odp-dim-n}, which expresses the logarithmic term of the BCOV invariant in terms of the number of ordinary double point singularities.
\end{proof}
\bibliographystyle{alpha}

\end{document}